\documentclass[11pt]{article}%
\usepackage[T1]{fontenc}
\usepackage[latin9]{inputenc}
\usepackage{geometry}
\usepackage{amstext}
\usepackage{amssymb,amsmath}
\usepackage{esint}
\usepackage{amsthm}
\usepackage{color}
\usepackage{bm}
\usepackage{enumerate}
\usepackage{graphicx}
\usepackage{appendix}
\usepackage{calrsfs}

\usepackage{float} 
\usepackage{subfigure} 

\usepackage{bbm}
\usepackage{xcolor}
\usepackage{hyperref}
\hypersetup{bookmarksopen=true}
\usepackage[T1]{fontenc}

\usepackage{textcomp}

\usepackage{relsize}



\theoremstyle{plain}
\numberwithin{equation}{section}
\usepackage{indentfirst}
\newtheorem{theorem*}{Theorem}

\newtheorem{definition*}{Definition}
\newtheorem{lemma}{Lemma}
\newtheorem{remark}{Remark}
\newtheorem{assumption}{Assumption}
\newtheorem{proposition*}{Proposition}
\newtheorem{exam}{Example}

\newcommand{\m}[1]{\ensuremath{\mathbf{#1}}}

\newcommand{\ds}{\displaystyle}
\usepackage{hyperref}
\hypersetup{bookmarksopen=true}

\def\rt{\right} \def\lt{\left}
\def\ra{\rangle} \def\la{\langle} 

  \def\to{\rightarrow} \def\Om{\Omega} 
\newcommand{\q}{\quad}

\begin{document}
\title{ Least-Squares Method for Inverse Medium Problems}
\author{Kazufumi Ito\footnote{ Department of Mathematics and Center for Research in Scientific Computation, North Carolina State University, Raleigh, NC 27695. (kito@ncsu.edu).}
\and Ying Liang\footnote{Department of Mathematics, Purdue University, West Lafayette, IN 47907. (liang402@purdue.edu).}
 \and Jun Zou\footnote{Department of Mathematics, The Chinese University of Hong Kong, Shatin, N.T., Hong Kong. The work of this author was substantially supported by Hong Kong RGC General Research Fund (projects 14306719 and 14306718). (zou@math.cuhk.edu.hk).}}
\date{}
\maketitle
\begin{abstract}
We present a two-stage least-squares method to inverse medium problems of reconstructing multiple unknown coefficients simultaneously from noisy data. A direct sampling method is applied to detect the location of the inhomogeneity in the first stage, while a total least-squares method with mixed regularization is used to recover the medium profile in the second stage. The total least-squares method is designed to minimize the residual of the model equation and the data fitting, along with an appropriate regularization, in an attempt to significantly improve the accuracy of the approximation obtained from the first stage. 
We shall also present an analysis on the well-posedness and convergence of this algorithm.
Numerical experiments are carried out to verify the accuracies and robustness of 
this novel two-stage least-squares algorithm, with great tolerance of noise.
\end{abstract}


\noindent {\footnotesize {\bf Mathematics Subject Classification (MSC2000)}: { 35R30, 65F22, 65N21, 65N06}}

\begin{section}{Introduction}\label{intro}

In this work, we propose a total least-squares formulation for recovering multiple medium coefficients for a class of inverse medium problems that are governed by the forward model of the form: 
\begin{equation}
L(u,q)=g\label{gov}
\end{equation}
where $L$ is a bilinear operator on $(u, q)$,
%Q: It is hard for the usual second-order elliptic operator from $X$ to itself. 
%A: The domain of the operator $L(\cdot, q)$  is smaller than $X$ and is defined case by case. Taking DOT (second order formulation) as an example, given some constant coefficient $q$, the domain of $L(\cdot,q)$  would be $H^2(\Omega)$, which is a subspace of $X: = L^2(\Omega)$. For a different $q$, the natural domain of $L(\cdot,q)$ would be different.  We would find the minimizer over space $X$ which is larger than domain of $L$.  I do find such expression of this general setting strange and we have discussed via email about this issue after I talked with Prof. Ito about it. A possible way to avoid so many notations in the very beginning would be keeping '$L$ is a bilinear differential operator on $(u,q)$' here and explain the domain latter when needed, or we introduce one more notation $D$ to denote the domain of $L(\cdot, q)$. A more clear expression using $dom(L)$ can be found later in Assumption 2.
% Moreover, $L$ is actually not a simple differential operator, as it requires the  state variable $u$ to satisfy certain boundary condition (such as Neumann boundary condition for DOT problem) to ensure the invertibility of $L_q$.
 $u\in X$ is the state variable, while $q\in Z$ represents one or multiple unknown coefficients in the model 
that are to be recovered, under some measurement data of $f:=Cu\in Y$. Here $X$, $Y$ and $Z$ are three Hilbert spaces, 
and $C$ is an observation map from $X$ to $Y$. 

%Our presentation and contribution are as follows.
%\begin{itemize}
%\item Inverse Medium problem as a bilinear constraint.
%\item Least squares formulation with regularization.
%\item Numerical algorithm using Alternative Direction Iterate (ADI) method and Convergence
%\item First order formulation of constraint  for diffuse  optical tomography (DOT).
%\item two-stage method and Direct sampling method.
%\item Numerical findings.
%\end{itemize} 

In many applications, it is often required to recover multiple coefficients simultaneously. For instance, 
the diffusive optical tomography (DOT) aims at recovering the diffusion and absorption coefficients $\sigma$ and $\mu$ 
from the governing equation \cite{arridge1999optical, gibson2005recent}:
\begin{equation} \label{dot}
-\nabla\cdot (\sigma(x) \nabla u)+\mu(x)\,u=0\, \q \mbox{in} ~~\Om
\end{equation}
using the Cauchy data $(f, h)$ collected at the boundary $\Gamma$ of $\Omega$:
\begin{equation} \label{bdr}
u|_\Gamma=f,\quad\frac{\partial u}{\partial \nu}|_\Gamma=h\,.
\end{equation}
%where $\frac{\partial}{\partial \nu}$ denotes the outward normal derivative on $\Gamma$. 
Another example would be the inverse electromagnetic medium problem 
to recover the unknown magnetic and electric coefficients $\mu$ and $\lambda$ in the Maxwell's system \cite{dorn1999nonlinear, bao2005inverse}:
$$\begin{array}{l}
\nabla \times \vec{H}+i\omega \mu(x)\,\vec{E}=0\, \q \mbox{in} ~~\Om\,,
\\ 
\nabla \times \vec{E}-i\omega \lambda(x)\,\vec{H}=0 \, \q \mbox{in} ~~\Om\,, 
\end{array}$$
using some electric or magnetic measurement data $\vec{E}$ or $\vec{H}$.

The inverse reconstruction of multiple medium coefficients is generally much more technical and difficult than the single coefficient case. We shall proposed a total least-squares formulation with appropriate regularization 
to transform the inverse problem into an optimization problem.
The total least-squares philosophy is not uncommon. 
One conventional approach for inverse medium problems is to seek an optimal parameter $q$ from a feasible 
set $K\subset Z$ such that it minimizes an output least-squares functional $j$ of the form 
\begin{equation} \label{standard}
j(q) = \Vert Cu(q)-f\Vert_{Y}^2+\alpha\,\psi(q),
\end{equation}
where $u(q)$ solves the forward model \eqref{gov} when $q$ is given, 
$\psi$ is a regularization term and $\alpha>0$ is a regularization parameter. 
We refer the readers to \cite{gibson2005recent,cheney1999electrical,ito2015inverse} for more details about this traditional approach.
A relaxed variational method of the least-squares form was proposed and studied 
in \cite{kohn1987relaxation, wexler1985impedance} for the impedance computed tomography.
The least-squares functional consists of a residual term associated with the governing equation while the measurement data is enforced at the feasible set. Different from the aforementioned approaches, we shall follow some basic principle of a total least-squares approach from \cite{chung2021least} 
and treat the governing equation and the data fitting separately, along with a regularization. That is, 
we look for an optimal parameter $q$ from $Z$ and a state variable $u$ from $X$ together such that they minimize an 
extended functional of the form 
\begin{equation} \label{inv0}
J(u,q)=\Vert L(u,q)- g\Vert_{X}^2+\Vert Cu-f\Vert_{Y}^2+\alpha\,\psi(q)\,.  
\end{equation}

This functional combines the residual of model equation,
data fitting and constraints on parameters in the least-squares criterion. The combination in \eqref{inv0} results in a regularization effect to treat the model equation and allows a robustness as a quasi-reversibility method \cite{lattes1969method}.
Compared with the conventional approaches, the domain of $J(u,q)$ is much more regular and the semi-norm defined by the formulation is much stronger. More precisely, the total least-squares formulation aims to find a solution pair $(u,q)$ simultaneously in a smooth class and is less sensitive to the noise and uncertainty in the inverse model.
%All points made in a previous paper are extended.
Another important feature of this formulation is that the functional $J(u,q)$ is quadratical and convex with respect to each variable 
$u$ and $q$, if the regularization $\psi$ is chosen to be quadratical and convex, while the traditional one $j(q)$ in \eqref{standard} is highly nonlinear and non-convex in general. 
This special feature facilitates us naturally to minimize the functional $J(u,q)$ effectively by the alternating direction iterative (ADI) 
method \cite{CsiszarAM1984, byrne2014iterative} so that only two quadratical and convex suboptimizations are required in terms of the variables $u$ and $q$ respectively at each iteration.

%For well-posedness of the formulation, we also assume  regularization term $\psi$ is strongly convex, and we will further explain the properties of convex functions in the Appendix \ref{bregman}.

In addition to the functional \eqref{inv0} that uses the residual of the forward model \eqref{gov}, we will also 
address another least-squares functional that makes use of the equivalent first-order system of the forward model \eqref{gov} and replaces the first term in \eqref{inv0} by the residuals of the corresponding first-order system. 
Using the first-order system has been the fundamental idea in modern least-squares methods in solving second-order PDEs \cite{jespersen1977least, bochev1998finite, pehlivanov1994least}.
The advantages of using first-order formulations are much more significant to the numerical solutions of inverse problems, 
especially when we aim at simultaneously reconstructing multiple coefficients as we do in this work. 
First, the multiple coefficients appear in separated first-order equations, hence are naturally decoupled. This would 
greatly reduce the nonlinearity and enhance the convexity of the resulting optimization systems. 
Second, the first-order formulation relaxes the regularity requirement of the solutions in the resulting analysis.

A crucial step to an effective reconstruction of multiple coefficients is to seek some reasonable initial approximations 
to capture some general (possibly rather inaccurate) geometric and physical profiles of all the unknown multiple coefficients. 
This is a rather technical and difficult task in numerical solutions. 
For this purpose, we shall propose to adopt the direct sampling-type method (DSM) that we have been developing 
in recent years (cf. \cite{ito2012two, chow2014direct, chow2015direct, chow2018time}). 
Using the index functions provided by DSM, we shall determine a computational domain that is often much smaller 
than the original physical domain, then the restricted index functions on the computational domain serve 
as the initial guesses of the unknown coefficients. 
In this work, we will apply a newly developed DSM \cite{chow2020direct}, where two groups of probing and index functions are constructed to identify and decouple the multiple inhomogeneous inclusions of different physical nature, 
which is different from the classical DSMs targeting the inhomogeneous inclusions of one single physical nature. 
As we shall see, DSMs turn out to be very effective and fast solvers to provide some reasonable initial approximations.

The rest of the paper is structured as follows. In Section~\ref{sec:well}, we justify the well-posedness of the least-squares formulation for the general inverse medium problems. In Section~\ref{sec:ADI}, we propose an alternating direction iterative method for solving the minimization problem and prove the convergence of the ADI method. 
We illustrate in Section~\ref{sec:DOT} how this total least-squares method applies to a concrete inverse problem, 
by taking the DOT problem as a benchmark problem. We present numerical results in Section~\ref{sec:Num} for a couple of 
different types of inhomogeneous coefficients for the DOT problem to demonstrate the stability and effectiveness 
of this proposed method. Throughout the paper, $c$, $c_0$ and $c_1$ denote generic constants which may differ at each occurrence.
\end{section}

\begin{section}{Well-posedness of the least-squares formulation for inverse medium problem}\label{sec:well}
Recall that to solve the inverse medium problems modeled by \eqref{gov}, we propose the following least-squares formulation
\begin{equation}\label{minilsq}
\min_{u\in X, q\in Z} J(u,q) =\Vert L(u,q)- g\Vert_{X}^2+\Vert Cu-f\Vert_{Y}^2+\alpha\,\psi(q)\, .
\end{equation}
This section is devoted to the well-posedness of the total least-squares formulation \eqref{minilsq}, namely, the existence of a solution to \eqref{minilsq} and the conditional stability of the reconstruction with respect to the measurement. 
 To provide a rigorous justification of the well-posedneness, we present several assumptions on the least-squares formulation, which are minimal for the proof.
 We will verify these checkable assumptions in Section~\ref{sec:DOT} for a concrete example of such inverse medium problems.

 Let us first introduce several notations.
For simplicity, for a given $q\in Z$ (resp.\,$u\in X$), we will write $L_q$ (resp.\,$\Phi_u$) as 
 \begin{equation}
L_q u:=L(u,q) ~~(\mbox{resp.}\,\Phi_u q:=L(u,q))\,. \label{notLq} 
\end{equation}
We denote the subdifferential of the regularization term $\psi$ at $q$ by $\partial \psi(q)$, and denote the inner products of the Hilbert spaces $X$, $Y$ and $Z$ by
$(\cdot,\cdot)_X$, $(\cdot,\cdot)_Y$ and $\la \cdot, \cdot\ra$ respectively. 

\subsection{Existence of a minimizer}
%To prove the existence of a minimizer to the least-squares formulation \eqref{minilsq}, 
We present the following assumptions on the regularization term $\psi$ and operators $L$ and $C$ in the forward model:

\begin{assumption}\label{assum1}
The regularization term $\psi$ is strictly convex and weakly lower semicontinuous.
Furthermore, $\psi$ is also coercive \cite{ito2015inverse}, i.e., $\psi(q)\geq c\Vert q\Vert^2_Z$.
% and the level set $\{q\in Z:\psi(q)\leq c_0\}$ defines a compact set in $L^\infty(\Omega)$.
\end{assumption}
This assumption implies that the level set $\{q\in Z:\psi(q)\leq c_0\}$ defines a bounded set in $Z$.

\begin{assumption}\label{assum2}
Given a constant $c_0$, for $q$ in the level set $\{q\in Z :\psi(q)\le c_0\}$, $L_q: dom(L)\to X$, where $dom(L)\subset\{u\in X:L(u,q)\in X \text{ and }Cu\in Y \text{ for all }q\in Z\text{ satisfying } \psi(q)\leq c_0\}$ with a specific side constraint that is not in the least-squares formulation \eqref{minilsq}, is a closed linear operator and is uniformly coercive, i.e., the graph norm $|u|_{W,q}:=\Vert L(u,q)\Vert_X$ satisfies $|u|_{W,q}\geq c_1\Vert u\Vert_X$ uniformly in $q$ for some constant $c_1>0$, and thus $L(u, q)=g\in X$ has a unique solution in $dom(L)$.
\end{assumption}

Under Assumption \ref{assum2}, we can define the inverse operator $L_q^{-1}:X\to dom(L)$, which is uniformly bounded by the coercivity of $L_q$. We also need the following assumption on the sequentially closedness of operators $L$ and $C$.
\begin{assumption}\label{assumLC}
 The operators $L$ and $C$ are weakly sequentially closed, i.e., if a sequence $\{(u_n,q_n)\}_{n=1}^\infty$ converges to $(u, q)$ weakly in $X\times Z$, then the sequence $\{L(u_n,q_n)\}_{n=1}^\infty$ converges to $L(u, q)$ weakly in $X$ and the sequence 
 $\{Cu_n\}_{n=1}^\infty$ converges to $C u$ weakly in $Y$.
\end{assumption}
Then we can verify the existence of the minimizers to the least-squares formulation \eqref{minilsq}.
\begin{theorem*}\label{exist}
Under Assumptions \ref{assum1}--\ref{assumLC}, there exists a minimizer $(u^\star, q^\star)$ in $X\times Z$ of the least-squares formulation \eqref{minilsq}.
\end{theorem*}
\begin{proof}
Since $X$ and $Z$ are nonempty, 
%there exists at least one $(\tilde{u},\tilde{q})\in X\times Z$ such that $J(\tilde{u},\tilde{q})$ is finite. Then 
there exists a minimizing sequence $\{(u_n,q_n)\}_{n=1}^\infty$ in $X\times Z$ such that 
\begin{equation}
\lim_{n\to\infty} J(u_n,q_n) =\inf_{(u,q)\in X\times Z}J(u,q).\label{liminf}
\end{equation}
By Assumptions \ref{assum1}--\ref{assum2}, $\psi$ is a coercive functional and the graph norm $|\cdot|_{W,q}$ 
is uniformly coercive, thus it follows \eqref{inv0} that the sequence $\{(u_n,q_n)\}_{n=1}^\infty$ is uniformly bounded. Then there exists a subsequence of $\{(u_n,q_n)\}_{n=1}^\infty$, still denoted by $\{(u_n,q_n)\}_{n=1}^\infty$, and some $(u^\star, q^\star)\in X\times Z$ such that $u_n$ converges to $u^\star$ weakly in $X$ and $q_n$ converges to $q^\star$ weakly in $Z$. 
 As $L$ and $C$ are weakly sequentially closed by Assumption \ref{assumLC}, there hold
\begin{eqnarray}
L(u_n,q_n)\text{ converges to } L(u^\star, q^\star)\text{ weakly in }X; \quad Cu_n \text{ converges to } Cu^\star \text{ weakly in }Y.\label{weakLCconv}
\end{eqnarray}
From the weak lower semicontinuity of the norms $\Vert\cdot\Vert_X$ and $\Vert\cdot\Vert_Y$, we have 
$$\Vert L(u^\star,q^\star)-g\Vert^2_X+ \Vert Cu^\star -f\Vert^2_Y\leq \lim_{n\to\infty}\inf\lt(\Vert L(u_n, q_n)-g\Vert^2_X+ \Vert Cu_n-f\Vert^2_Y\rt).$$
 Together with the lower semicontinuity of the regularization term $\psi$, we can deduce that
 \begin{equation*}
 J(u^\star, q^\star)\leq \lim_{n\to\infty}\inf J(u_n, q_n).
 \end{equation*}
Hence it follows \eqref{liminf} that $(u^\star,q^\star)$ is indeed a minimizer of the functional $J$ in $X\times Z$.
\end{proof}

\begin{remark}\label{rmkweakclose}
While the weakly sequentially closedness in Assumption \ref{assumLC} is
nontrivial to verify for most nonlinear inverse medium problems, one can reach the weak convergence \eqref{weakLCconv} with the compactness of $dom(L)$ and the concrete formula of $L$, as shown in Section \ref{sec:DOT}.
\end{remark}

\subsection{Conditional stability}

In this subsection, we present some conditional stability estimates of the total least-squares formulation \eqref{minilsq} for the general inverse medium problems. First we introduce two pairs $(\bar{u},\bar{q})$ and $(\bar{g},\bar{f})$ that satisfy
\begin{equation}\label{0}
L(\bar{u},\bar{ q})=\bar{g},\;\; C\bar{u}=\bar{f}.
\end{equation}
Letting $(u^\star, q^\star)$ be the unique minimizer of \eqref{minilsq} in a neighborhood of $(\bar{u},\bar{q})$, we study the approximation error $(\delta u,\delta q): = (u^\star-\bar{u}, q^\star-\bar{q})$ to illustrate the stability of the least-squares formulation \eqref{minilsq} with respect to the measurement $f$ and also the term $g$ in the governing equation \eqref{gov}. Denote the residual of the governing equation by $\epsilon_1 :=L(u^\star, q^\star)-g$.
As $(u^\star,q^\star)$ is the local minimizer of functional $J$ in \eqref{inv0}, we have $J(u^\star, q^\star)\leq J(\bar{u},\bar{q})$. Therefore, by the definition of $J$ we have the inequality 
\begin{equation}
\Vert \epsilon_1\Vert_X^2+\Vert C(u^\star-\bar{u})-(f-\bar{f})\Vert_Y^2
+\alpha\psi(q^\star)\le \Vert g-\bar{g}\Vert_X^2+\Vert f-\bar{f}\Vert_Y^2+\alpha\psi(\bar{q}),\label{ineqbar}
\end{equation}
 which directly leads to the following observation on $\psi(q^\star)$:
 %\begin{lemma}\label{lemma2}
%For  minimizer $(u^\star, q^\star)$ of \eqref{inv0} and $\bar{q}$ defined in \eqref{0}, 
\begin{equation}
\alpha \psi(q^\star)\le \Vert g-\bar{g}\Vert_X^2+\Vert f-\bar{f}\Vert_Y^2+\alpha\,\psi(\bar{q}).\label{bdq0}
\end{equation}
%\end{lemma}
If $\psi$ is coercive, \eqref{bdq0} provides a rough estimate of the reconstruction $q^\star$ with respect to the data noise. 

We can further derive an estimate of the approximation error $\delta q$, under the following assumption on the operator $L$.
\begin{assumption}\label{assum3}
%There exists a metric $W$ on $Z$ such that for any $q\in Z$,
There exists a norm $\Vert\cdot\Vert_W$ on $Z$ such that for any $q\in Z$,
$$
\Vert CL_q^{-1}\Phi_{\bar{u}}(q-\bar{q})\Vert^2_Y \ge \Vert q-\bar{q}\Vert_W^2.
$$
\end{assumption}
Assumption \ref{assum3} holds when $\delta q$ belongs to a finite rank subspace $S$ of $Z$
and $\Vert\delta q\Vert_W\neq 0$ for all non-zero $\delta q\in S$. We can now deduce the following result of the approximation error $\delta q$.
\begin{lemma}\label{lemmaW}
Under Assumptions \ref{assum1}--\ref{assum3}, the approximation error 
 $\delta q$ is bounded in $W$-norm by the data noise and the regularization term, i.e., 
 \begin{equation}
\Vert \delta q\Vert_W^2+\alpha\psi(q^\star)\le c_0\left(\Vert g-\bar{g}\Vert_X^2+\Vert f-\bar{f}\Vert_Y^2+\alpha\,\psi(\bar{q})\right).\label{111}
\end{equation}
%where $c_0$ is a constant.
\end{lemma}
\begin{proof}
Using the bilinear property of $L$, one can rewrite the difference $L(u^\star, q^\star)-L(\bar{u},\bar{q})$ as
\begin{equation}
L(u^\star, q^\star)-L(\bar{u},\bar{q})=L_{q^\star}(\delta u)+\Phi_{\bar u}(\delta q).\label{sep}
\end{equation}
By Assumption \ref{assum2}, $L_q$ admits an inverse operator $L_q^{-1}$ from $X$ to $dom(L)$, which, together with \eqref{0}, \eqref{sep} and the definition of $\epsilon_1$, implies
\begin{equation}
 \delta u=L_{q^\star}^{-1}(\epsilon_1+g-\bar{g}-\Phi_{\bar{u}}( \delta q)).\label{del}
\end{equation}
Plugging \eqref{del} into \eqref{ineqbar} leads to an inequality:
\begin{eqnarray}
\Vert CL_{q^\star}^{-1}\Phi_{\bar{u}}\delta q-(\epsilon_1+L_{q^\star}^{-1}(g-\bar{g}))+f-\bar{f}\Vert_Y^2
+\Vert \epsilon_1\Vert_X^2+\alpha\,\psi(q^\star)
\le \Vert g-\bar{g}\Vert_X^2+\Vert f-\bar{f}\Vert_Y^2+\alpha\,\psi(\bar{q}).\label{bdq1}
\end{eqnarray}
It follows Assumption \ref{assum3} that there exists a norm $\Vert\cdot\Vert_W$ such that
\begin{equation}\label{3W}
\Vert CL_q^{-1}\Phi_{\bar{u}}\delta q\Vert_Y^2 \ge \Vert \delta q\Vert^2_W.
\end{equation}
Then we can deduce from \eqref{bdq1}, the triangle inequality, the boundedness of $L_q^{-1}$ and \eqref{3W} that 
\begin{eqnarray*}
\Vert \delta q\Vert_W^2+\alpha\,\psi(q^\star)\le c_0 (\Vert g-\bar{g}\Vert_X^2+\Vert f-\bar{f}\Vert_Y^2+\alpha\,\psi(\bar{q})),
\end{eqnarray*}
where $c_0$ is a constant, which completes the proof.
\end{proof}

The rest of this section is devoted to verifying the consistency of the least-squares formulation \eqref{minilsq} as the noise level of measurement goes to zero, which is an essential property of a regularization scheme. If we choose an appropriate regularization parameter $\alpha$ according to the noise level of the data, we can deduce the convergence result of the reconstructed coefficients associated with 
the regularization parameter $\alpha$. 
More precisely, given a set of exact data $(\overline{g}, \overline{f})$, we consider a parametric family $\{(g_\alpha, f_\alpha)\}$ such that $ \| g_\alpha - \overline{g} \|_X^2 +\| f_\alpha - \overline{f} \|_Y^2 = o(\alpha)$. 
In the rest of this section, we denote the functional $J$ in \eqref{inv0} with $g= g_\alpha$ and $f=f_\alpha$ 
by $J_\alpha$, and the minimizer of $J_\alpha$ by $(u_\alpha, q_\alpha)$.
Then we justify the consistency of the least-squares formulation \eqref{minilsq} by proving the convergence of the sequence of minimizers $\{q_\alpha\}$ to the minimum norm solution \cite{ito2015inverse} of the system \eqref{0} as $\alpha\to 0$.
%See page 107 of Tikhonov regularization book by Ito and Jin
% $(u_{\alpha_n}, q_{\alpha_n})$ in the next theorem is not necessary the unique minimizer of $J$. As long as they are global minimizers, the theorem hold true.
\begin{theorem*}
Let $\{\alpha_n\}_{n=1}^\infty\subset \mathbb{R}^+$ be a sequence converging to zero, and 
$\{(u_{\alpha_n}, q_{\alpha_n})\}_{n=1}^\infty$ be the corresponding sequence of minimizers of $J_{\alpha_n}$. 
Then 
under Assumptions \ref{assum1}--\ref{assumLC}, 
$\{(u_{\alpha_n}, q_{\alpha_n})\}_{n=1}^\infty$ has a subsequence that converges weakly to a minimum norm solution 
$(\hat{u},\hat{q})$ of the system \eqref{0}, i.e., 
\begin{eqnarray*}
&&L(\hat{u},\hat{q})=\bar{g},\quad C\hat{u}=\bar{f}.\\
&&\psi(\hat{q})\le \psi(\bar{q}) \mbox{ for all $(\bar{u},\bar{q})$ satisfying \eqref{0}}.
\end{eqnarray*}
% \textcolor{red}{If $\psi$ is a norm of $X$, the convergence is strong.}
\end{theorem*}
\begin{proof}
As $(u_{\alpha_n},q_{\alpha_n})$ is the minimizer of $J_{\alpha_n}$, there holds
\begin{equation}\label{ineqJ}
J_{\alpha_n}(u_{\alpha_n},q_{\alpha_n})\leq J_{\alpha_n}(\bar{u},\bar{q}).
\end{equation}
By definition, $J_{\alpha_n}(\bar{u},\bar{q})= \Vert L(\bar{u},\bar{q})- g_{\alpha_n}\Vert_{X}^2+\Vert C\bar{u}-f_{\alpha_n}\Vert_{Y}^2+\alpha_n\,\psi(\bar{q})$, and thus $\{J_{\alpha_n}(u_{\alpha_n},q_{\alpha_n})\}_{n=1}^\infty$ is uniformly bounded. 
 Following the similar argument in the proof of Theorem \ref{exist}, there exist a subsequence of $\{(u_{\alpha_n}, q_{\alpha_n})\}_{n=1}^\infty$, still denoted as $\{(u_{\alpha_n}, q_{\alpha_n})\}_{n=1}^\infty$, and some $(\hat{u},\hat{q})$ such that 
 $\{(u_{\alpha_n},q_{\alpha_n})\}_{n=1}^\infty$ converges to $ (\hat{u}, \hat{q})$ weakly in $X\times Z$. 
 By Assumption \ref{assumLC}, we have
\begin{eqnarray*}
L(u_{\alpha_n},q_{\alpha_n})\text{ converges to } L(\hat{u}, \hat{q})\text{ weakly in }X; \quad Cu_{\alpha_n} \text{ converges to } C\hat{u}\text{ weakly in }Y.
\end{eqnarray*}
 From \eqref{ineqJ} one can also derive that
 \begin{equation}\label{LCalpha}
  \Vert L(u_{\alpha_n},q_{\alpha_n})- g_{\alpha_n}\Vert_{X}^2+\Vert Cu_{\alpha_n}-f_{\alpha_n}\Vert_{Y}^2+\alpha_n\,\psi(q_{\alpha_n})\leq \Vert \bar{g}- g_{\alpha_n}\Vert_{X}^2+\Vert \bar{f}-f_{\alpha_n}\Vert_{Y}^2+\alpha_n\,\psi(\bar{q}),
  \end{equation}
  which implies 
  \begin{eqnarray*}
L(u_{\alpha_n},q_{\alpha_n})\text{ converges to } \bar{g} \text{ strongly in }X; \quad Cu_{\alpha_n} \text{ converges to } \bar{f} \text{ strongly in }Y.
\end{eqnarray*}
  Therefore, as $\alpha_n\to 0$,  $\{(u_{\alpha_n},q_{\alpha_n})\}_{n=1}^\infty$ will converge to $(\hat{u},\hat{q})$ satisfying
\begin{equation}\label{hatuq}
L(\hat{u},\hat{q})=\bar{g},\quad C\hat{u}=\bar{f}.
\end{equation}
Recall that one has an estimate of $\psi(q_{\alpha_n})$ from \eqref{LCalpha} that \begin{equation}
\alpha_n\psi(q_{\alpha_n})\le \Vert g_{\alpha_n}-\bar{g}\Vert_X^2+\Vert f_{\alpha_n}-\bar{f}\Vert_Y^2+\alpha_n\,\psi(\bar{q}),\nonumber
\end{equation}
which leads to 
\begin{equation}
\psi(q_{\alpha_n})\le\psi(\bar{q})+o(1),\nonumber
\end{equation}
as $\alpha_n\to 0$. Using the lower semicontinuity of functional $\psi$, one obtain that 
$$
\psi(\hat{q})\le \psi(\bar{q}) \mbox{ for all $(\bar{u},\bar{q})$ satisfying \eqref{0}}.
$$
Together with \eqref{hatuq}, we conclude that  $(\hat{u},\hat{q})$ is a minimum norm solution of \eqref{0}.
\end{proof}
%Minimum norm argument.page 35  an page 111 of inverse problem by Bangti

\end{section}
\begin{section}{ADI method and convergence analysis}\label{sec:ADI}
An important feature of the least-squares formulation \eqref{minilsq} is that the functional $J(u,q)$ is quadratical 
and convex with respect to each variable $u$ and $q$. 
This unique feature facilitates us naturally to minimize the functional $J$ effectively by the alternating direction iterative (ADI) 
method \cite{CsiszarAM1984, byrne2014iterative} so that only two quadratical and convex suboptimizations of one variable are required at each iteration. 
We shall carry out the convergence analysis of the ADI method in this section for general inverse medium problems.

\medskip
{\bf Alternating direction iterative method} for the minimization of \eqref{inv0}.

Given an initial pair $(u_0, q_0)$, find a sequence of pairs 
$(u_k, q_k)$ for $k\ge 1$ as below: 
\begin{itemize}
\item Given $q_k$, find $u=u_{k+1}\in X$ by solving 
\begin{equation}  \label{A-1}
\min_{u\in X} \Vert  L(u,q_k)- g\Vert_{X}^2+\Vert Cu-f\Vert_{Y}^2\,.
\end{equation}
\item  Given $u_{k+1}$, find $q=q_{k+1}\in Z$ by solving 
\begin{equation}  \label{A-2}
\min_{q\in Z} \Vert L(u_{k+1},q)- g\Vert_{X}^2+\alpha\psi(q)\,.
\end{equation}
\end{itemize}
%Note that $L$ is a bilinear operator. 

We shall establish the convergence of the sequence $\{(u_k,q_k)\}_{k=1}^\infty$ generated by the ADI method, 
under Assumptions~\ref{assum1}--\ref{assumLC} on the least-squares formulation \eqref{minilsq}.  For this purpose, we would like to introduce the Bregman distance \cite{bregman1967relaxation} with respect to  $\xi\in\partial \psi(p)$,
\begin{equation}\label{breg}
E(q, p):=\psi(q)-\psi(p)-\la\xi,q- p\ra ~~\forall\, q, p\in Z\,,
\end{equation}
which is always nonnegative for convex function $\psi$. % see Appendix \ref{bregman} for details. 
%$X$ has already the norm $\|\cdot\|_X$, why do we need to introduce another norm? See next note.
%Cite a reference for this result. I shall reply these two comments together. This inequality \eqref{equiv} is actually an explanation of 'uniformly bounded inverse' or a more explicit expression of Assumption 2, that is, there exists 'a norm' (not any norm) of $dom(L)$ such that this inequality \eqref{equiv} holds. Still taking DOT as a concrete example, the norm $|\cdot|_{W_0}$ denotes 'a norm' of $dom(L)\subset X: = L^2(\Omega)$, which could be the standard $H^2$ norm of $u$, and this assumption \eqref{equiv} holds as proved in subsection 4.2 for DOT. I have rephrased these lines marked as green. 
%Let $|\cdot|_{W_0}$ be the norm of $dom(L)\subset X$, then we know
%that the graph norms corresponding to $q$ satisfying $\psi(q)\leq c_0$  are equivalent to the norm $|\cdot|_{W_0}$, 
%i.e., 
%\begin{equation}
%c_2 |u|_{W_0} \le \sqrt{\Vert L(u,q)\Vert^2_X+\Vert Cu\Vert_Y^2} \le c_1|u|_{W_0}\label{equiv}
%\end{equation}
%for some constants $c_1$ and $c_2$ independent of $q$.
Now we are ready to present the following lemma on convergence of the sequence $\{(u_k,q_k)\}_{k=1}^\infty$ generated by \eqref{A-1}--\eqref{A-2}. 
\begin{lemma}
Under Assumptions \ref{assum1}--\ref{assumLC}, the sequence $\{(u_k,q_k)\}_{k=1}^\infty$ generated by the ADI method \eqref{A-1}--\eqref{A-2} converges to a pair $(u^\star, q^\star)$ that satisfies the optimality condition of \eqref{minilsq}:
\begin{equation} \label{nec}
\begin{aligned}
(L_q)^*(L_q u-g)+C^*(Cu-f)=0,
\\ 
-2(\Phi_u)^*(\Phi_u q-g) \in \partial\big(\alpha\psi(q)\big)\,.
\end{aligned} 
\end{equation}
\end{lemma}
\begin{proof}
Using the optimality condition satisfied by the minimizer $u_{k+1}$ of \eqref{A-1}, we deduce
\begin{align}
0&=\big(L_{q_k}^*(L_{q_k}u_{k+1}-g),u_{k+1}-u_{k}\big)_X+\big(C^*(Cu_{k+1}-f),u_{k+1}-u_{k}\big)_Y
\nonumber \\ 
 \ds\q &=\frac{1}{2}\Big(\Vert L_{q_k}u_{k+1}-g\Vert_X^2+\Vert Cu_{k+1}-f\Vert_Y^2
-\lt(\Vert L_{q_k}u_{k}-g\Vert_X^2+\Vert Cu_{k}-f\Vert_Y^2\rt) \nonumber\\ 
\ds\q\q &\q+\Vert L_{q_k}(u_{k+1}-u_{k})\Vert_X^2+\Vert C(u_{k+1}-u_{k})\Vert_Y^2\Big)\,.\label{uq1}
\end{align} 
Similarly, from the optimality condition satisfied by the minimizer $q_{k+1}$ of \eqref{A-2}, we obtain 
\begin{equation*}
-2\Phi_{u_{k+1}}^*(\Phi_{u_{k+1}} q_{k+1}-g) \in \partial\big(\alpha\psi(q_{k+1})\big).
\end{equation*}
 Taking $q=q_k$, $p=q_{k+1}$ and 
$\xi = -2\Phi_{u_{k+1}}^*(\Phi_{u_{k+1}} q_{k+1}-g)$ in \eqref{breg}, we can derive that 
\begin{equation}
\la-2\Phi_{u_{k+1}}^*(\Phi_{u_{k+1}} q_{k+1}-g), q_{k}-q_{k+1}\ra +\alpha\psi(q_{k+1})-\alpha\psi(q_{k})+E(q_{k},q_{k+1}) = 0. \label{eqqk}
\end{equation}
%Give more details how this is derived. Done.
 The following equality hold for the first term in \eqref{eqqk}:
\begin{align}
&\, \, \quad\la\Phi_{u_{k+1}}^*(\Phi_{u_{k+1}} q_{k+1}-g),q_{k+1}-q_{k}\ra \nonumber\\
&=\la\Phi_{u_{k+1}} q_{k+1}-g,\Phi_{u_{k+1}}(q_{k+1}-q_{k})\ra \nonumber\\
&=\frac{1}{2}\la(\Phi_{u_{k+1}} q_{k+1}-g)+(\Phi_{u_{k+1}} q_{k}-g),\Phi_{u_{k+1}}q_{k+1}-g-(\Phi_{u_{k+1}}q_{k}-g)\ra \nonumber\\
&\quad +\frac{1}{2}\la(\Phi_{u_{k+1}} q_{k+1}-g)-(\Phi_{u_{k+1}} q_{k}-g),\Phi_{u_{k+1}}q_{k+1}-g-(\Phi_{u_{k+1}}q_{k}-g)\ra \nonumber\\
&=\frac{1}{2}\left(\Vert \Phi_{u_{k+1}} q_{k+1}-g\Vert_X^2-\Vert \Phi_{u_{k+1}} q_{k}-g\Vert_X^2+\Vert \Phi_{u_{k+1}}(q_{k+1}-q_{k})\Vert_X^2\right).\label{qqk}
\end{align}

Plugging \eqref{qqk} into \eqref{eqqk}, it follows that  
\begin{equation}
\Vert \Phi_{u_{k+1}} q_{k+1}-g\Vert_X^2+\alpha\psi(q_{k+1})-\lt(\Vert \Phi_{u_{k+1}} q_{k}-g\Vert_X^2+\alpha\psi\lt( q_{k}\rt)\rt)+\Vert \Phi_{u_{k+1}}( q_{k+1}- q_{k})\Vert_X^2
+E( q_{k}, q_{k+1})=0.\label{uq2}
\end{equation}
As the sequence $\{(u_k,q_k)\}_{k=1}^\infty$ is generated by ADI method \eqref{A-1}--\eqref{A-2}, in each iteration the updated $u_{k+1}$(resp. $q_{k+1}$) minimizes the functional $J(u, q_{k})$ (resp. $J(u_{k+1}, q)$), which would lead to
\begin{equation*}
J(u_k,  q_k) \ge J(u_{k+1}, q_k)\ge  J(u_{k+1}, q_{k+1})
\end{equation*}
for all $k\geq 0$.
Then we further derive from \eqref{uq1} and \eqref{uq2} that for any $m\geq 1$, $J(u_m, q_m)$ satisfies
\begin{align}
&J(u_m, q_m) + \sum_{k=0}^{m-1} E( q_k, q_{k+1}) \nonumber \\
&+\sum_{k=0}^{m-1}\left(\Vert L\lt(u_{k+1}, q_k\rt)-L\lt(u_k, q_k\rt)\Vert_X^2+\Vert C\lt(u_{k+1}-u_k\rt)\Vert_Y^2 +\Vert L\lt(u_{k+1}, q_{k+1}\rt)-L\lt(u_{k+1}, q_k\rt)\Vert_X^2\right) \nonumber \\
&\leq J(u_0, q_0),\label{est}
\end{align}

 which implies that  $\sum_{k=0}^{\infty}\Vert L\lt(u_{k+1}-u_k, q_k\rt)\Vert_X^2$ is bounded.  Then we can conclude using Assumption~\ref{assum2}
%in \eqref{equiv} 
that $\{u_k\}_{k=1}^\infty$ forms a Cauchy sequence and thus converges to some $u^\star\in {dom(L)}$. Since $\sum_{k=0}^{m-1} E(q_k, q_{k+1})$ is uniformly bounded for all $m$, we can derive that $\{q_k\}_{k=1}^\infty$ converges to some $q^\star\in Z$ from the strict convexity of $\psi$. As the sequence $\{J(u_k,q_k)\}_{k=1}^\infty$ is monotone decreasing, there exists a limit $J^\star$.  Following the similar argument in the proof of Theorem \ref{exist}, we conclude that  $J^\star = J(u^\star, q^\star)$ by Assumption~\ref{assumLC}.
This completes the proof of the convergence. 
\end{proof}

\begin{remark}
 If \eqref{nec} has a unique solution in a neighborhood of initial guess $(u_0, q_0)$, then the solution is a local minimizer of the least-squares formulation \eqref{minilsq}, and we can apply the ADI method to generate a sequence that converges to this local minimizer 
as a plausible approximation of the exact coefficients.
\end{remark}

\end{section}

\begin{section}{Diffusive optical tomography}\label{sec:DOT}
  We take the diffusive optical tomography (DOT) as an example to illustrate the total least-squares approach for the concrete inverse medium problem in this section. We will introduce the mixed regularization term, present the least-squares formulation of the first-order system of DOT, and then verify the assumptions in Section~\ref{sec:well} for the proposed formulation. 
We shall use the standard notations for Sobolev spaces. 
%Moreover, $\mathcal{D}(\overline{\Omega})$ denotes the space of infinitely differentiable functions 
%with compact support on $\Omega$ and $H^{-1/2}(\partial\Omega)$ denotes the dual space of  $H^{1/2}(\partial\Omega)$.
The objective of the DOT problem is to determine the unknown diffusion and absorption coefficients $\sigma$ and $\mu$ 
 simultaneously in a Lipschitz domain $\Omega\in \mathbb{R}^d$ ($d=2,3$) from the model equation 
\begin{equation}\label{eq:dot}
-\nabla \cdot (\sigma(x) \nabla u)+\mu(x)\,u=0 ~~\mbox{in} ~~\Om
\end{equation}
 with a pair of Cauchy data $(f,h)$ on the boundary $\partial\Omega$, i.e., 
\begin{equation*}
u|_{\partial\Omega}=f,\quad\frac{\partial u}{\partial \nu}|_{\partial\Omega}=h.
\end{equation*}
%In this DOT model, we have the exact data of the Neumann boundary condition $h\in H^{1/2}(\partial \Omega)$, and 
Throughout this section, we shall use the notation $g: = \delta_{\partial\Omega} h$ in the total least-squares formulation, where $\delta_{\partial \Omega}$ denotes the Neumann to source map. To define the Neumann to source map, we first introduce the  boundary restriction mapping  $\gamma_0$ on $\mathcal{D}(\overline{\Omega})$, i.e., 
$\gamma_0 u$ denotes the boundary value of $u\in \mathcal{D}(\overline{\Omega})$. Then we use  $T$ to denote the trace operator \cite{girault2012finite}, which is formally defined to be the unique linear continuous extension of $\gamma_0$ as an operator from $L^2(\Omega)$ onto $H^{-1/2}(\partial\Omega)$.  
Using Riesz representation theorem, there exists a function in $L^2(\Omega)$, denoted by $\delta_{\partial\Omega} h$, such that for any $v\in L^2(\Omega)$,
$$(\delta_{\partial\Omega} h, v)_{L^2(\Omega)} = \langle h, Tv\rangle_{H^{1/2}(\partial\Omega), H^{-1/2}(\partial\Omega)},$$
which will be denoted by $g$ in the least-squares formulation.
% Now we can define $\delta_{\partial \Omega} h$ corresponding to the Neumann boundary condition $h\in H^{1/2}(\partial \Omega)$ as a functional on $L^2(\Omega)$: for any $v\in L^2(\Omega)$, define
%$$\delta_{\partial\Omega} h(v): = \langle h, Tv\rangle_{H^{1/2}(\partial\Omega), H^{-1/2}(\partial\Omega)}.$$
%%It is better to give the definition of this operator $T$. Done
%% (in page 8)
%%For any $v\in L^2(\Omega)$,
%%\begin{eqnarray*}
%%\delta_{\partial\Omega} h(v)\leq c_0 \Vert h\Vert_{H^{1/2}(\partial\Omega)} \Vert v\Vert_{L^2(\Omega)},
%%\end{eqnarray*}
%%where $c_0$ denotes some constant,  
%It is noted that $\delta_{\partial\Omega} h$ is a bounded linear functional on $L^2(\Omega)$. Using the Riesz representation theorem, there exists a function in $L^2(\Omega)$, still denoted as $\delta_{\partial\Omega}h$, such that for any $v\in L^2(\Omega)$,
% $$(\delta_{\partial\Omega} h, v)_{L^2(\Omega)}=\delta_{\partial\Omega} h(v). $$
%Therefore $\delta_{\partial\Omega} h$ is well-defined as a function in $L^2(\Omega)$ and will be denoted by $g$ in the least-squares formulation.
%

\begin{subsection}{Mixed regularization}

In this subsection, we present the mixed regularization term for the DOT problem.
As the regularization term $\psi$ in the least-squares formulation \eqref{minilsq} shall encode the priori information, e.g., sparsity, continuity, lower or upper bound and other properties of the unknown coefficients, it is essential to choose an appropriate regularization term for a concrete inverse problem to ensure satisfactory reconstructions. 
In this work, we introduce a mixed $L^1$--$H^1$ regularization term $\phi$ for a coefficient $q:\Omega\to \mathbb{R}$:
%Does this regularization term fit the framework (1.4) and will the results in sections 2 and 3 be still true? 
%Basically we have now two regularization terms with two regularization parameters, but the previous 
%framework has only one term with one parameter. This regularization term  
%$\phi$ does not violate the assumptions we need on $\psi$. Yes, there is an abuse of notation as the regularization parameters %$\alpha$ and $\beta$ are inside of $\phi$, i.e., we take $\phi(q; \alpha, \beta, q_{_0}, {q_{_1}})$ 
%as $\alpha\psi(q)$. The verification is added to subsection 4.2.
\begin{equation}\label{reg}
\phi(q; \alpha, \beta, q_{_0}, {q_{_1}})=\int_\Omega \frac{\alpha}{2}(|\nabla q|^2+q^2)dx+
\int_\Omega\beta|q|\,dx+\chi(q; q_{_0}, q_{_1}),
\end{equation}
In this revision we change the semi-norm of $H^1$ into the full norm in the regularization to ensure the strict convexity of $\phi$.
where $\chi(q; q_{_0}, q_{_1})$ is given by 
\begin{equation*}
\chi(q; q_{_0}, q_{_1})=
\left\{
\begin{array}{cc} 0 & {q_{_0}}\le q\le q_{_1}, \\ \\ 
\infty & \mbox{otherwise,}\end{array}
 \right.
\end{equation*}
and $q_{_0}$ and $q_{_1}$ are the {lower and upper} bounds of the coefficient $q$ respectively.  The first term $\int_\Omega \frac{\alpha}{2}(|\nabla q|^2+q^2)dx$ is the $H^1$ regularization term, the second term $\int_\Omega\beta|q|\,dx$ corresponds to the $L^1$ regularization, and the third term  $\chi(q; q_{_0},q_{_1})$ enforces the reconstruction to meet the constrains of the coefficient.

In practice, the $L^1$ regularization enhances the reconstruction and helps find geometrically sharp inclusions, but might cause spiky results. 
The $H^1$ regularization generates the reconstructions with overall clear structures, while the retrieved background may be blurry. 
Compared with other more conventional regularization methods, this mixed regularization technique in \eqref{reg} combines two penalty terms and effectively promotes multiple features of the target solution, that is, it enhances the sparsity of the solution while it still preserves the overall structure at the same time. 
The scalar parameters $\alpha$ and $\beta$ need to be chosen carefully to compromise the two regularization terms.
\end{subsection}
\begin{subsection}{First-order formulation of DOT}
As we have emphasized in the Introduction, it may have some advantages to make use of 
the residuals of the first-order system of the model equation \eqref{eq:dot}, instead of the residual of the 
original equation in the formulation \eqref{minilsq}, when 
we aim at recovering two unknown coefficients $\sigma$ and $\mu$ simultaneously. 
Similarly to the formulation \eqref{minilsq}, we now have $ q=(\sigma, \mu)$, and the operator $L$ is given by 
\begin{equation}\label{opeL}
L(u,\m{p},\sigma,\mu)=\begin{pmatrix}
-\nabla\cdot \m{p}+\mu\,u\\
\m{p}-\sigma\,\nabla u
\end{pmatrix},
\end{equation}
where we have introduced an auxiliary vector flux $\m{p}$ and each entry of $L$ is of a first-order form such that two coefficients are separated naturally. Clearly, $L(v, q)$ is still bilinear with respect to the state variables 
$v=(u,\m{p})$ and coefficients $q=(\sigma, \mu)$.  %In this model we look for coefficients $\sigma$ 
%in 
%\begin{equation*}
% \{\sigma\in H^1(\Omega):  \sigma \geq \sigma_{_0} \ a.e.\ in \ \Omega \}, 
%\end{equation*}
%and $\mu$ in \begin{equation*}
% \{\mu\in H^1(\Omega):\mu \geq \mu_{_0} \ a.e.\ in \ \Omega \},
%\end{equation*}
Using the first-order system, we can then come to the following total least-squares functional: 
\begin{equation}  \label{inv-sigmamu}
J(u, \m{p},\sigma,\mu)=\Vert-\nabla\cdot \m{p}+\mu(x)u- g\Vert_{L^2(\Omega)}^2+\Vert \m{p}-\sigma(x)\,\nabla u\Vert_{(L^2(\Omega))^d}^2+\Vert Cu-f\Vert_{L^2(\partial\Omega)}^2
+\psi_1(\sigma)+\psi_2(\mu),
\end{equation}
where $C$ is the trace operator,  and $\psi_1(\sigma) =\phi(\sigma; \alpha_\sigma, \beta_\sigma, \sigma_{_0},\sigma_{_1})$ and $\psi_2(\mu) =\phi(\mu; \alpha_\mu, \beta_\mu, \mu_{_0},\mu_{_1})$ are the corresponding  mixed regularization terms of $\sigma$ and $\mu$ defined as in \eqref{reg},  $ \mu_{_0}$, $\mu_{_1}$ are lower and upper bounds of $\mu$,  and $ \sigma_{_0}$, $\sigma_{_1}$ are lower and upper bounds of $\sigma$.  
We shall minimize 
\eqref{inv-sigmamu} over $(u, \m{p},\sigma,\mu)\in L^2(\Omega)\times (L^2(\Omega))^d\times L^2(\Omega)\times L^2(\Omega)$, that is, $X =  L^2(\Omega)\times (L^2(\Omega))^d$, $Z =  L^2(\Omega)\times L^2(\Omega)$.  

We will apply the ADI method to solve the least-squares formulation of $v=(u,\m{p})$ and $q=(\sigma, \mu)$:
\begin{equation}
\min_{(v,q)\in X\times Z} J(u, \m{p},\sigma,\mu)=\Vert-\nabla\cdot \m{p}+\mu(x)u- g\Vert_{L^2(\Omega)}^2+\Vert \m{p}-\sigma(x)\,\nabla u\Vert_{(L^2(\Omega))^2}^2+\Vert Cu-f\Vert_{L^2(\partial\Omega)}^2
+\psi_1(\sigma)+\psi_2(\mu).\label{DOTminilsq}
\end{equation}
Given an initial guess $(u_0, \m{p}_0,\sigma_0,\mu_0)$, we find a sequence $(u_k, \m{p}_k,\sigma_k,\mu_k)$ for $k\geq 1$ as below:
\begin{itemize}
\item
 Given $\sigma_k$, $\mu_k$, find $u=u_{k+1}$, $\m{p}=\m{p}_{k+1}$ by solving
\begin{equation*}
\min_{(u,\m{p})\in X}J_1(u,\m{p})=\Vert-\nabla\cdot \m{p}+\mu_k(x)u-g\Vert_{L^2(\Omega)}^2+\Vert \m{p}-\sigma_k(x)\,\nabla u\Vert_{(L^2(\Omega))^d}^2+\Vert Cu-f\Vert_{L^2(\partial\Omega)}^2,
\end{equation*}
\item
Given $u_{k+1}$, $\m{p}_{k+1}$, find $\sigma=\sigma_{k+1}$, $\mu =\mu_{k+1}$ by solving
\begin{equation*}
\min_{(\sigma,\mu)\in Z}J_2(\sigma,\mu)=\Vert-\nabla\cdot \m{p}_{k+1}+\mu(x)u_{k+1}-g\Vert_{L^2(\Omega)}^2+\Vert \m{p}_{k+1}-\sigma(x)\,\nabla u_{k+1}\Vert_{(L^2(\Omega))^d}^2
+\psi_1(\sigma)+\psi_2(\mu).
\end{equation*}
\end{itemize}
\end{subsection}

\begin{subsection}{Well-posedness of the least-squares formulation for DOT}
Recall that we have proved the well-posedness of the least-squares formulation 
in Section \ref{sec:well} for general inverse medium problems. This part is devoted to the verification of assumptions in Section \ref{sec:well} for the formulation \eqref{DOTminilsq} to ensure its well-posedness. 
% In order to ensure the well-posedness of this least-squares formulation \eqref{DOTminilsq} with the proof
%in Section \ref{sec:well}, we shall verify the assumptions in Section  \ref{sec:well} for the formulation \eqref{DOTminilsq}.

Firstly we consider Assumption \ref{assum1} on the regularization terms. 
%$\psi_1(\sigma) =\phi(\sigma; \alpha_\sigma, \beta_\sigma, \sigma_{_0},\sigma_{_1})$ and $\psi_2(\mu) =\phi(\mu; \alpha_\mu, \beta_\mu, \mu_{_0},\mu_{_1})$. 
%As the mixed regularization terms enforce the coefficients to have uniform lower and upper bounds, combined with the $H^1$-norm terms, 
It is observed from the formula \eqref{reg} that each term of $\psi_1(\sigma)$ and $\psi_2(\mu)$ in \eqref{DOTminilsq} is convex and weakly lower semicontinuous. 
As the first term of \eqref{reg} is the $H^1$ regularization term, $\psi_1(\sigma)$ and $\psi_2(\mu)$ are strictly convex by definition.
One can also observe that there exists a positive constant $c$ such that
\begin{equation}\label{DOTcoer}
\begin{aligned}
\psi_1(\sigma) = \int_\Omega \frac{\alpha}{2}(|\nabla \sigma|^2+\sigma^2)dx+
\int_\Omega\beta|\sigma|\,dx+\chi(\sigma; \sigma_{_0}, \sigma_{_1})\geq c \Vert\sigma\Vert^2_{H^1(\Omega)},\\
\psi_2(\mu) = \int_\Omega \frac{\alpha}{2}(|\nabla \mu|^2+\mu^2)dx+
\int_\Omega\beta|\mu|\,dx+\chi(\mu; \mu_{_0}, \mu_{_1})\geq c \Vert\mu\Vert^2_{H^1(\Omega)},
\end{aligned}
\end{equation}
which imply that $\psi_1$ and $\psi_2$ are coercive in $H^1$-norm.

Next we verify Assumption \ref{assum2} on the closedness of $L_q$ and the coercivity of its graph norm. 
For fixed $q=(\sigma,\mu)$, we denote the entries of $L_q$ by $L_{q,1}$, $L_{q,2}$, i.e.,  for $\delta v= (\delta u,\delta\m{p})\in dom(L)$, 
$$L_{q,1} \delta v =-\nabla\cdot \delta\m{p}+\mu\delta u,\quad L_{q,2} \delta v = \delta\m{p}-\sigma\nabla\delta  u\,.$$
Applying H\"{o}lder's inequality leads to 
\begin{align}
\Vert L_q\delta v\Vert_{X}\leq& \Vert \delta\m{p}\Vert_{H(\text{div}, \Omega)}+\Vert \delta u\Vert_{H^1(\Omega)}\Vert \mu\Vert_{L^{\infty}(\Omega)}+\Vert \delta u\Vert_{H^1(\Omega)}\Vert \sigma\Vert_{L^{\infty}(\Omega)}  \label{contL}
%\leq&c ( \Vert \delta\m{p}\Vert_{H(div, \Omega)} +\Vert \delta u\Vert_{H^1(\Omega)})   
\end{align}
for $q$ in the level set $\{q\in Z: \psi(q) = \psi_1(\sigma)+\psi_2(\mu)\leq c_0\}$ for some constant $c_0$. Then we can deduce from \eqref{contL} that $L_q$ is a closed linear operator.
Next, we verify the coercivity of the graph norm. For simplicity, we consider the model problem with homogeneous Neumann boundary condition $h = 0$ for the state variable, and set a side constraint for $dom(L)$ as $\m{p}\cdot\nu =0$ on $\partial \Omega$. Introduce the following notations
\begin{equation}
\sigma \nabla u-\m{p}=\tilde{m},\quad -\nabla\cdot \m{p}+\mu\,u=\tilde{g}. \label{mg}
\end{equation}
From \eqref{mg} and integration by part, we can derive
\begin{equation*}
\int_\Omega \sigma \nabla u\cdot\nabla u dx+\int_\Omega\mu u\cdot udx=\int_\Omega\nabla u\cdot \tilde{m}dx+\int_\Omega u\cdot \tilde{g}dx\,,
\end{equation*}
which implies 
\begin{equation}\label{graphsecond}
\Vert \nabla u\Vert_{(L^2(\Omega))^d}^2+\Vert u\Vert_{L^2(\Omega)}^2 +\Vert \nabla \cdot \m{p}\Vert_{L^2(\Omega)}^2+\Vert \m{p}\Vert_{(L^2(\Omega))^d}^2\le c\,(\Vert \tilde{m}\Vert_{(L^2(\Omega))^d}^2+\Vert \tilde{g}\Vert_{L^2(\Omega)}^2)
\end{equation}
for some constant $c>0$ when $q=(\sigma,\mu)$ satisfies $\psi_1(\sigma)\leq c_0$ and $\psi_2(\mu)\leq c_0$ for some constant $c_0>0$. In this way we verify that the graph norm corresponding to $L_q$ defined as 
$$|(u,\m{p}) |^2_{W,q}= \Vert\sigma \nabla u-\m{p}\Vert_{(L^2(\Omega))^d}^2+\Vert -\nabla\cdot \m{p}+\mu\,u\Vert_{L^2(\Omega)}^2$$
 is uniformly coercive.
 
Then we consider Assumption~\ref{assumLC} on the weakly sequentially closedness of operators $L$ and $C$. 
It is noted that Assumption~\ref{assumLC} is applied in Section~\ref{sec:well} for general inverse medium problems to prove that the operator $L$ maps a subsequence of a bounded sequence $\{(u_n, q_n)\}_{n=1}^\infty$ to a converging sequence $\{L(u_n,q_n)\}_{n=1}^\infty$ with its limit equal to $L(u,q)$, where $(u,q)$ is the limit of $\{(u_n, q_n)\}_{n=1}^\infty$.
 In the analysis of the concrete DOT problem \eqref{eq:dot}, the corresponding sequence $\{(u_n, \m{p}_n,\sigma_n,\mu_n)\}_{n=1}^\infty$ is actually bounded in a stronger norm than  $\Vert\cdot\Vert_X$ and $\Vert\cdot\Vert_Z$, as shown in \eqref{DOTcoer} and \eqref{graphsecond}. 
 As stated in Remark~\ref{rmkweakclose}, to prove the well-poseness of the least-squares formulation \eqref{DOTminilsq}, it suffices to verify that for a sequence $\{(u_n, \m{p}_n,\sigma_n,\mu_n)\}_{n=1}^\infty$ bounded in $H^1(\Omega)\times H(\text{div},\Omega)\times H^1(\Omega)\times H^1(\Omega)$, there exists a subsequence weakly converging to $(u,\m{p},\sigma,\mu)$, and operator $L$ defined in \eqref{opeL} satisfies $\{L(u_n, \m{p}_n,\sigma_n,\mu_n)\}_{n=1}^\infty$ weakly converges to $L(u, \m{p},\sigma,\mu)$ in $X$.
% 
% \mm{(We will first study the continuity of $L$ }
% For given $(u,\m{p})$,  we denote by $\Phi_u: =(\Phi_{u,1},\Phi_{u,2})$ the terms 
%$$
%\Phi_{u,1}\delta q=\nabla u\delta\sigma,\quad \Phi_{u,2}\delta q=u\delta\mu\,.
%$$
%We can deduce the continuity of the operator $\Phi_u$ by the Sobolev embedding that 
%\begin{eqnarray}
%\Vert \Phi_u\delta q\Vert_{L^2(\Omega)} &\le& c(\Vert\nabla u\Vert_{L^4(\Omega)} \Vert \delta\sigma\Vert_{L^4(\Omega)}+\Vert u\Vert_{L^4(\Omega)}\Vert \delta\mu\Vert_{L^4(\Omega)})\nonumber \\
%&\leq& c(\Vert u\Vert_{H^2(\Omega)} \Vert \delta\sigma\Vert_{H^1(\Omega)}+\Vert u\Vert_{H^2(\Omega)}\Vert \delta\mu\Vert_{H^1(\Omega)})\,.
%\end{eqnarray}
%

The verification of the desired property of $L$ is as follows.  Given the bounded sequence  $\{(u_n, \m{p}_n,\sigma_n,\mu_n)\}_{n=1}^\infty$, there exists a subsequence, still denoted as $\{(u_n, \m{p}_n,\sigma_n,\mu_n)\}_{n=1}^\infty$, weakly converging to $(v,q):=(u, \m{p},\sigma,\mu)$ in $H^1(\Omega)\times H(\text{div},\Omega)\times H^1(\Omega)\times H^1(\Omega)$. Denoting $(u_n,\m{p}_n)$ by $v_n$ and $(\sigma_n,\mu_n)$ by $q_n$, there holds for any $V=(V_1,V_2)\in X$ that
 \begin{equation*}
\begin{split}
&\,\quad \lim_{k\to\infty}(L(v_n, q_n), V)_X-(L(v, q), V)_X\\
&=\lim_{k\to\infty}((L_{q_n}(v_n-v), V)_X+(L(v, q_n-q), V)_X)\\
&=\lim_{k\to\infty}\big((-\nabla\cdot (\m{p}_n-\m{p})+\mu_n(u_n-u),V_1)+ ((\m{p}_n-\m{p})-\sigma_n\nabla(u_n-u),V_2)\\
&\quad +((\sigma_n-\sigma)\nabla u, V_1)+((\mu_n-\mu)u,V_2)\big)\\
&=\,0,
\end{split}
\end{equation*}
%\lim (\sigma_n\nabla(u_n-u),V_2)\leq \lim \Vert \sigma_n\Vert_{L^\infty}(|\nabla(u_n-u)|,|V_2|) =0
where we have used the weak convergence of $\{(u_n, \m{p}_n,\sigma_n,\mu_n)\}_{n=1}^\infty$ and H\"{o}lder's inequality.  Therefore, $\{L(u_n, \m{p}_n,\sigma_n,\mu_n)\}_{n=1}^\infty$ weakly converges to $L(u, \m{p},\sigma,\mu)$ in $X$.
On the other hand, as $C$ is the trace operator on $\partial \Omega$ in DOT problem, it is linear and thus weakly sequentially closed, which completes our verification.

\end{subsection}

\end{section}

\begin{section}{Numerical experiments}\label{sec:Num}

In this section, we carry out some numerical experiments on the DOT problem in different scenarios to illustrate the efficiency 
and robustness of the proposed two-stage algorithm in this work.
Throughout these examples, we shall assume that we apply the Neumann boundary data $h$ on $\partial \Omega$ and measure the corresponding Dirichlet data $f$ to reconstruct the diffusion coefficient $\sigma$ and the absorption coefficient $\mu$ simultaneously. The basic algorithm involves two stages: 
we apply the direct sampling method (DSM) in the first stage to get some initial approximations of the two unknown coefficients, and then adopt the total least-squares method to achieve more accurate reconstructions of the coefficients. 

\begin{subsection}{Direct sampling method for initialization}
For all the numerical experiments, we shall use the DSM in the first stage of our algorithm, in an attempt 
to effectively locate the multiple inclusions inside the computational domain with limited measurement data. 
Here we give a brief description of DSM and refer the readers to \cite{chow2020direct} for more technical details 
about this DSM that can identify multiple coefficients.  The DSM develops two separate families of probing functions, 
i.e., the monopole and dipole probing functions, for constructing separate index functions for multiple physical coefficients. 
The inhomogeneities of coefficients can be approximated based on index functions  due to the following two observations: 
the difference of scattered fields caused by inclusions can be approximated by the sum of Green's functions of the homogeneous medium and their gradients; and the two sets of probing functions  have the mutually almost orthogonality property, i.e., they only interact closely with the Green's functions and their gradients respectively. Thus we can decouple the  monopole and dipole effects and derive index functions for separate physical properties.

In practice, if the value of an index function $\phi$ for one physical coefficient at a sampling point $x$ is close to $1$,  
the sampling point is likely to stay in the support of inhomogeneity; whereas if $\phi(x)$ is close to $0$, the sampling point $x$ probably stays outside of the support. Hence the index functions  give an image of the approximate support of the inhomogeneity, and we can determine a subdomain $D$ to locate the support from index functions. The subdomain $D$  could be chosen as $D = \{x \in \Omega:  \phi(x) \geq \theta\}$ with $\theta$ being a suitable cut-off value, then we restrict the index function $\phi$ on $D$.  We adopt this choice in the numerical experiments to remove the spurious oscillations in the homogeneous background. Once we have the restricted index function $\phi|_D$, we set the value of the approximation $\tilde{\phi}$ in $D$ as $\tilde{\phi}|_D = c_\phi \phi|_D$, where the
constant $c_\phi$ is some priori estimate of the true coefficient, and set the value of $\tilde{\phi}$ outside of $D$ as the background coefficient.
%This constant is assumed to be known previously and is to compensate for the fact that in DSM the index function $\phi$ is normalized with a maximum of unity. 
In this way, we obtain $\tilde{\phi}$  as the initial guess for the total least-squares method for our further 
reconstruction in the second stage.
\end{subsection}

\begin{subsection}{Examples}
%Experiments code saved in '(latest)Final_eachtwoinclusion_finitedifference_medium'
We present numerical results for some two-dimensional examples and
showcase the proposed two-stage least-squares method for inverse medium problems using both exact and noisy data. 
Here the objective functional $J$ in \eqref{inv-sigmamu} is discretized
by a staggered finite difference scheme \cite{virieux1984sh,robertsson1994viscoelastic}. 
%see Appendix \ref{staggered} for details. 
The  computational domain $\Omega= [0,1]\times[0,1]$ is divided into a uniform mesh consisting of small squares of width $h=0.01$. The noisy measurements $u^\delta$ are generated point-wise by the formula
$$u^\delta(x) = u(x)+\epsilon \eta \max_{x\in\partial\Omega}|u(x)|$$
where $\epsilon$ is the relative noise level and $\eta$ follows the standard Gaussian distribution. The subdomain $D$ is chosen using the formula
\begin{equation}\label{setD}
D = \{x\in\Omega: \phi(x)\geq \theta\},
\end{equation}
where $\phi$ is the index function from DSM and the cutoff value $\theta$ is taken in the range $(0.4,0.7)$.   As there exist limited theoretical results for the choice of regularization parameters for mixed regularization that we use, the regularization parameters $(\alpha_\sigma,\beta_\sigma,\alpha_\mu,\beta_\mu)$ in the functional for reconstructions are determined in a trial-and-error manner, and we present them for different examples in Table \ref{table1}. The maximum number $K$ of alternating iterations is set at 50.

All the computations were performed in MATLAB (R2018B) on a desktop computer.
\begin{exam}\label{example1}
 We consider the discontinuous diffusion and absorption coefficients $\sigma$ and $\mu$  with one inclusion each. The inclusion of $\sigma$ is of width $0.05$ and centered at $(0.25,0.65)$ as shown in Fig.~\ref{Ex1sig.sub.1} ; 
 the inclusion of $\mu$ is of width $0.05$ and centered at $(0.35,0.3)$ as shown in Fig.~\ref{Ex1sig.sub.5}. The magnitudes of the coefficients inside the inclusions are $20$.
\end{exam}
We shall use only one set of measurements to reconstruct the diffusion and absorption coefficients.
It is shown  in Fig.~\ref{Ex1sig.sub.2} and \ref{Ex1sig.sub.6} that the index functions from the DSM separate inclusions of different physical nature well and give approximated locations, while the exact locations of small inclusions are difficult to detect. 
If we simply take the maximal points of the index functions  in Fig.~\ref{Ex1sig.sub.2} and \ref{Ex1sig.sub.6} as  locations of the reconstructed inclusions, we may not be able to identify the true locations of inhomogeneity. 
Then we set the subdomain $D$ using information from the first stage by \eqref{setD}, see Fig.~\ref{Ex1sig.sub.3}, \ref{Ex1sig.sub.7},  and set the value of approximation out of the subdomain as equal to the background coefficients. 
As in Fig.~\ref{Ex1sig.sub.4}, \ref{Ex1sig.sub.8}, this example illustrates that the least-squares formulation in the second stage works very well to improve the reconstruction and provides a much more accurate location than 
that provided by DSM. With $10 \%$ noise in the measurements, the reconstructions remain accurate as in Fig.~\ref{Ex1.3_2.main}, which shows that the two-stage algorithm gives quite stable final reconstructions with respect to the noise 
even though the DSM reconstructions become more blurry in Figs.~\ref{Ex2sig.sub.2} and \ref{Ex2sig.sub.6}.  
The regularization parameters $(\alpha_\sigma,\beta_\sigma,\alpha_\mu,\beta_\mu)$ are presented in Table \ref{table1}. 

Compared with the reconstructions derived from the DSM, the improvement of the approximations for both $\sigma$ and $\mu$ is significant: the recovered background is now mostly homogeneous, and the magnitude and size of the inhomogeneity approximate those of the true coefficients well. These results indicate clearly the significant potential of the proposed least-squares formulation with mixed regularization for inverse medium problems.

From Table \ref{table1} we obtain an insightful observation about the mixed regularization: the magnitude of parameter $\beta$ is much larger than that of $\alpha$.
We can conclude that the $L^1$ penalty plays a predominant role in improving the performance of reconstruction for such inhomogeneous coefficients, whereas the $H^1$ penalty yields a locally smooth structure.

\begin{table}\label{table1}
\centering
\begin{tabular}{c || c c || c c}
\hline
Noise level & \multicolumn{2}{c}{$\epsilon=0$} & \multicolumn{2}{c}{$\epsilon > 0$} \\
\hline
Example & $(\alpha_\sigma, \beta_\sigma)$ & $(\alpha_\mu, \beta_\mu)$ & $(\alpha_\sigma, \beta_\sigma)$ & $(\alpha_\mu, \beta_\mu)$ \\
\hline\hline
1 & (1.0e-2, 2.0e-2 ) & (5.0e-4, 5.0e-4) & (1.0e-2, 2.0e-2) & (5.0e-4, 1.0e-3) \\
2.1 & (1.0e-3, 5.0e-3) & N/A & (1.0e-3, 1.0e-2) & N/A \\
2.2 & (1.0e-6, 1.0e-3) & N/A & (1.0e-6, 2.0e-3) & N/A \\
3 & (1.0e-3, 1.0e-2) & (1.0e-2, 5.0e-3) & (1.0e-3, 2.0e-2) & (1.0e-2, 5.0e-3) \\
4 & (1.0e-5, 5.0e-4) & N/A & (1.0e-5, 1.0e-3) & N/A \\
\end{tabular}
\caption{The regularization parameters $(\alpha,\beta)$ in each example for $\sigma$ and $\mu$ without noise and with noise. }\label{table1}
\end{table}

\begin{figure}[ht!]
\centering
\subfigure[ true $\sigma$]{
\label{Ex1sig.sub.1}
\includegraphics[width=0.23\textwidth]{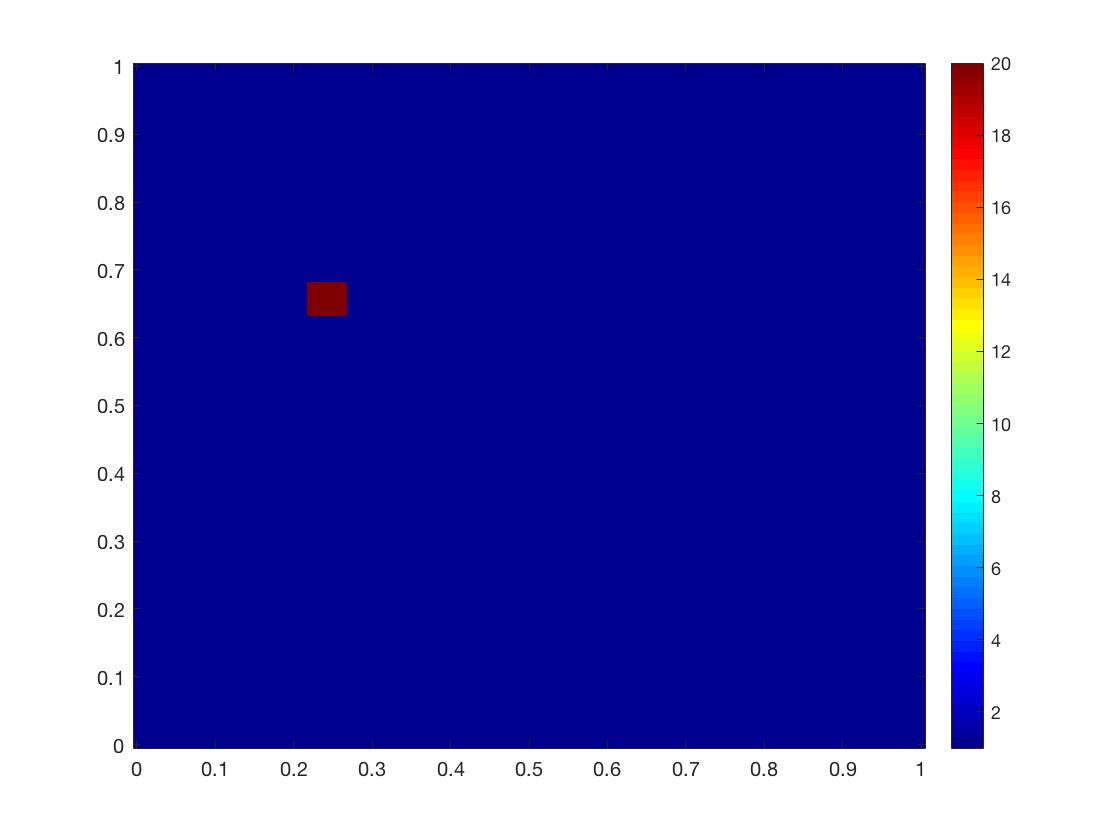}}
\subfigure[ index $\Phi$]{
\label{Ex1sig.sub.2}
\includegraphics[width=0.23\textwidth]{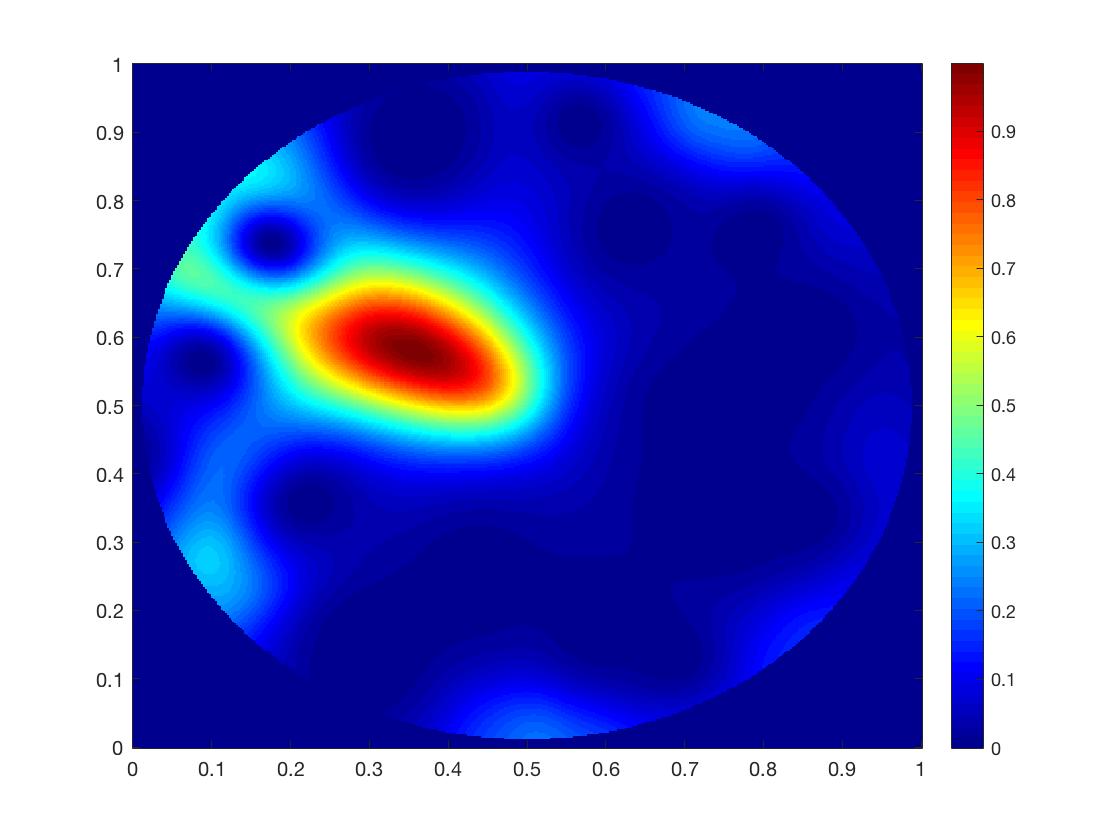}}
\subfigure[index $\Phi|_D$ ]{
\label{Ex1sig.sub.3}
\includegraphics[width=0.23\textwidth]{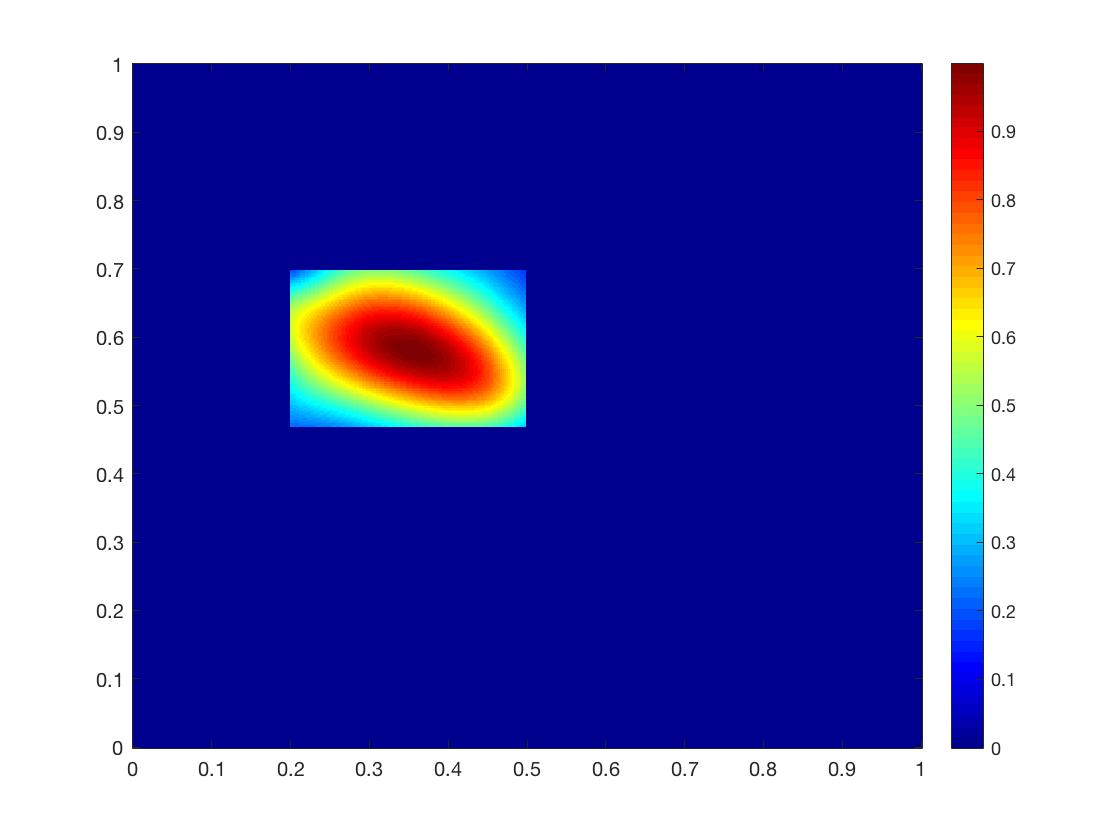}}
\subfigure[least-squares]{
\label{Ex1sig.sub.4}
\includegraphics[width=0.23\textwidth]{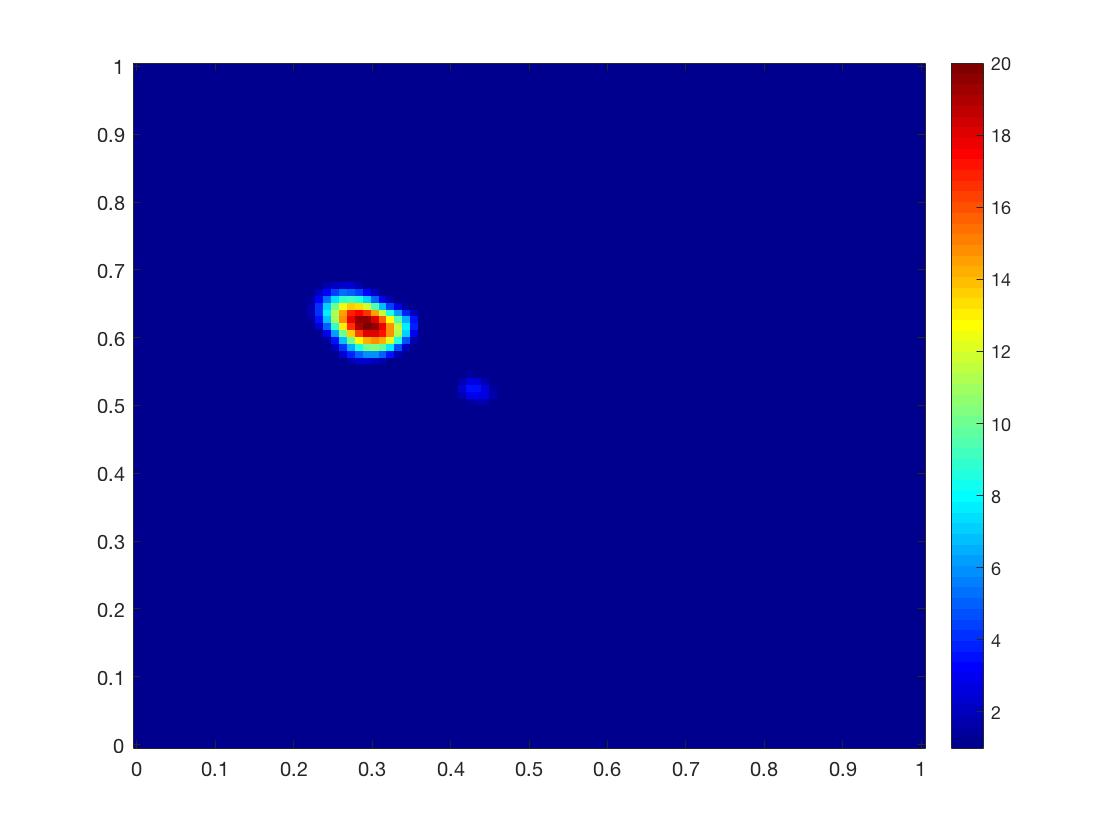}}
\subfigure[ true $\mu$]{
\label{Ex1sig.sub.5}
\includegraphics[width=0.23\textwidth]{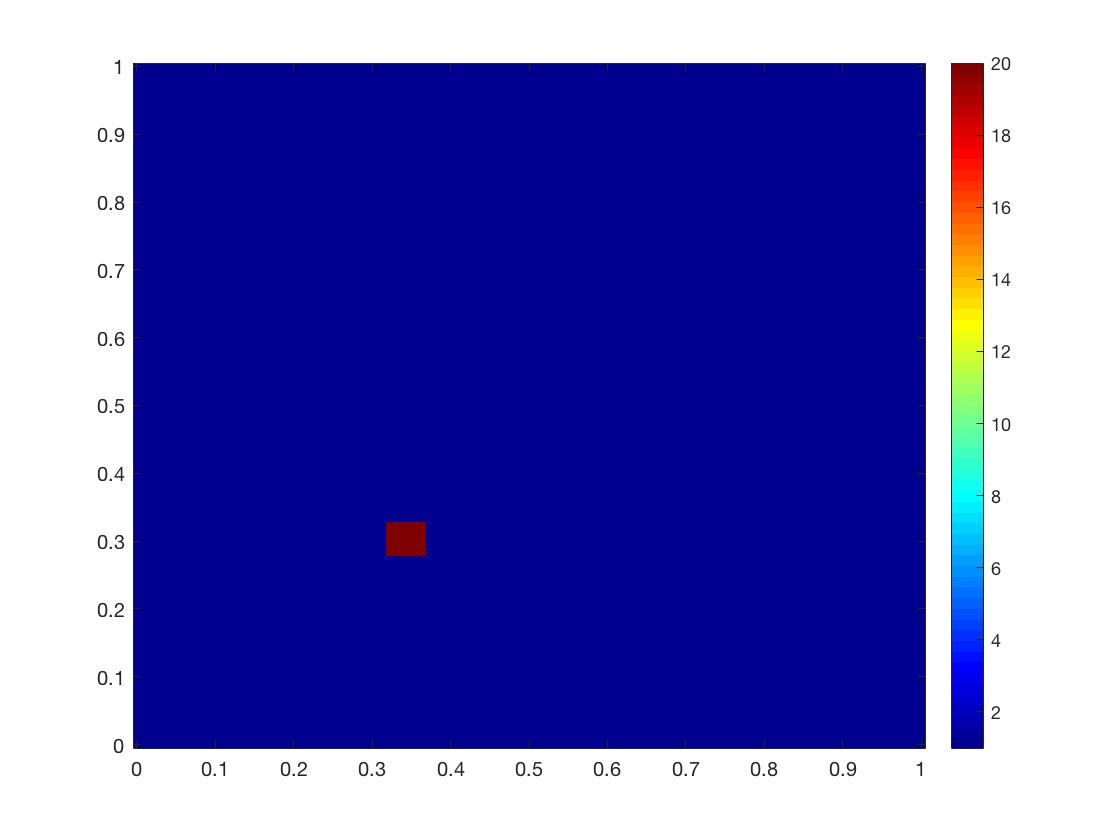}}
\subfigure[ index $\Phi$]{
\label{Ex1sig.sub.6}
\includegraphics[width=0.23\textwidth]{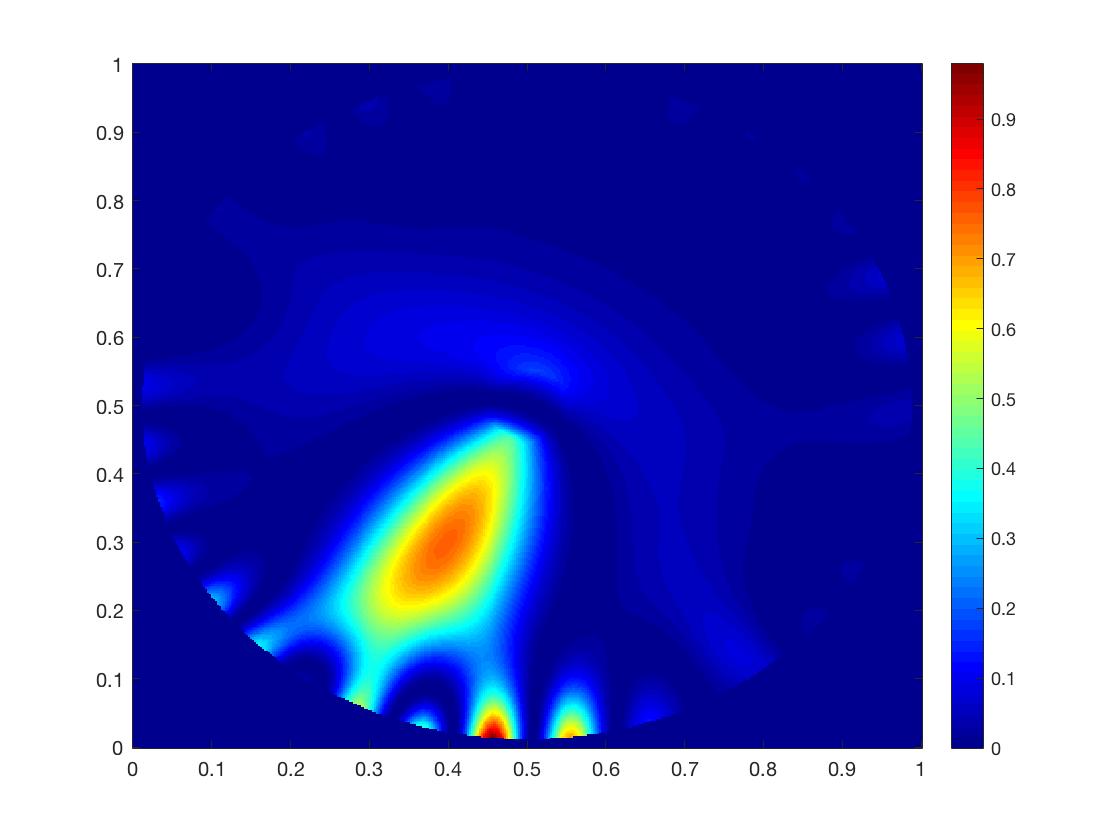}}
\subfigure[ index $\Phi|_D$]{
\label{Ex1sig.sub.7}
\includegraphics[width=0.23\textwidth]{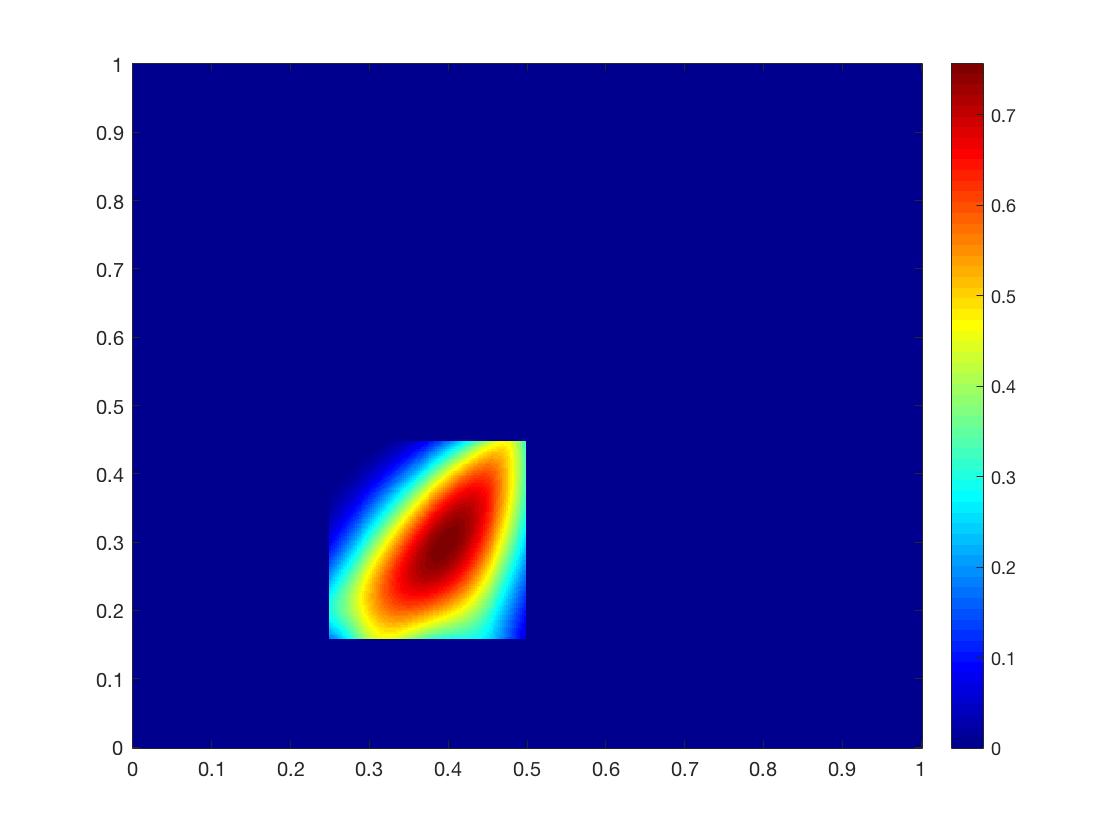}}
\subfigure[ least-squares]{
\label{Ex1sig.sub.8}
\includegraphics[width=0.23\textwidth]{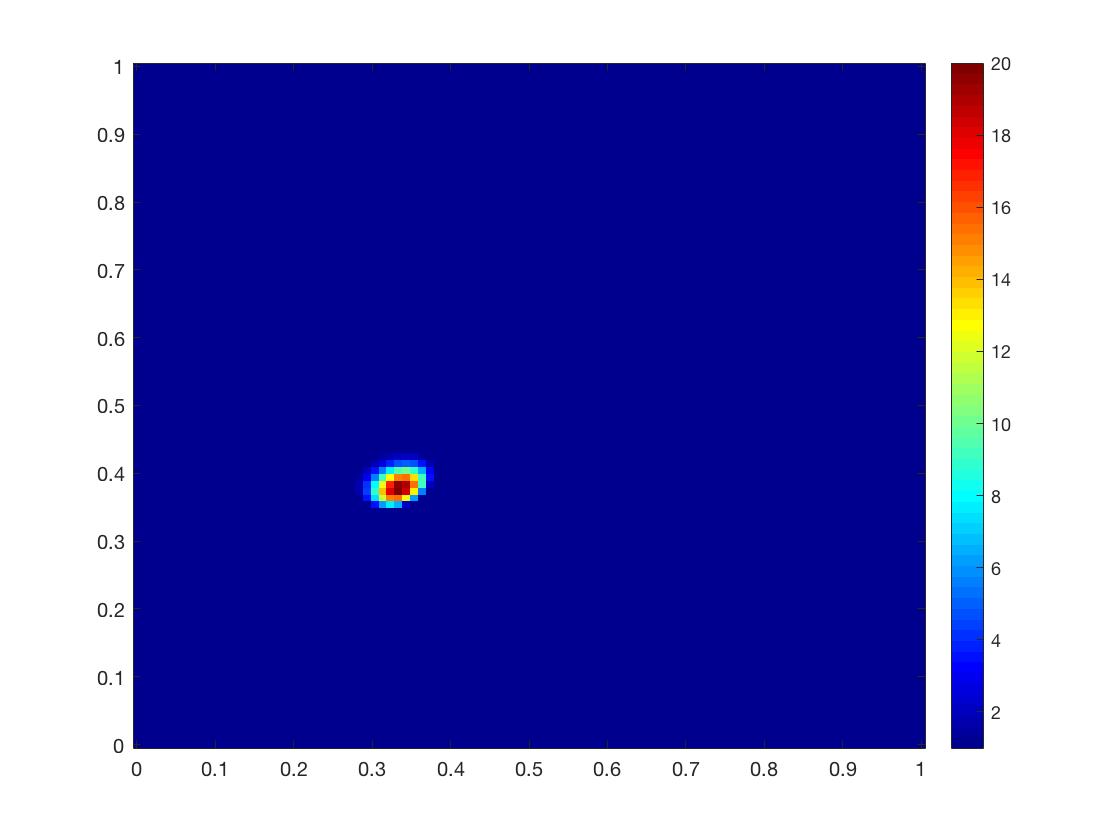}}
\caption{Numerical results for Example \ref{example1}: (a) true $\sigma$, (b) index $\Phi$ from DSM, (c) index $\Phi|_D$ (index function constrained to the chosen subdomain $D$), (d) least-squares reconstruction. (e), (f), (g), (h) are corresponding graphs for $\mu$.  These two rows are derived using exact data. }
\label{Ex123.main}
\end{figure}

\begin{figure}[ht!]
\centering
\subfigure[ true $\sigma$]{
\label{Ex2sig.sub.1}
\includegraphics[width=0.23\textwidth]{figure/oneeachsigmatrue.jpg}}
\subfigure[ index $\Phi$]{
\label{Ex2sig.sub.2}
\includegraphics[width=0.23\textwidth]{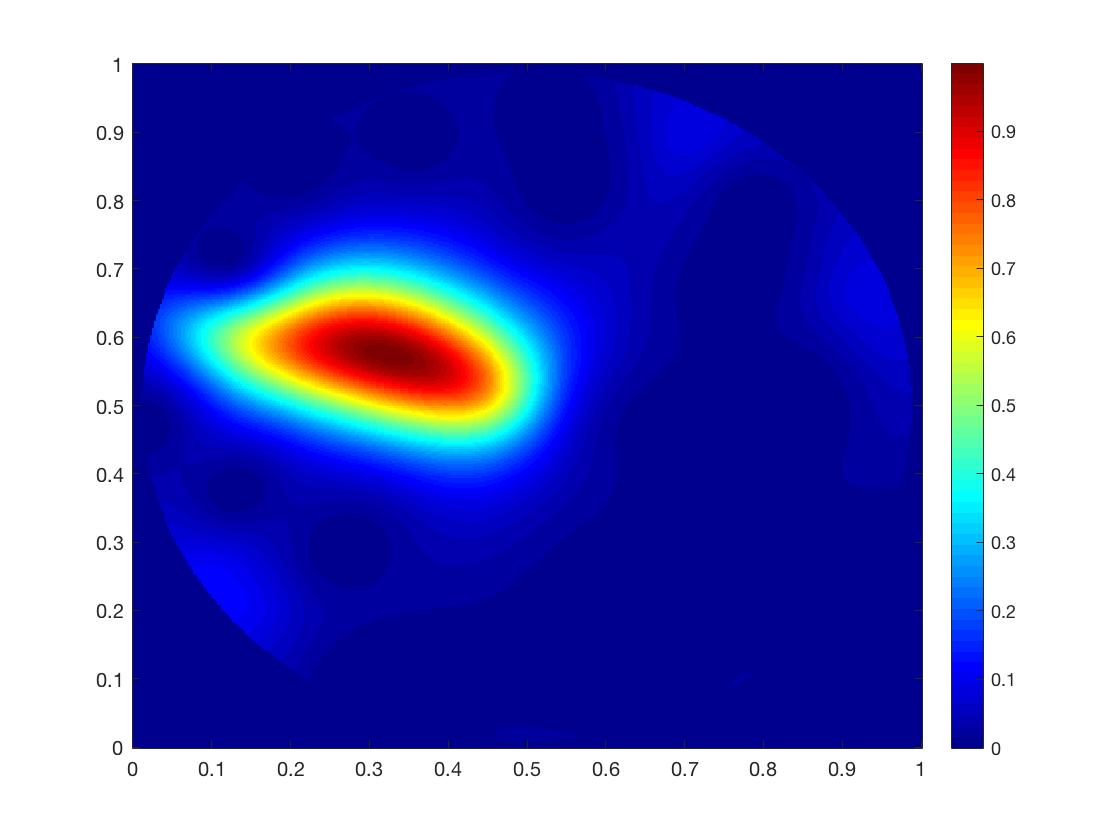}}
\subfigure[ index $\Phi|_D$ ]{
\label{Ex2sig.sub.3}
\includegraphics[width=0.23\textwidth]{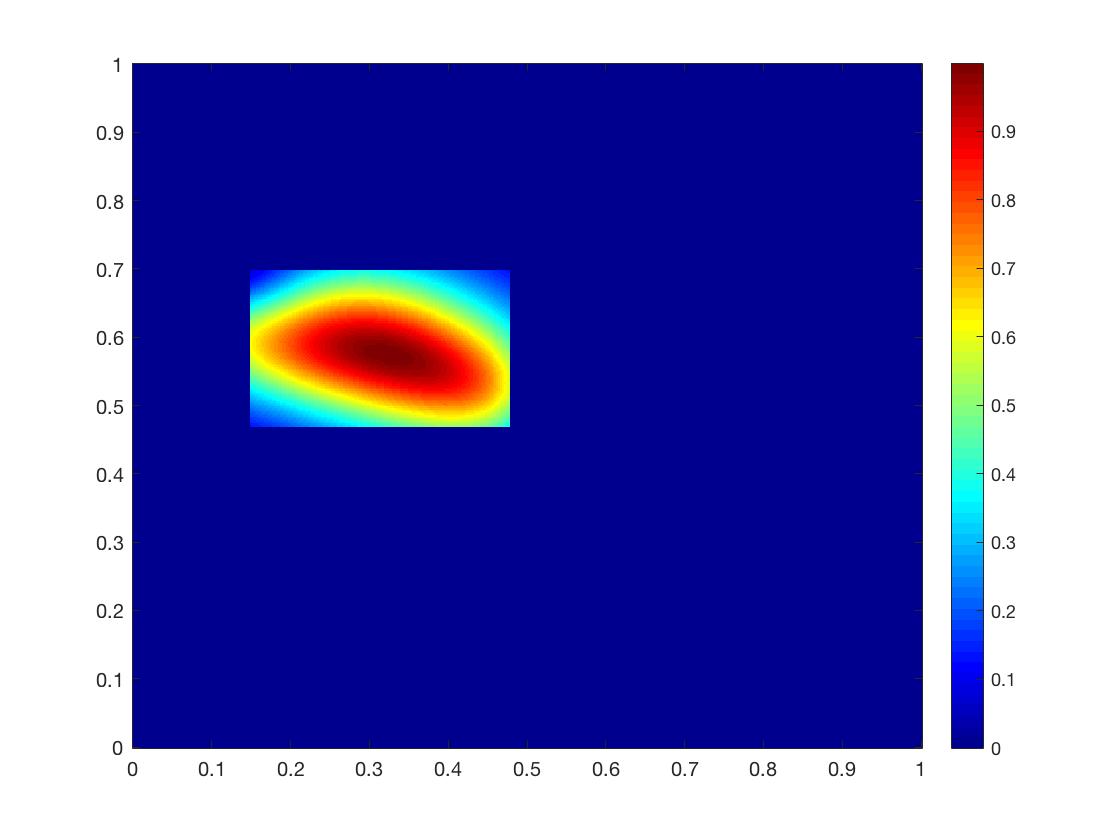}}
\subfigure[least-squares]{
\label{Ex2sig.sub.4}
\includegraphics[width=0.23\textwidth]{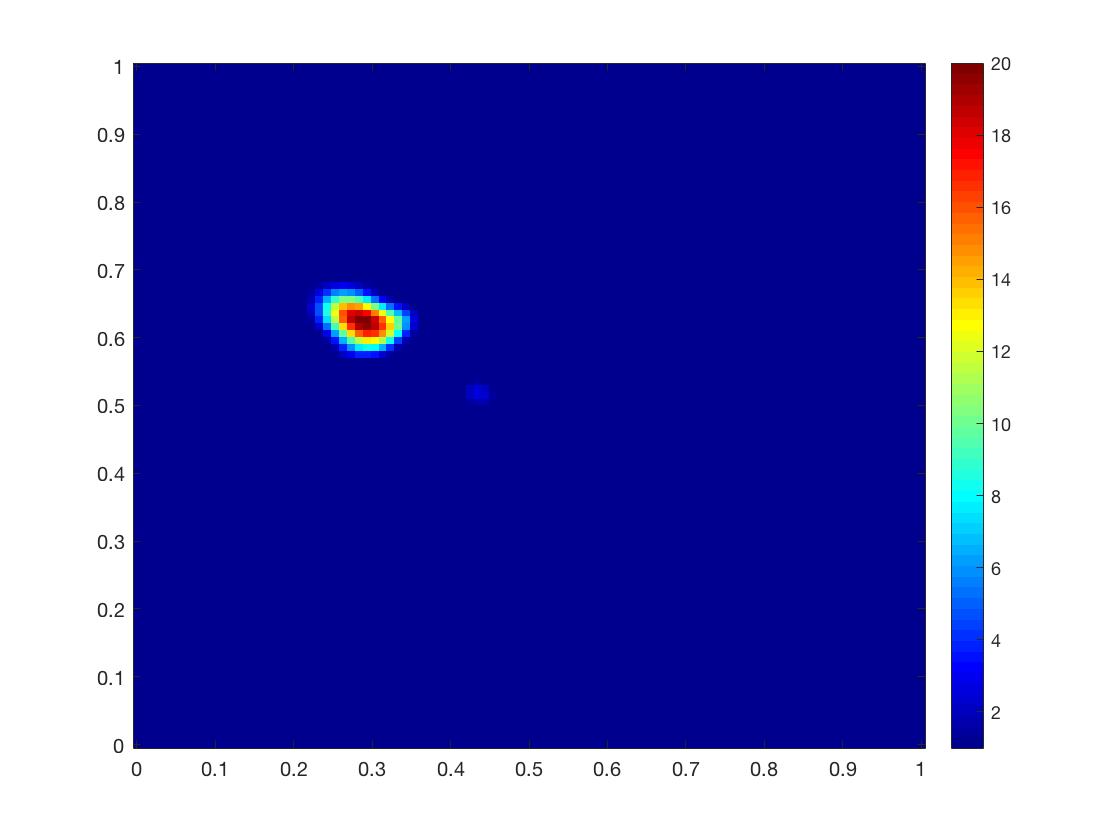}}
\subfigure[ true $\mu$]{
\label{Ex2sig.sub.5}
\includegraphics[width=0.23\textwidth]{figure/oneeachmutrue.jpg}}
\subfigure[index $\Phi$]{
\label{Ex2sig.sub.6}
\includegraphics[width=0.23\textwidth]{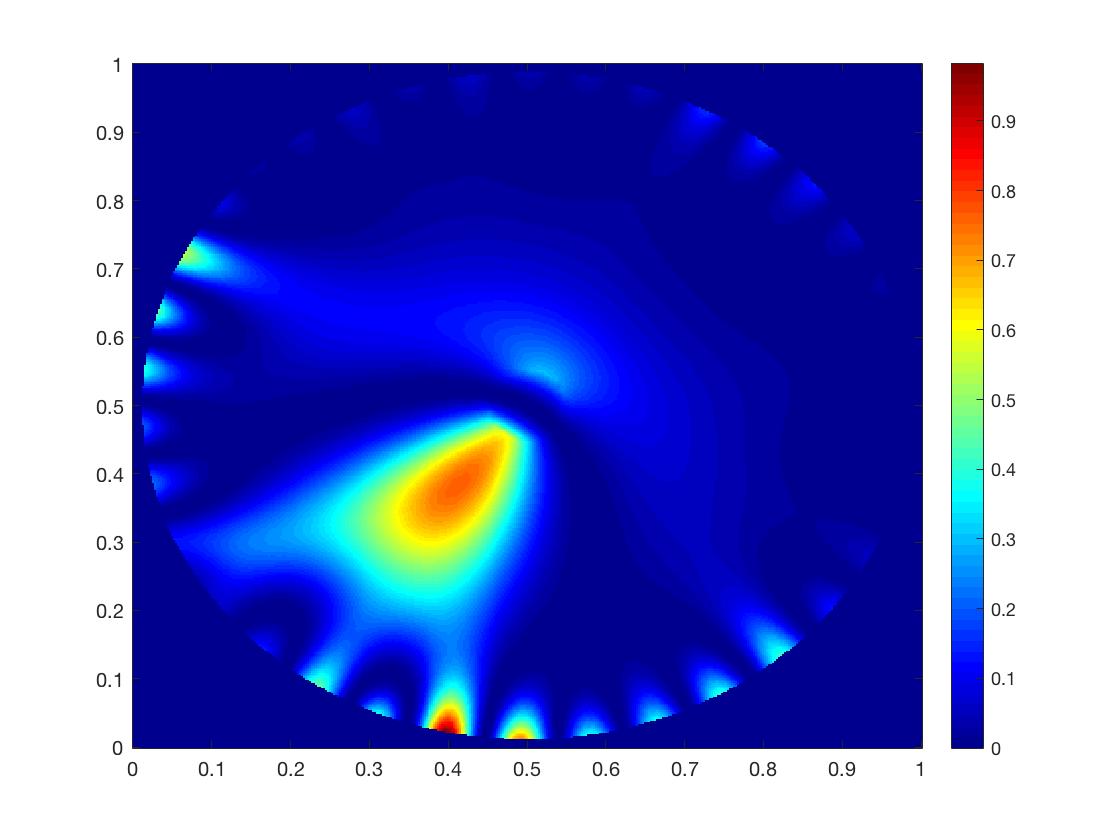}}
\subfigure[ index $\Phi|_D$]{
\label{Ex2sig.sub.7}
\includegraphics[width=0.23\textwidth]{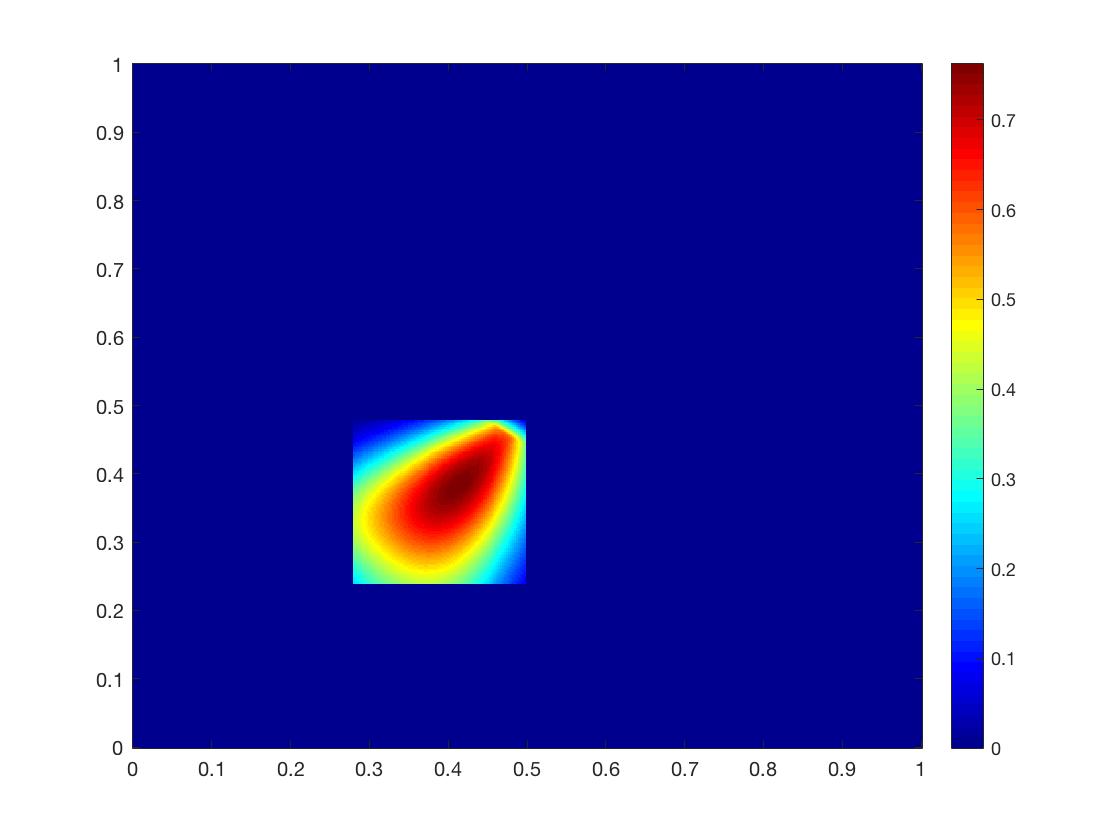}}
\subfigure[ least-squares]{
\label{Ex2sig.sub.8}
\includegraphics[width=0.23\textwidth]{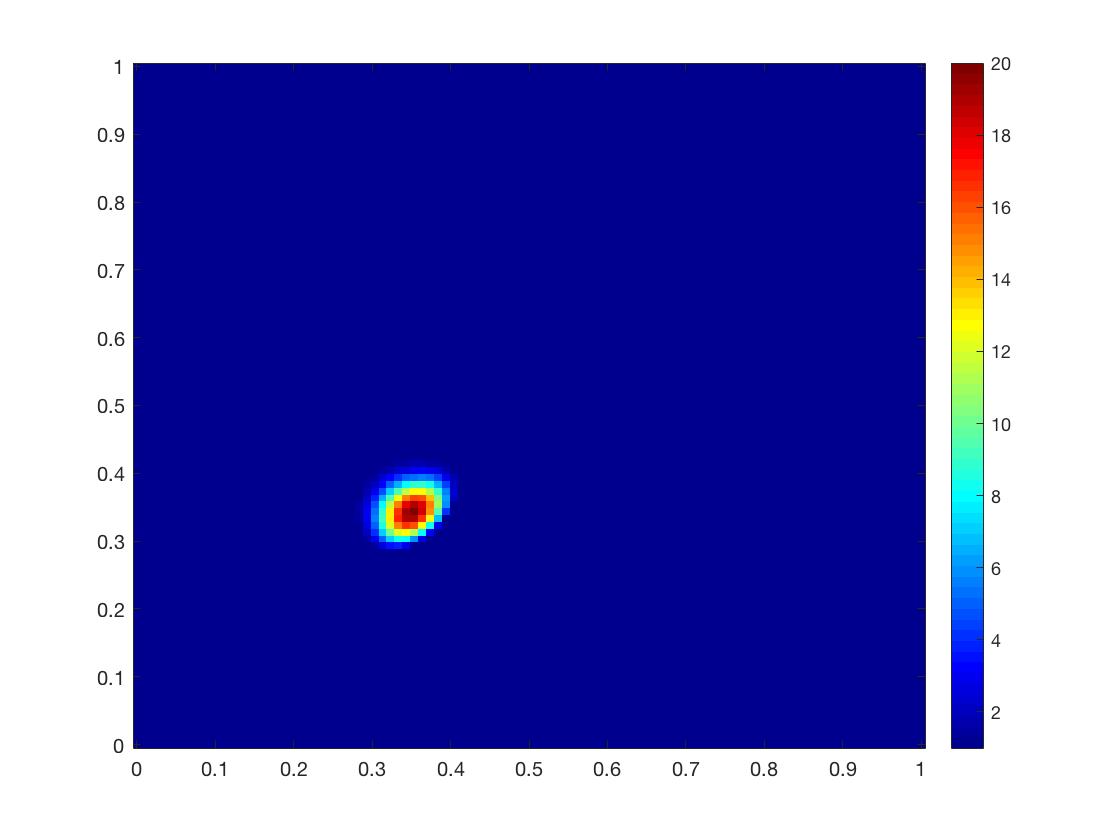}}
\caption{Numerical results for Example \ref{example1} with noise: (a) true $\sigma$, (b) index $\Phi$ from  DSM, (c) index $\Phi|_D$ (index function constrained to the chosen subdomain $D$), (d) least-squares reconstruction. (e), (f), (g), (h) are corresponding graphs for $\mu$.  These two rows are derived using data with $10\%$ noise.}
\label{Ex1.3_2.main}
\end{figure}
\begin{exam}\label{example2.1}
 We implement the algorithm to reconstruct the discontinuous diffusion coefficient $\sigma$ with two inclusions. We assume that $\mu$ is known and equal to the background coefficient $\mu_0$ in this example. One set of data is measured to locate two inclusions of $\sigma$, which are centered at $(0.15,0.5)$ and $(0.5,0.85)$ respectively and of width $0.05$.  $\sigma$ is taken to be $20$ inside the regions as shown in Fig.~\ref{Ex3sig.sub.1}. 
  \end{exam}
  The reconstructions for $\sigma$ using exact measurements and measurements with $20\%$ noise are shown in Fig.~\ref{Ex1.2_1.main}. In the first stage, the DSM gives the index function that separates these two inclusions well using only one set of data as in Fig.~\ref{Ex3sig.sub.2}, while some inhomogeneity is observed in the background. This phenomenon comes from the ill-posedness of inverse medium problems and also the oscillation of fundamental solutions used in the DSM. Even so, the approximations in Fig.~\ref{Ex3sig.sub.2} and \ref{Ex3sig.sub.6} still provide the basic modes of the inhomogeneity. We can identify that there are two inclusions, and capture the subdomain $D$ for the second stage in Figs.~\ref{Ex3sig.sub.3} and \ref{Ex3sig.sub.7}. The least-squares formulation with mixed regularization significantly improves the reconstruction: the locations of both inclusions are captured better with clear background and accurate size in Figs.~\ref{Ex3sig.sub.4}, \ref{Ex3sig.sub.8}. Comparing with Figs.~\ref{Ex3sig.sub.4} and \ref{Ex3sig.sub.8}, one can observe that the reconstruction deteriorates only slightly in that the left inclusion shrinks a little bit when the noise level $\epsilon$ increases from 0 to $20\%$. This example  verifies that the two-stage algorithm is very robust with respect to the data noise.

\begin{figure}[ht!]
\centering
\subfigure[ true $\sigma$]{
\label{Ex3sig.sub.1}
\includegraphics[width=0.23\textwidth]{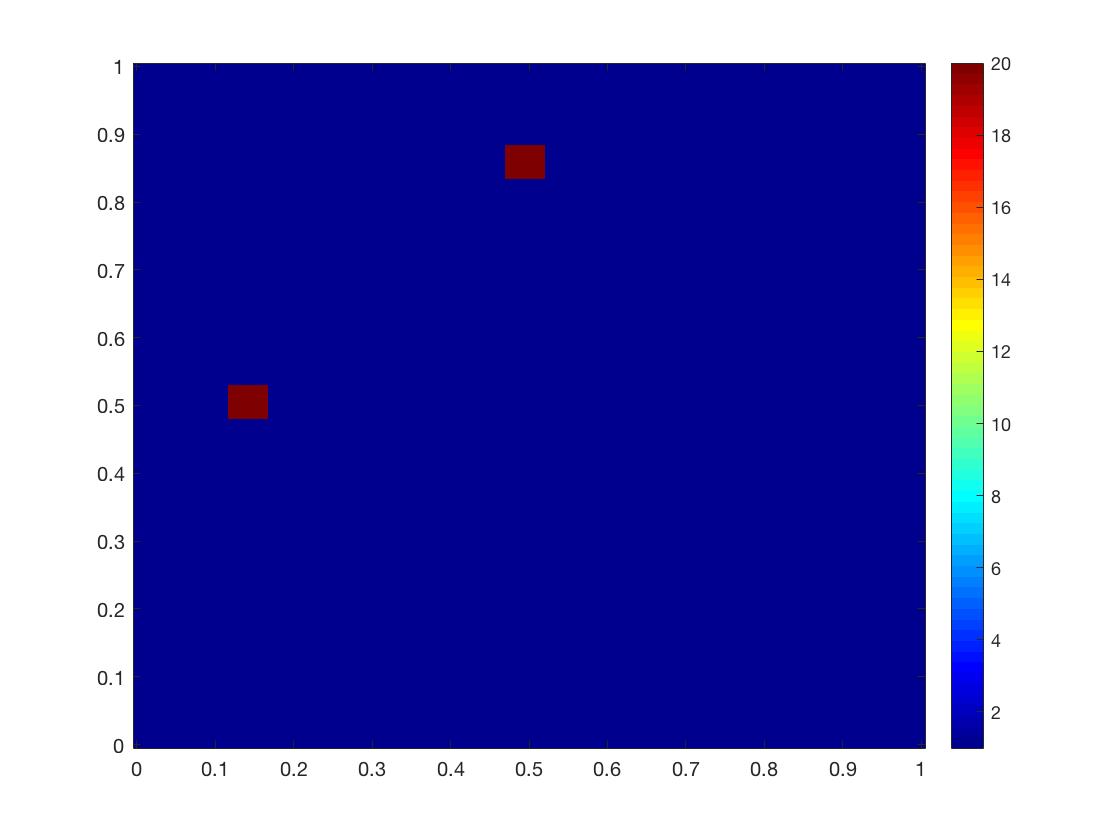}}
\subfigure[ index $\Phi$]{
\label{Ex3sig.sub.2}
\includegraphics[width=0.23\textwidth]{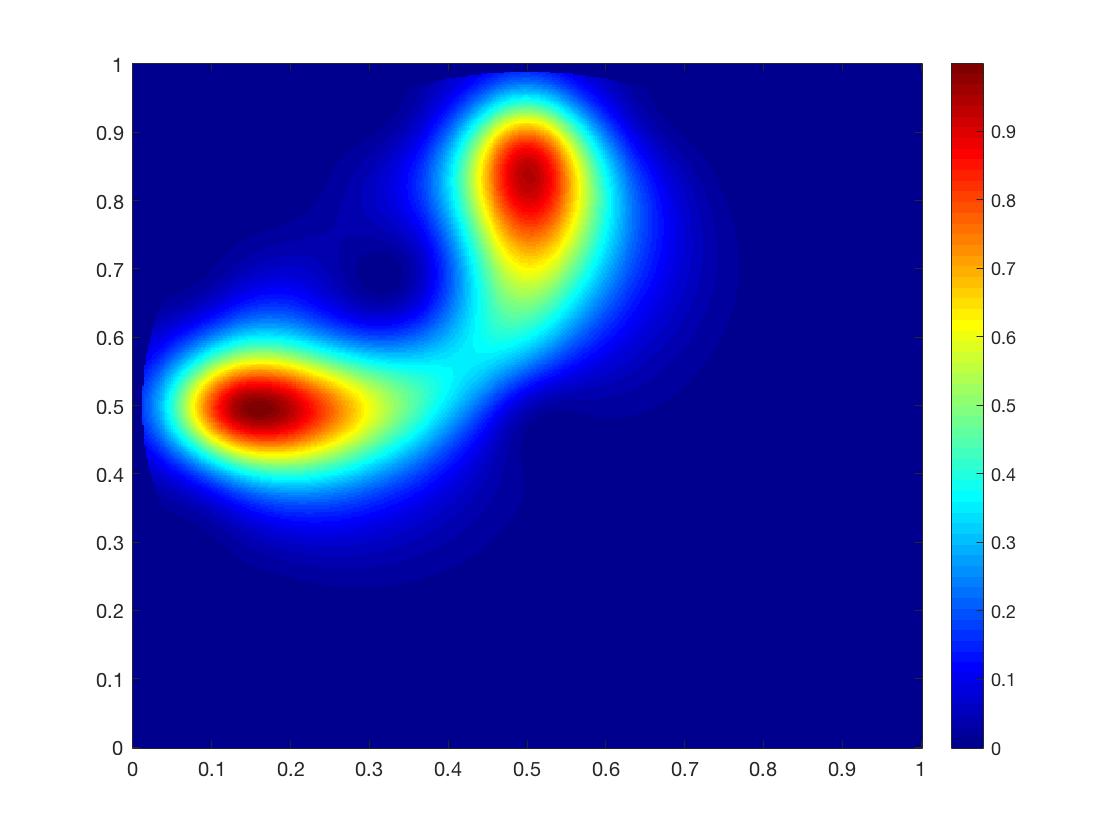}}
\subfigure[ index $\Phi|_D$ ]{
\label{Ex3sig.sub.3}
\includegraphics[width=0.23\textwidth]{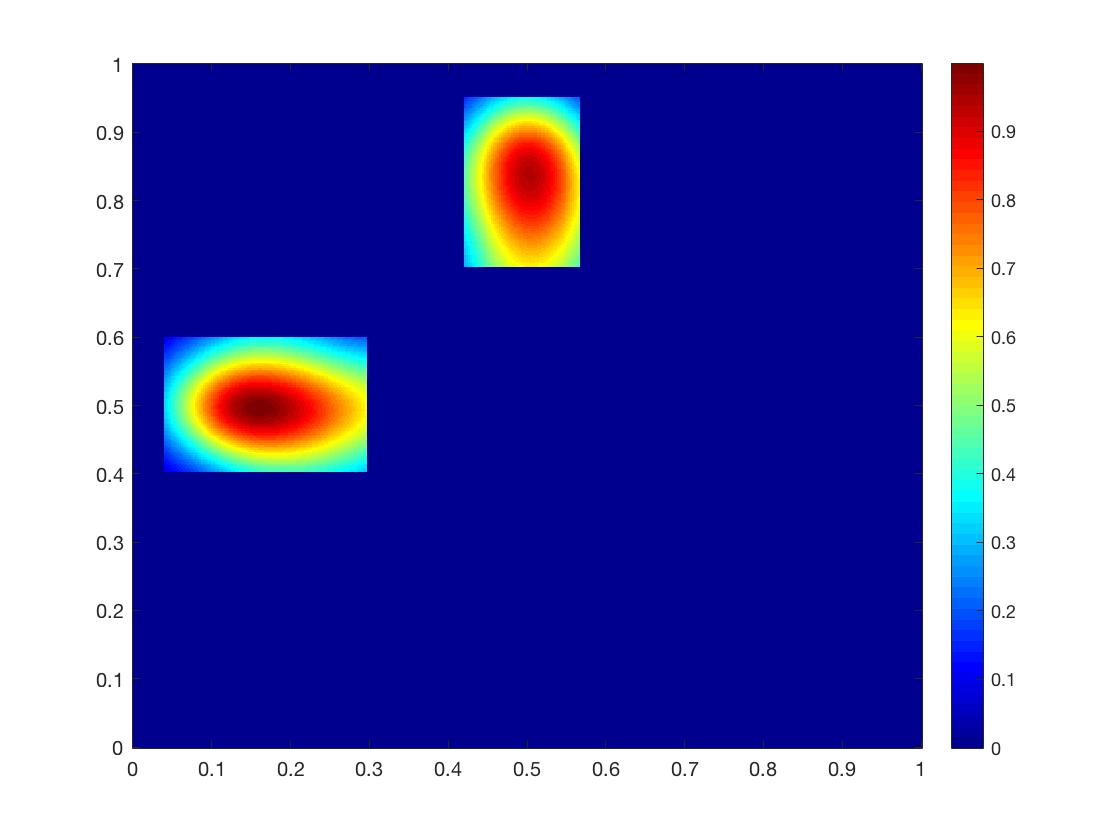}}
\subfigure[least-squares]{
\label{Ex3sig.sub.4}
\includegraphics[width=0.23\textwidth]{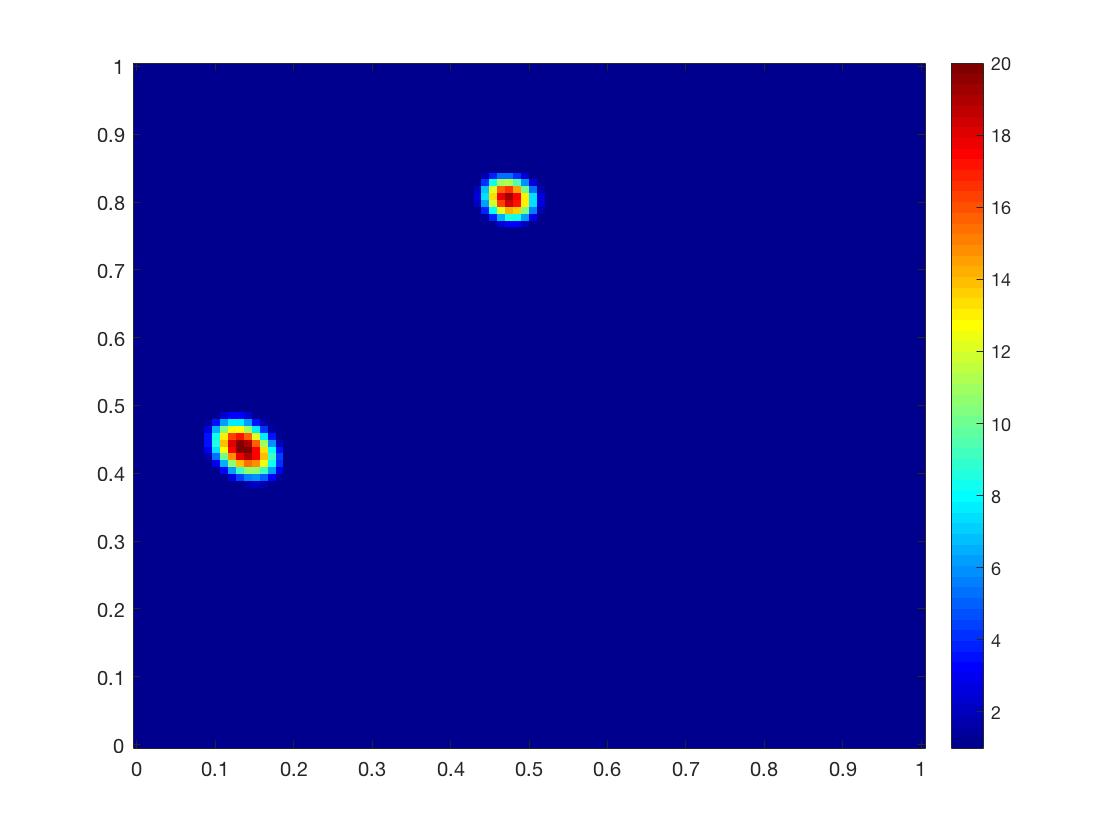}}
\subfigure[ true $\sigma$]{
\label{Ex3sig.sub.5}
\includegraphics[width=0.23\textwidth]{figure/sigmatwoinclusiontrue.jpg}}
\subfigure[index $\Phi$]{
\label{Ex3sig.sub.6}
\includegraphics[width=0.23\textwidth]{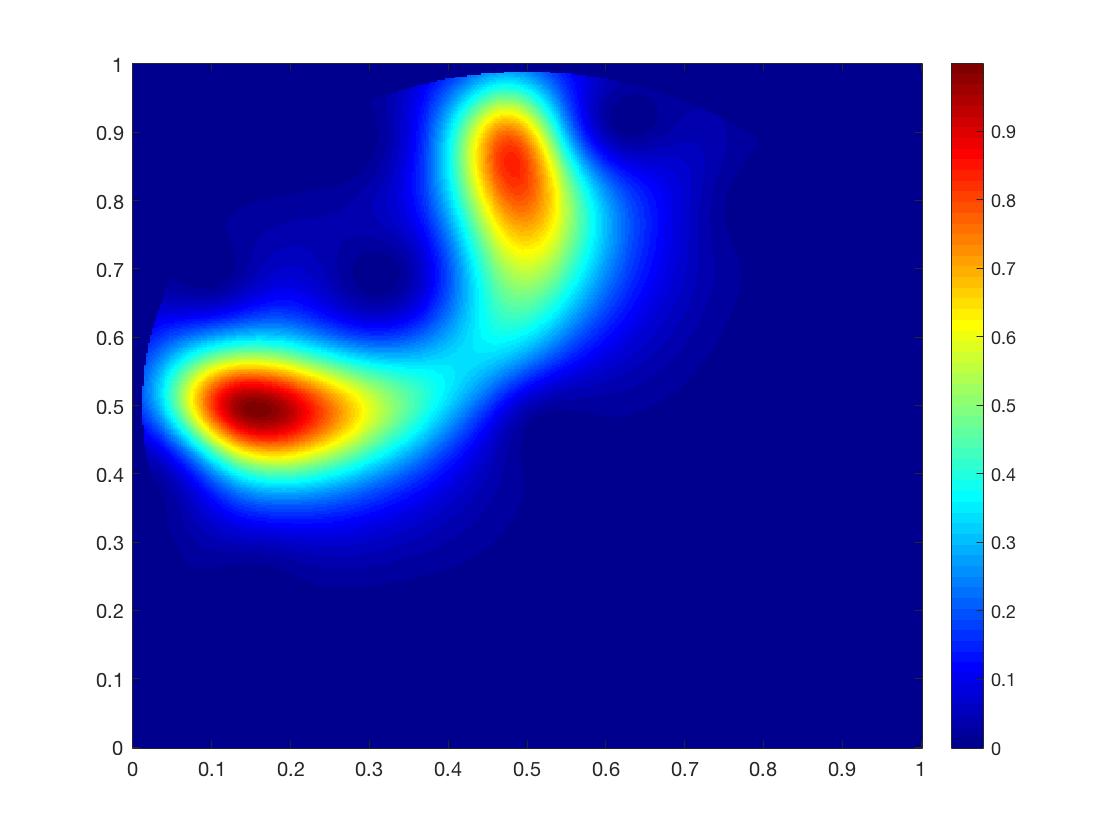}}
\subfigure[ index $\Phi|_D$]{
\label{Ex3sig.sub.7}
\includegraphics[width=0.23\textwidth]{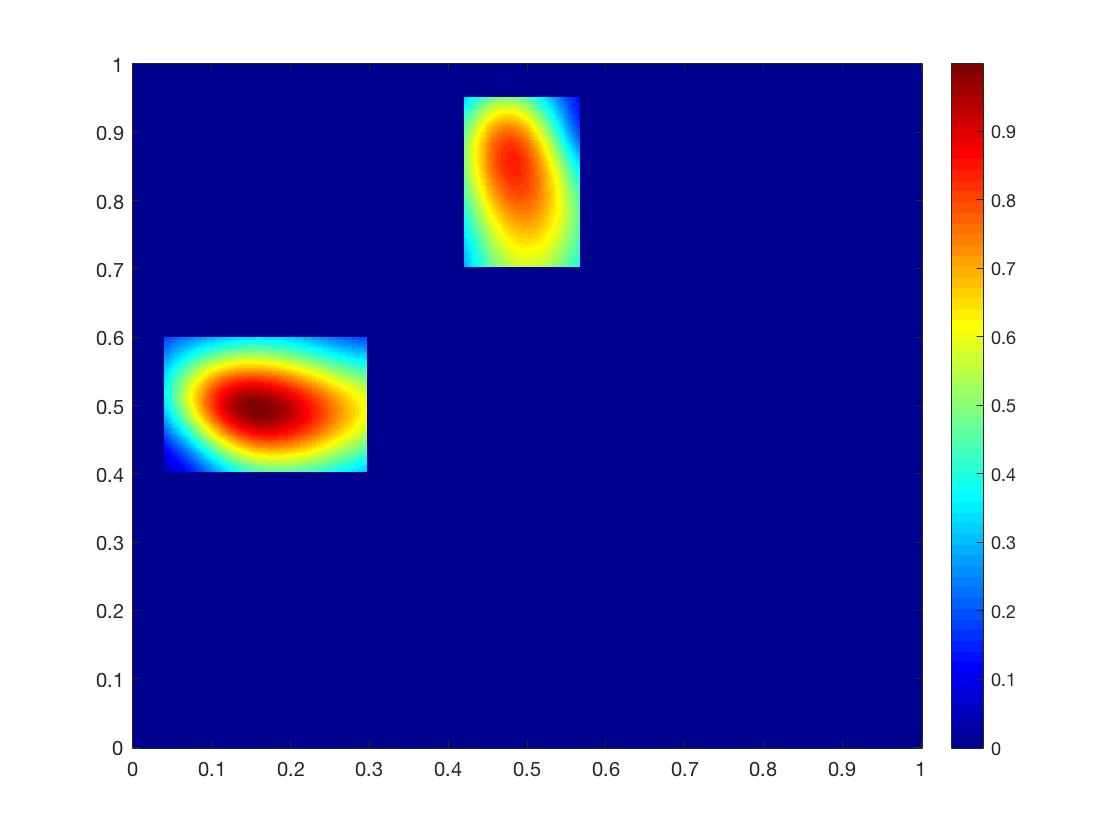}}
\subfigure[ least-squares]{
\label{Ex3sig.sub.8}
\includegraphics[width=0.23\textwidth]{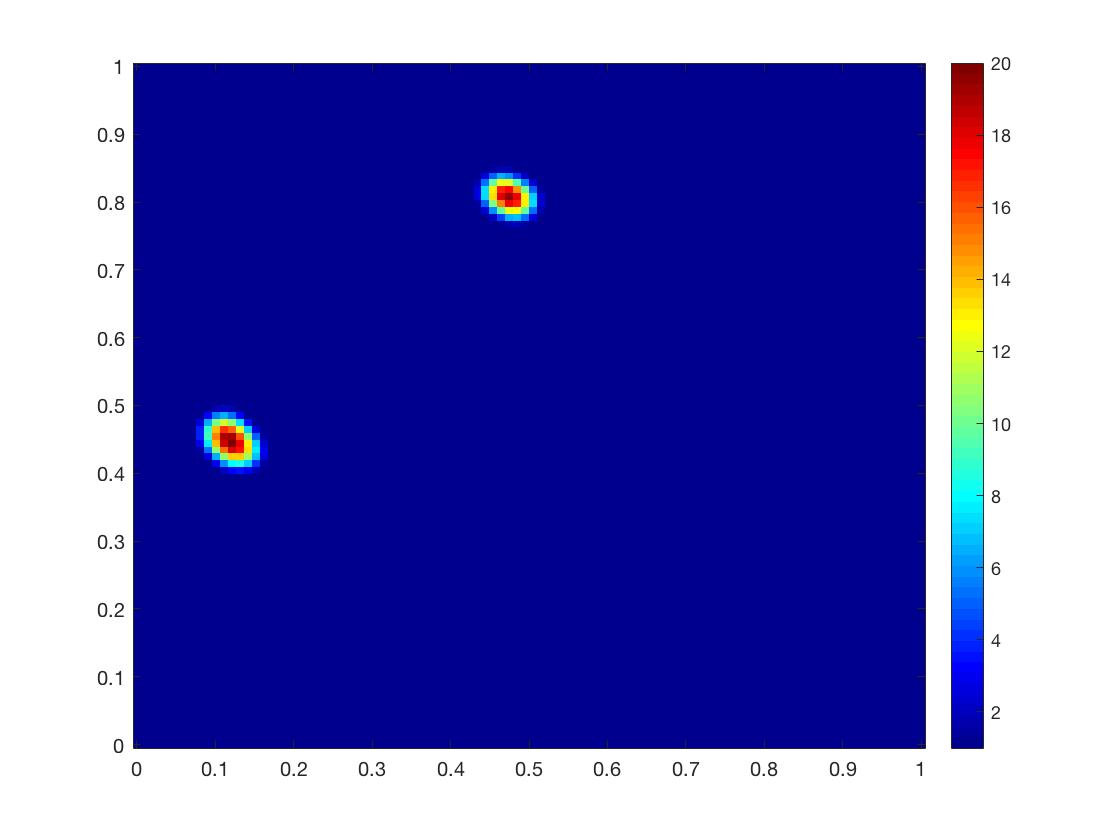}}
\caption{Numerical results for Example \ref{example2.1}: (a) true $\sigma$,  (b) index $\Phi$ from DSM, (c)index $\Phi|_D$ (index function constrained to the chosen subdomain $D$), (d) least-squares reconstruction. (e), (f), (g), (h) are corresponding graphs using data with $20\%$  noise. }
\label{Ex1.2_1.main}
\end{figure}

\begin{exam}\label{example2.2} 
In this example we reconstruct two inclusions of  diffusion coefficient $\sigma$ that stay very close to each other. 
The two inclusions of diffusion coefficient $\sigma$ are centered at $(0.45,0.425)$ and $(0.55,0.575)$ respectively and of width $0.1$ as shown in Fig.~\ref{Ex4sig.sub.1}. The coefficient $\sigma$ is $20$ in both regions.
\end{exam}

The two inclusions in Example \ref{example2.1} are relatively far from each other, while in Example \ref{example2.2} we consider the case when two inclusions are quite close to each other, which is more challenging as it would be difficult to distinguish these two separated inclusions and reconstruct their locations and magnitudes precisely. As shown in Fig.~\ref{Ex4sig.sub.2}, the index function from DSM presents limited information on the diffusion coefficient with only one set of data, and shows one connected inclusion. With $2\%$ noise, the index function is blurred a lot as shown in Fig.~\ref{Ex4sig.sub.6}. In both noisy and noiseless cases, only one subdomain can be detected from this index function from the first stage. 
The second stage still presents well separated reconstructions with the size and magnitude that match the exact diffusion coefficient well as shown in Fig.~\ref{Ex4sig.sub.4}. For the case with $2\%$ noise, this two-stage algorithm also gives satisfying reconstruction  in Fig.~\ref{Ex4sig.sub.8}. Compared with Fig.~\ref{Ex4sig.sub.4},  it is observable that the left inclusion moves towards x-axis and elongates a little bit, while we can still tell the sizes and locations of two inclusions from the reconstruction. 
This shows that the least-squares method provides much more details than DSM and is relatively robust with respect the the noise in the measurement. 

Note that for the noise level higher than $2\%$, the DSM can not provide a feasible initial guess for the least-squares method in the second stage with only one set of data, but with an initial guess that reflects the mode of the true coefficient, the least-squares method in the second stage has great tolerance of noise as shown in other examples.

\begin{figure}[ht!]
\centering
\subfigure[ true $\sigma$]{
\label{Ex4sig.sub.1}
\includegraphics[width=0.23\textwidth]{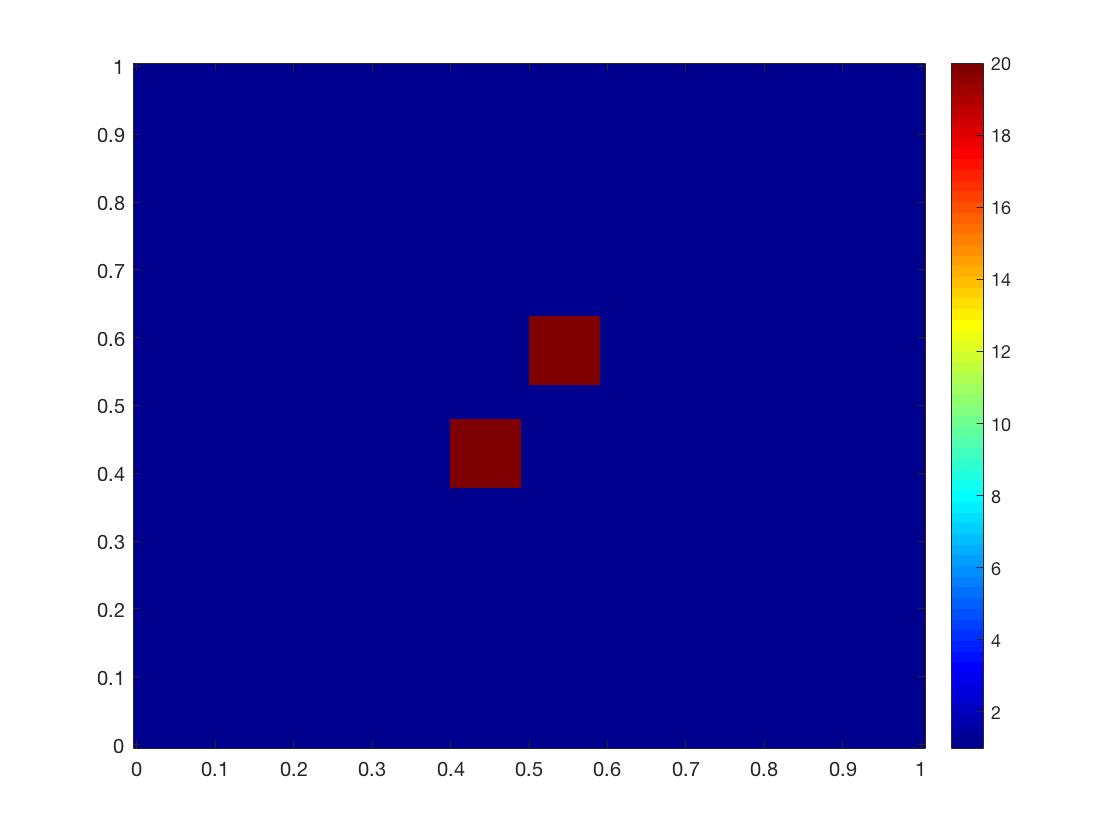}}
\subfigure[  index function $\Phi$]{
\label{Ex4sig.sub.2}
\includegraphics[width=0.23\textwidth]{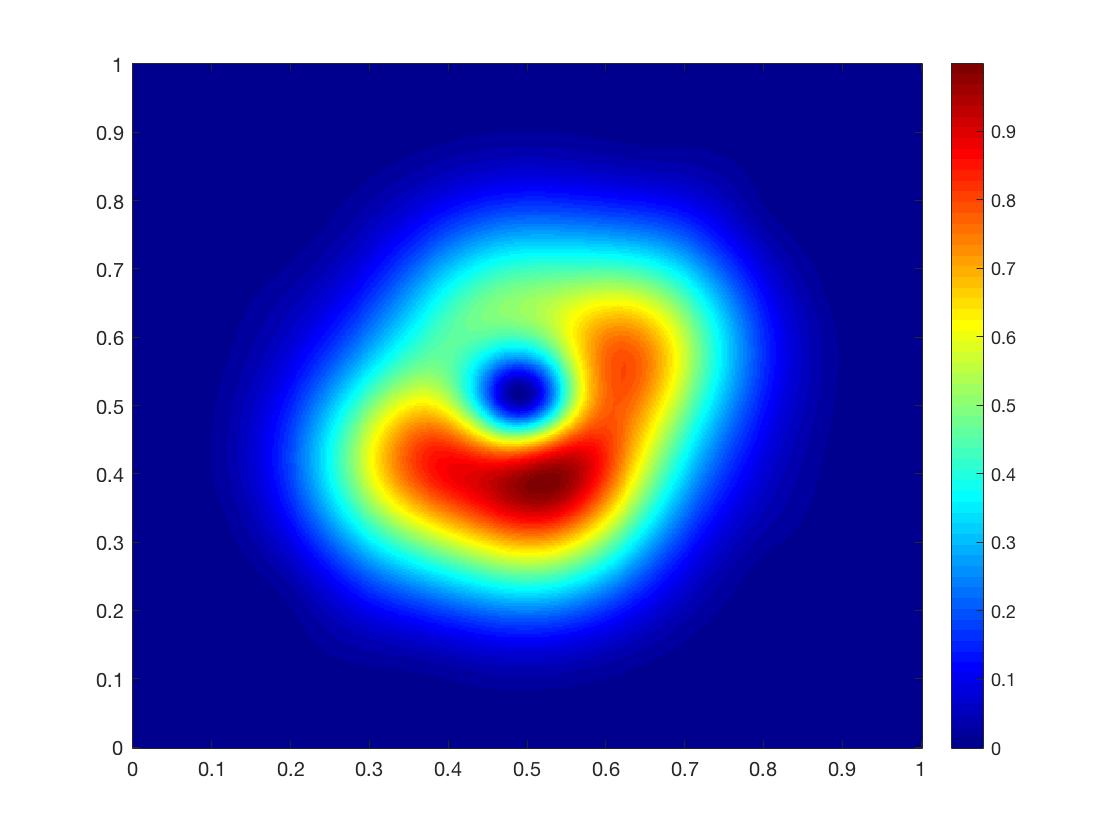}}
\subfigure[ index function $\Phi|_D$]{
\label{Ex4sig.sub.3}
\includegraphics[width=0.23\textwidth]{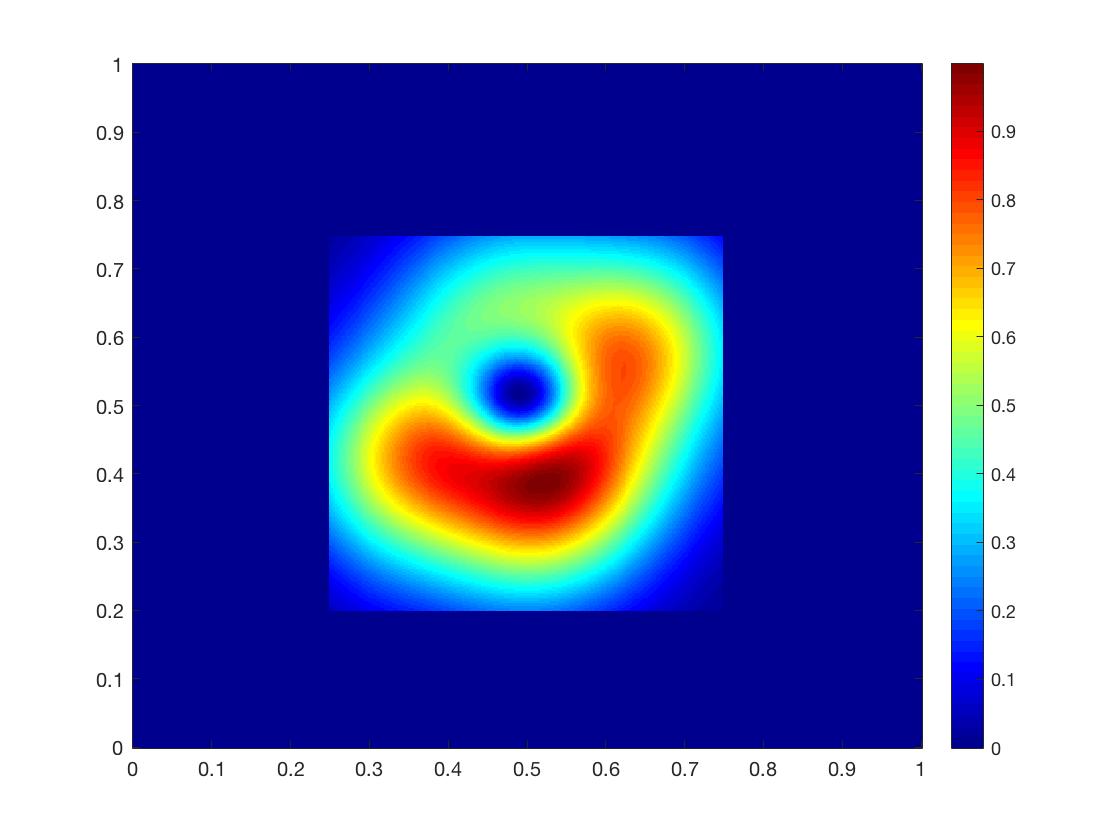}}
\subfigure[ least-squares]{
\label{Ex4sig.sub.4}
\includegraphics[width=0.23\textwidth]{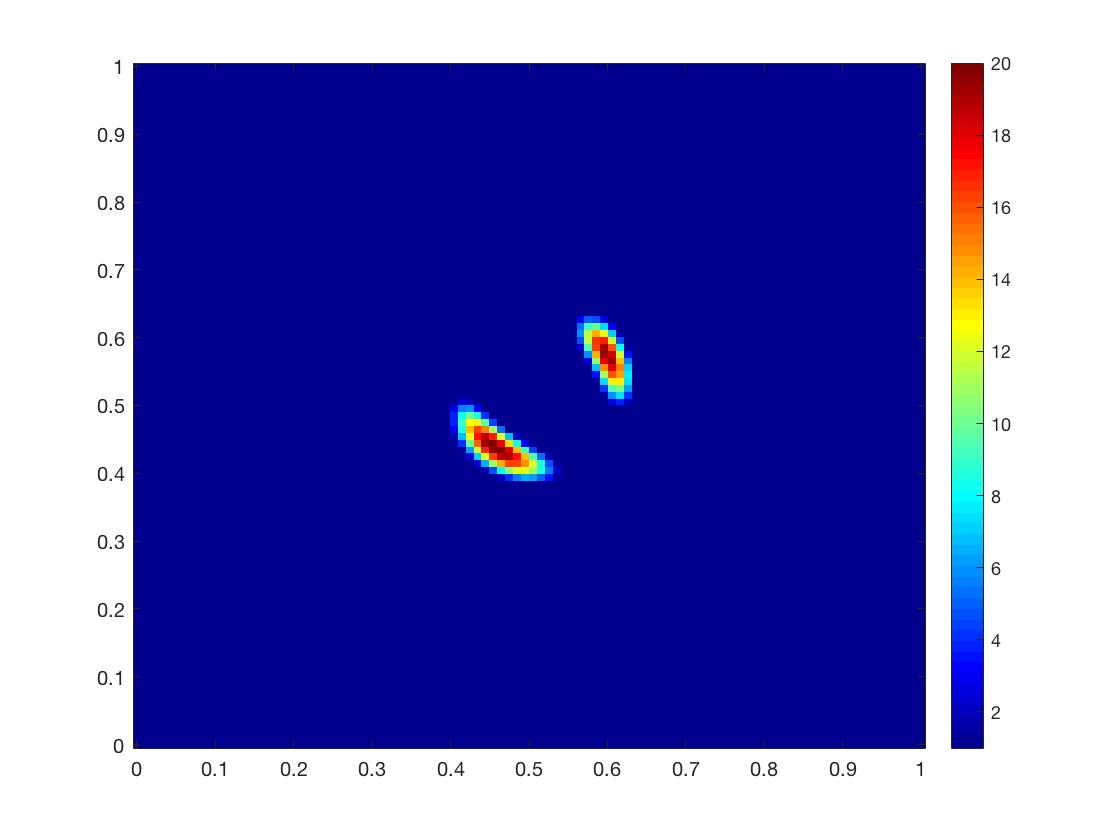}}
\subfigure[ true $\sigma$]{
\label{Ex4sig.sub.5}
\includegraphics[width=0.23\textwidth]{figure/sigmatrueclose.jpg}}
\subfigure[  index function $\Phi$]{
\label{Ex4sig.sub.6}
\includegraphics[width=0.23\textwidth]{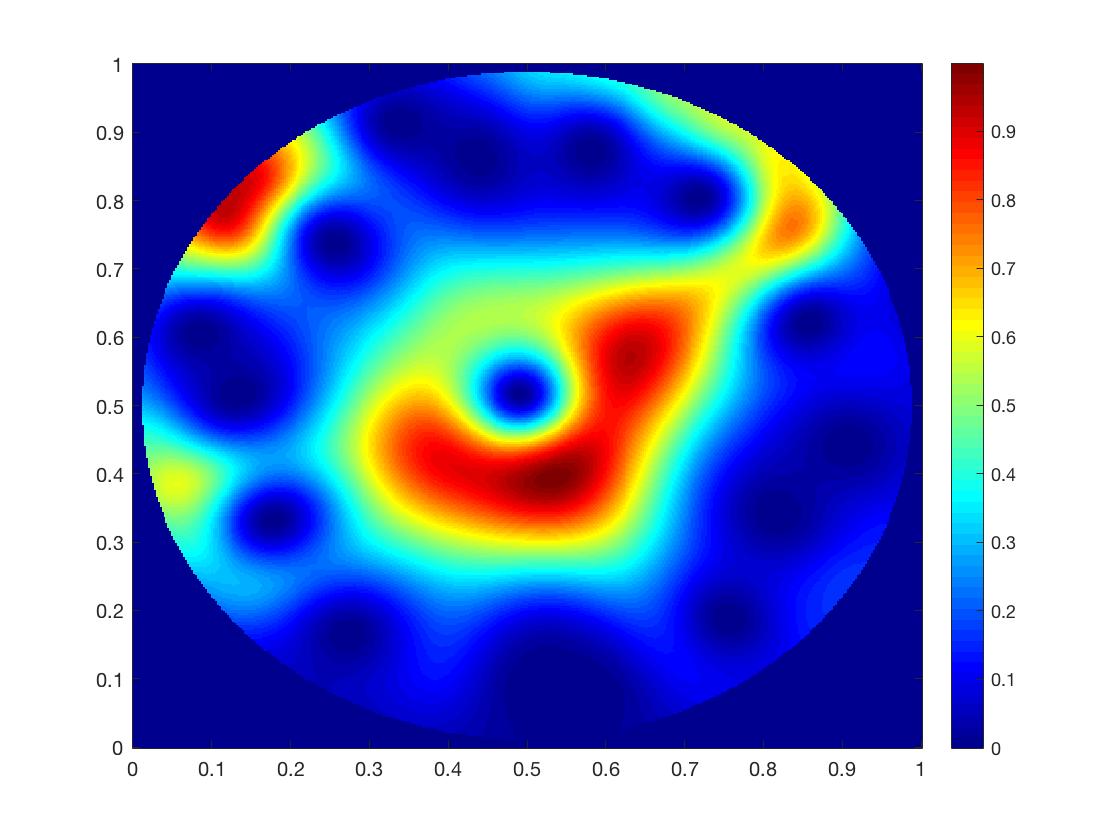}}
\subfigure[ index function $\Phi|_D$]{
\label{Ex4sig.sub.7}
\includegraphics[width=0.23\textwidth]{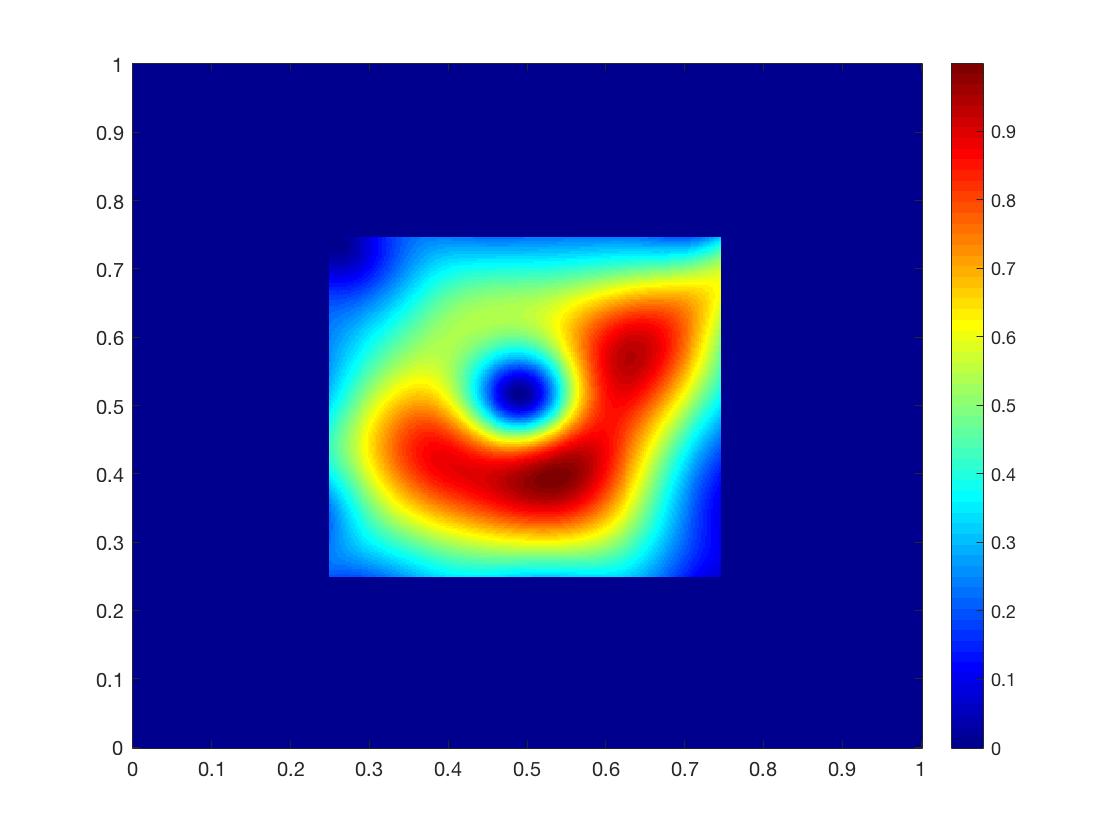}}
\subfigure[ least-squares]{
\label{Ex4sig.sub.8}
\includegraphics[width=0.23\textwidth]{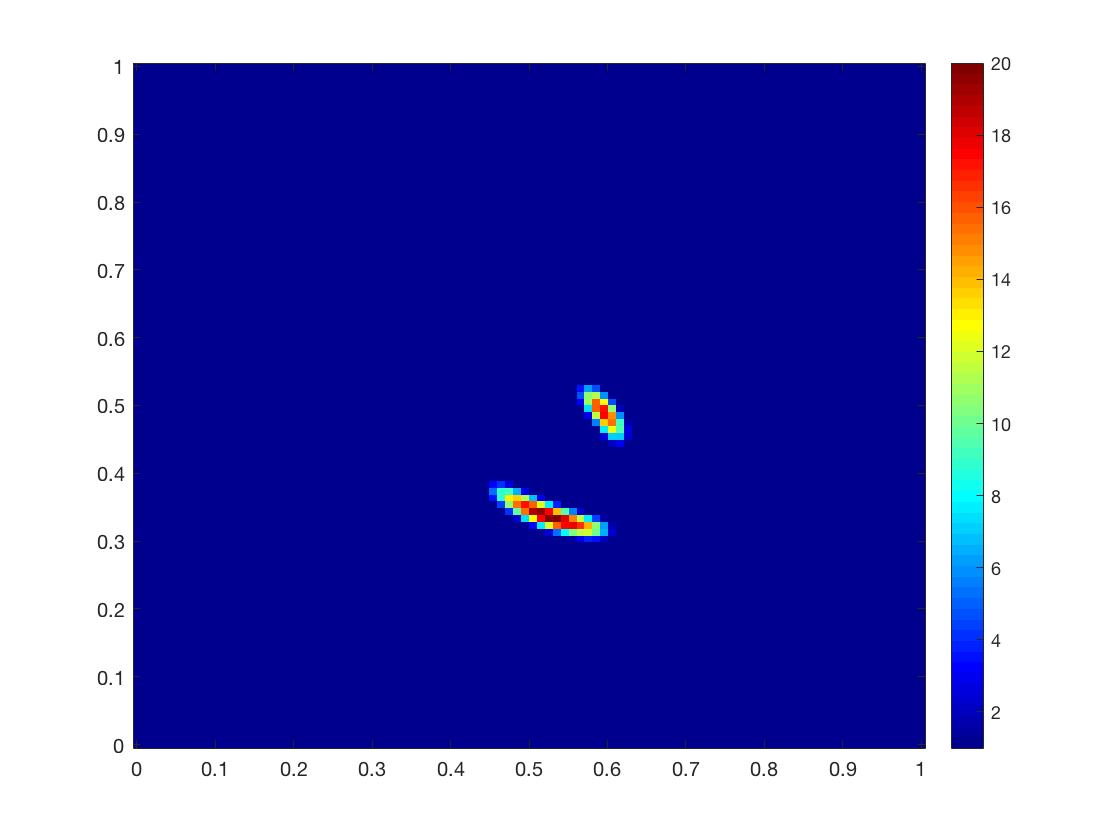}}
\caption{Numerical results for Example \ref{example2.2}: (a) true $\sigma$, (b) index $\Phi$ from DSM, (c) index $\Phi|_D$ (index function constrained to the chosen subdomain $D$), (d) least-squares reconstruction. These experiments are derived using exact data. (e), (f), (g), (h) are corresponding graphs using data with $2\%$ noise}
\label{Ex2.2}
\end{figure}

\begin{exam}\label{example3} 
With this example, we reconstruct $\sigma$ and $\mu$ simultaneously, with two inclusions for each coefficient. The inclusions are in the following scenario: The inclusions of diffusion coefficient $\sigma$ are of width $0.1$ and centered at $(0.5,0.25)$, $(0.5,0.75)$ respectively, and the magnitude inside the region is $20$; The inclusions of absorption coefficient $\mu$ are of width $0.1$ and centered at $(0.25,0.5)$, $(0.75,0.5)$ respectively, and the magnitude inside the region is $20$, as shown in Figs.~\ref{Ex5sig.sub.1} and \ref{Ex5sig.sub.5}.
\end{exam}
 In Example \ref{example1}, we have one inclusion for coefficients $\mu$ and $\sigma$ each, and it is shown that this algorithm can reconstruct both medium coefficients well. Example \ref{example3} is more challenging than Example \ref{example1}, as the existence of two inclusions for each coefficient will influence the reconstruction of the other coefficient. As shown in Figs.~\ref{Ex5sig.sub.2} and \ref{Ex5sig.sub.6}, the index functions separate these inclusions well  and give a rough approximation for their locations with only one set of data. However, it can be observed that the maximal points of index functions  actually differ from the exact coefficients, and one can not capture the information of inclusions by simply taking square of the index functions. In the second stage of the proposed algorithm, the locations are modified as shown in Figs.~\ref{Ex5sig.sub.4} and \ref{Ex5sig.sub.8}. When we consider the case with $20\%$ noise, DSM provides blurry  approximations as shown in Figs.~\ref{Ex6sig.sub.2} and \ref{Ex6sig.sub.6}.
  In the second stage of reconstruction, Figs.~\ref{Ex6sig.sub.4} and \ref{Ex6sig.sub.8} present results almost the same as Figs.~\ref{Ex5sig.sub.4} and \ref{Ex5sig.sub.8}, which once again prove the robust performance of the least-squares method. 
\begin{figure}[ht!]
\centering
\subfigure[ true $\sigma$]{
\label{Ex5sig.sub.1}
\includegraphics[width=0.23\textwidth]{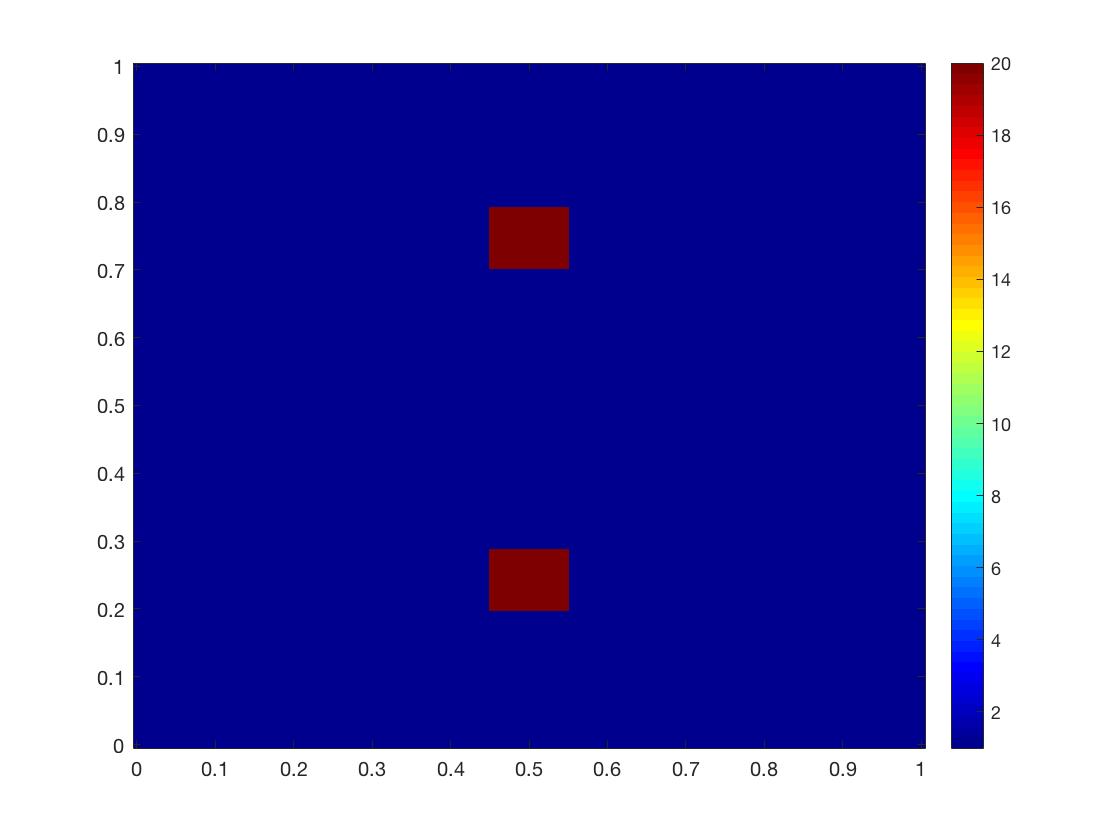}}
\subfigure[ index function $\Phi$]{
\label{Ex5sig.sub.2}
\includegraphics[width=0.23\textwidth]{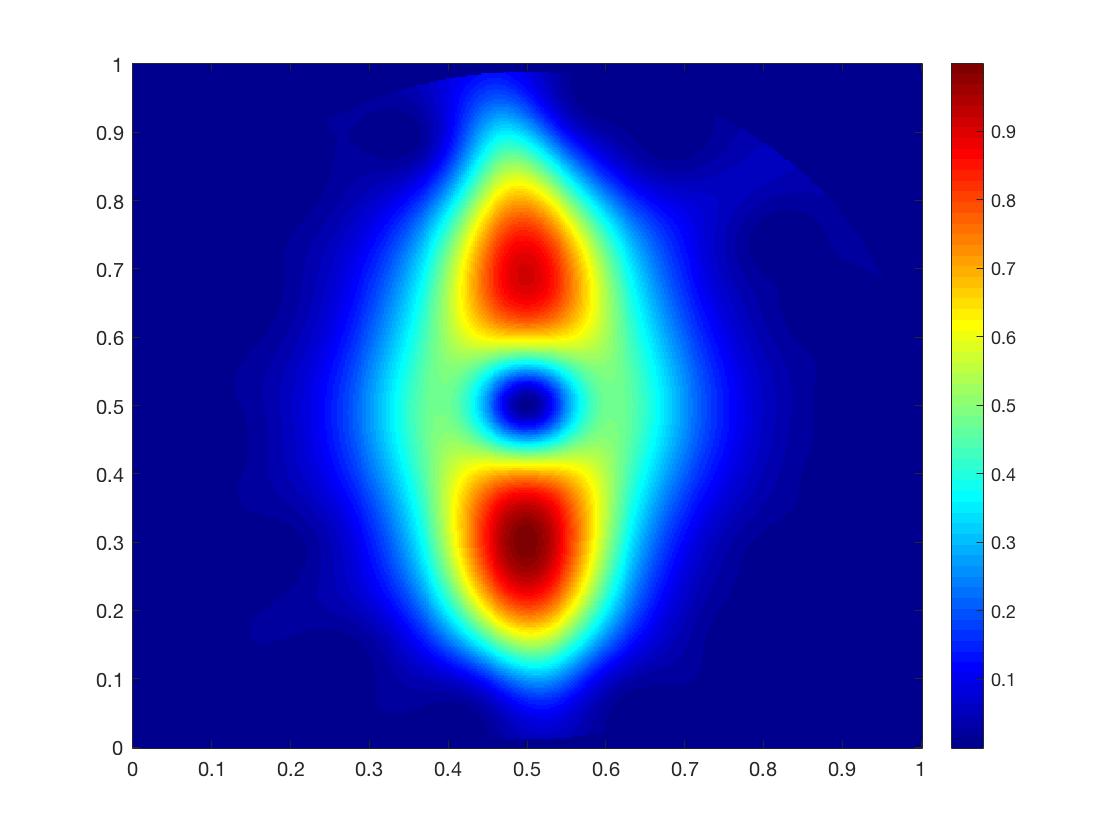}}
\subfigure[ index function $\Phi|_D$]{
\label{Ex5sig.sub.3}
\includegraphics[width=0.23\textwidth]{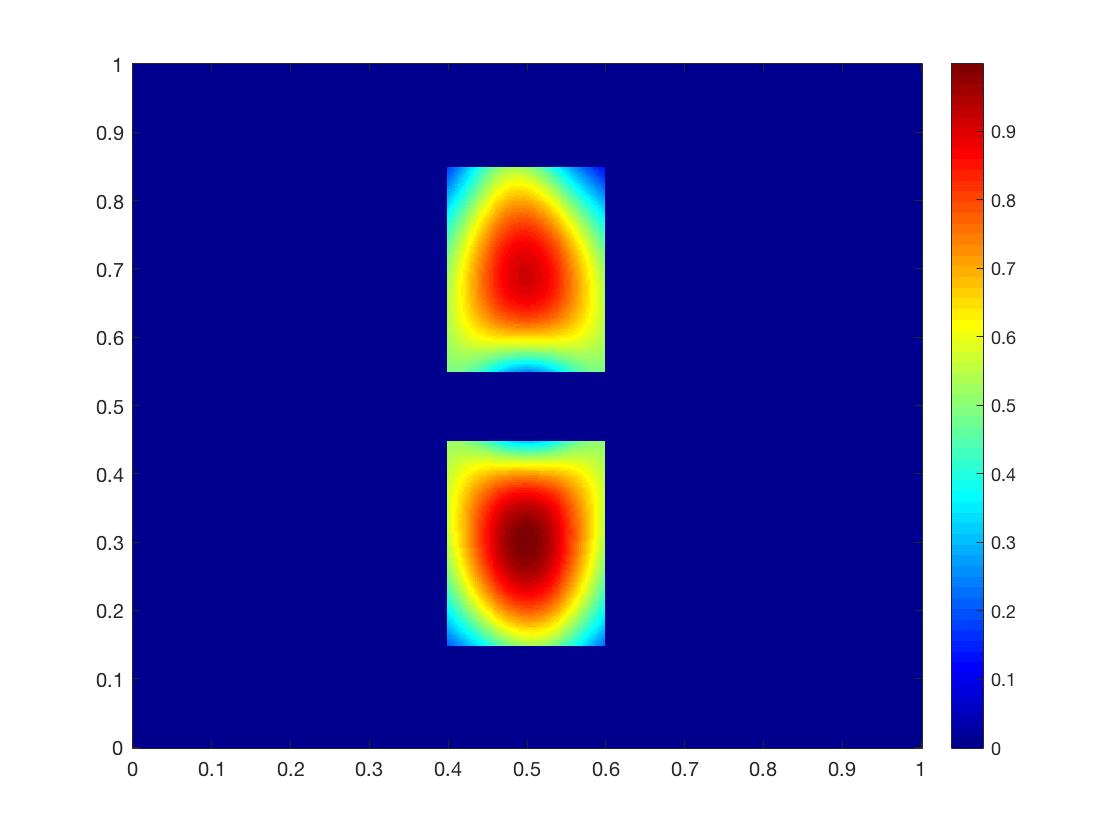}}
\subfigure[ least-squares]{
\label{Ex5sig.sub.4}
\includegraphics[width=0.23\textwidth]{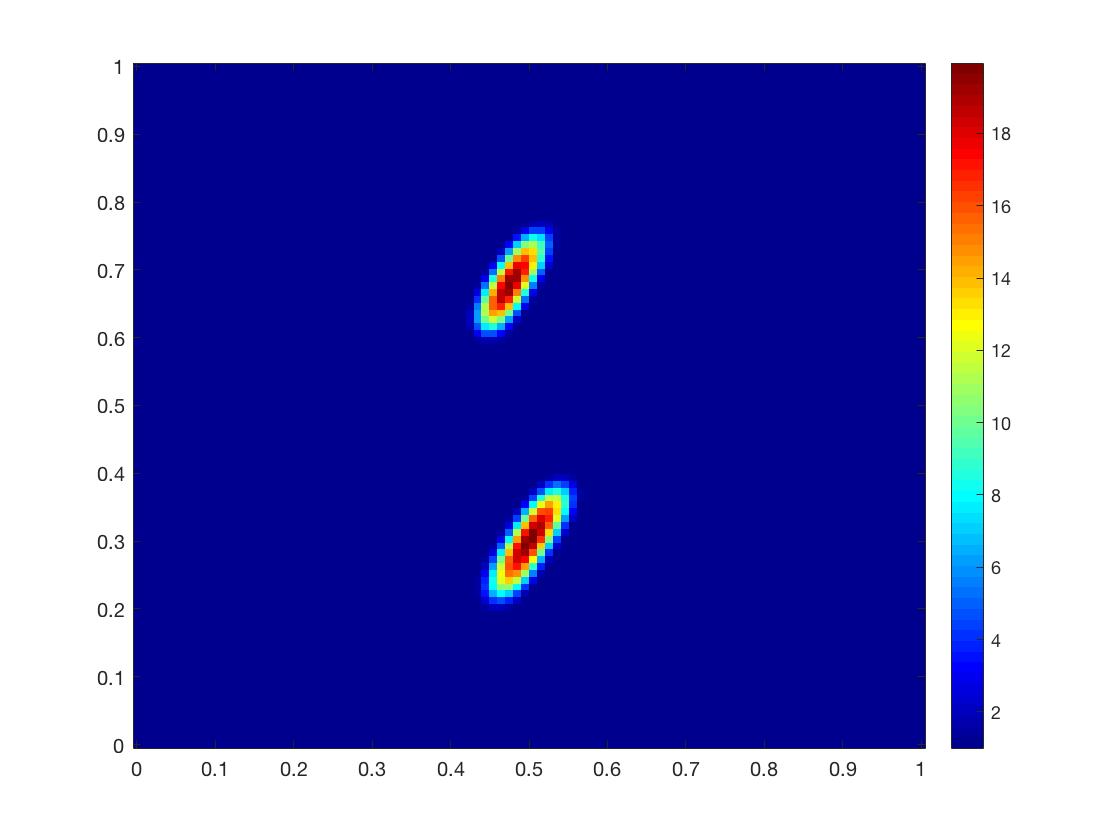}}
\subfigure[ true $\mu$]{
\label{Ex5sig.sub.5}
\includegraphics[width=0.23\textwidth]{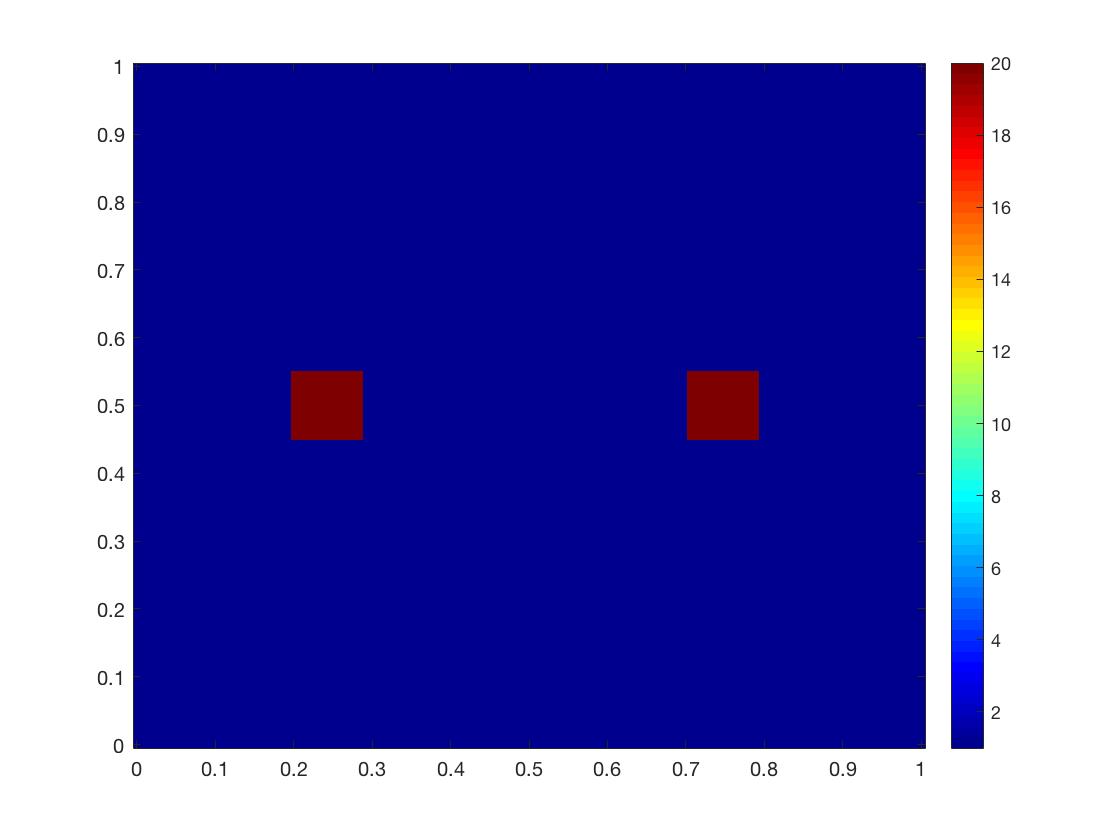}}
\subfigure[ index function $\Phi$ ]{
\label{Ex5sig.sub.6}
\includegraphics[width=0.23\textwidth]{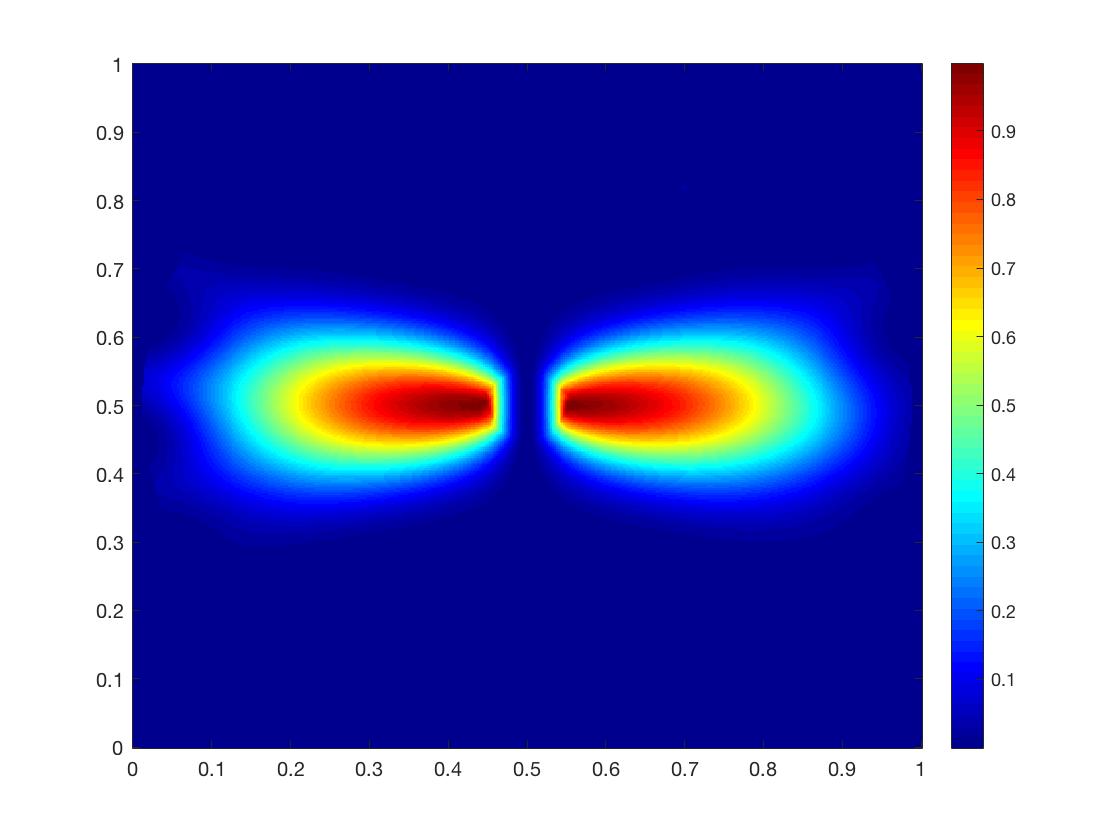}}
\subfigure[ index function $\Phi|_D$]{
\label{Ex5sig.sub.7}
\includegraphics[width=0.23\textwidth]{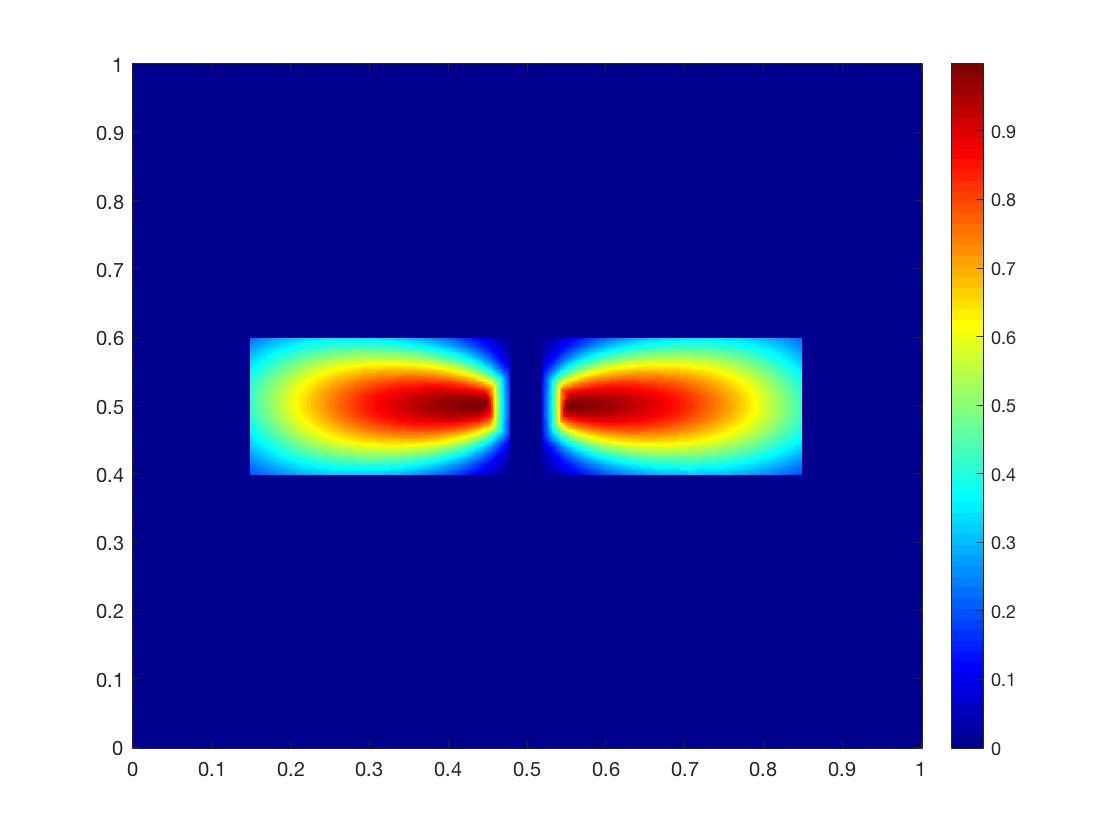}}
\subfigure[ least-squares]{
\label{Ex5sig.sub.8}
\includegraphics[width=0.23\textwidth]{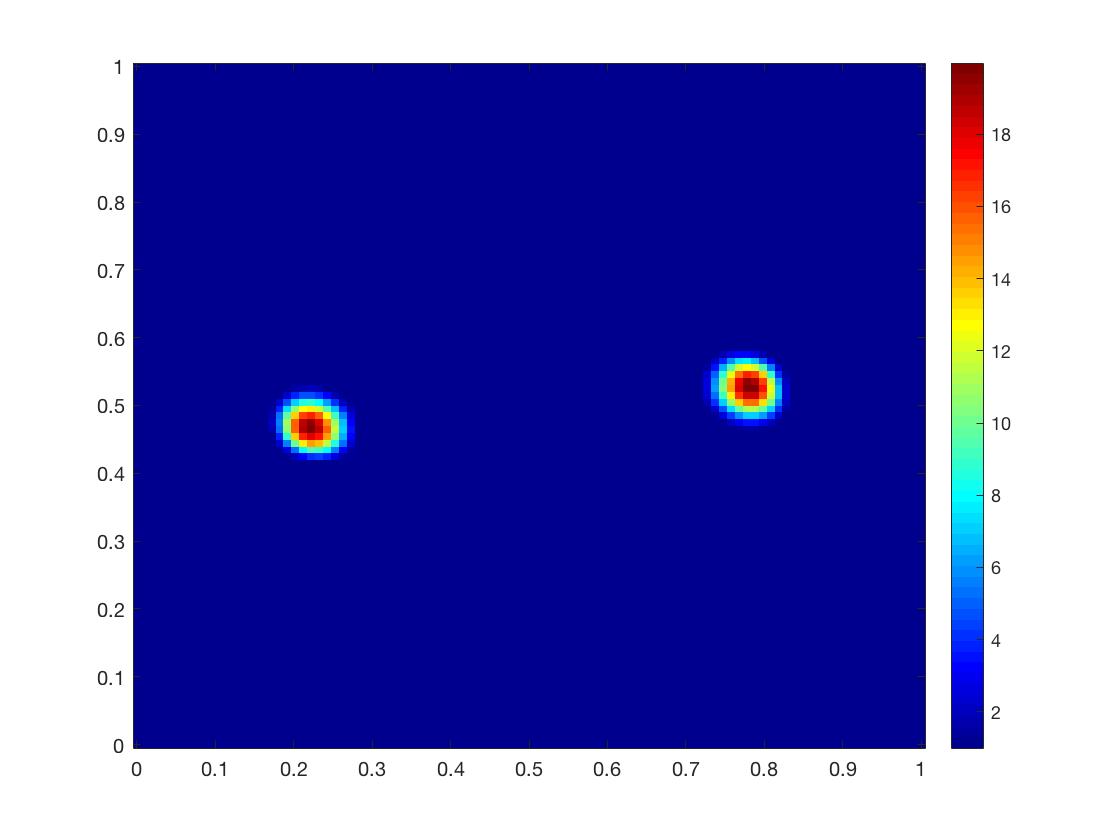}}
\caption{Numerical results for Example \ref{example3}: (a) true $\sigma$, (b) index $\Phi$ for $\sigma$ from DSM, (c) index $\Phi|_D$ for $\sigma$ (index function constrained to the chosen subdomain $D$), (d) least-squares reconstruction. (e), (f), (g), (h) are corresponding graphs for $\mu$. These two rows are derived using exact data.}
\label{Ex126.main}
\end{figure}

\begin{figure}[ht!]
\centering
\subfigure[ true $\sigma$]{
\label{Ex6sig.sub.1}
\includegraphics[width=0.23\textwidth]{figure/eachtruesigma.jpg}}
\subfigure[ index function $\Phi$]{
\label{Ex6sig.sub.2}
\includegraphics[width=0.23\textwidth]{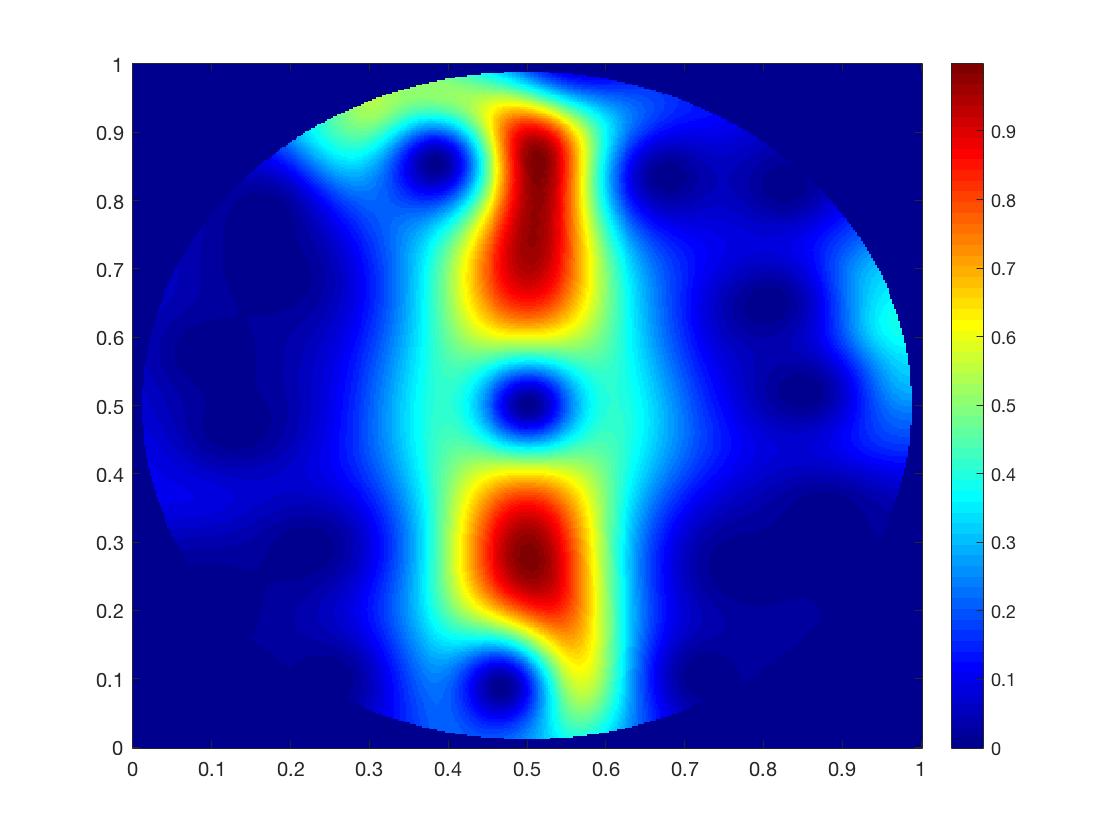}}
\subfigure[ index function $\Phi|_D$]{
\label{Ex6sig.sub.3}
\includegraphics[width=0.23\textwidth]{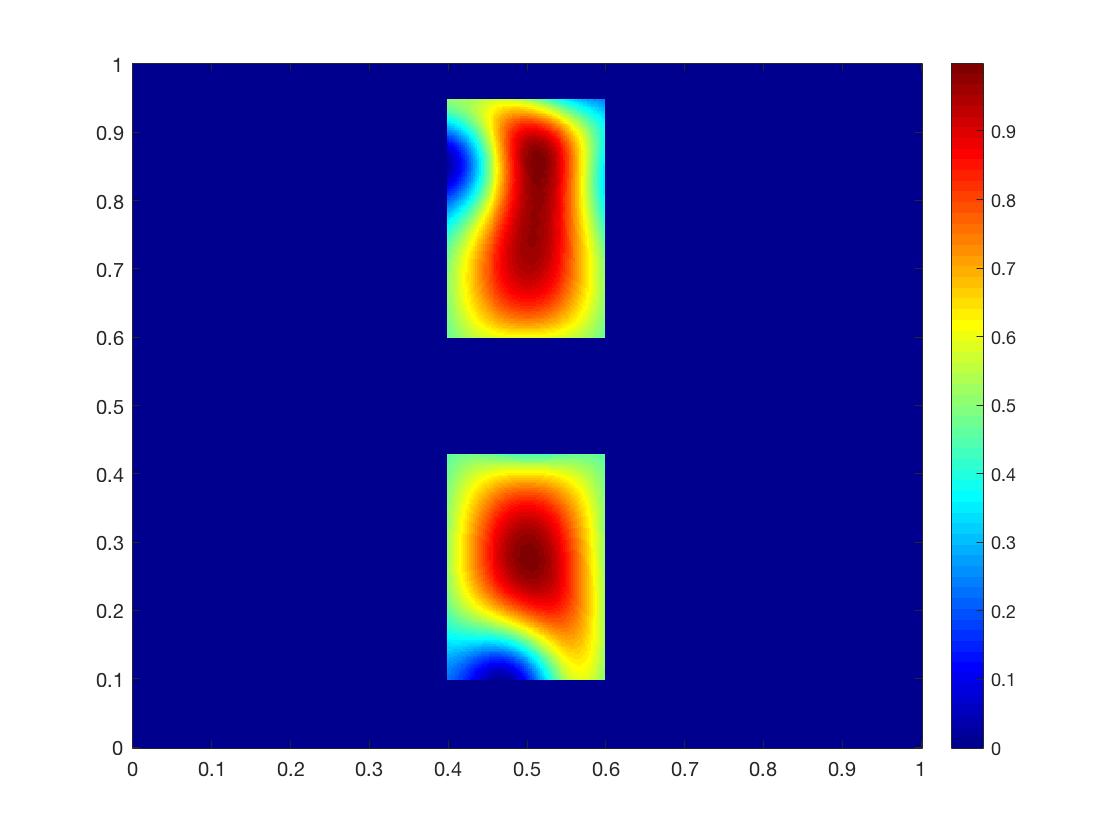}}
\subfigure[ least-squares]{
\label{Ex6sig.sub.4}
\includegraphics[width=0.23\textwidth]{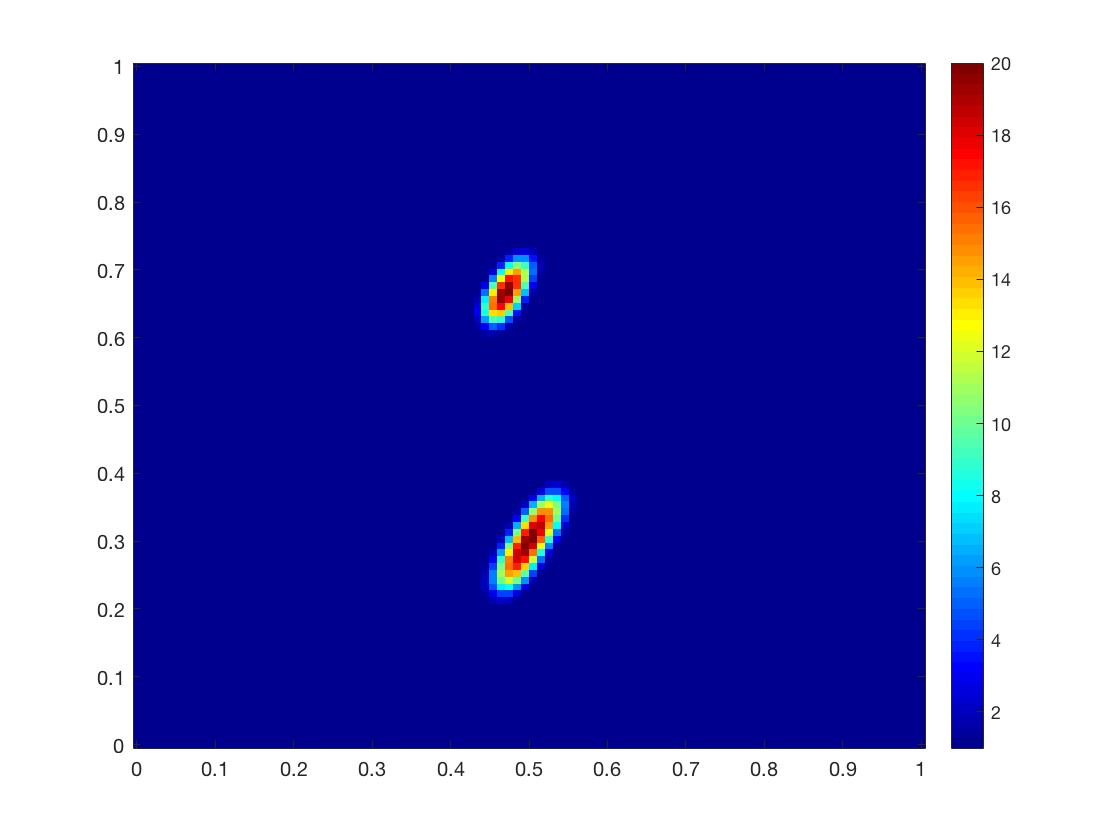}}
\subfigure[ true  $\mu$]{
\label{Ex6sig.sub.5}
\includegraphics[width=0.23\textwidth]{figure/eachtruemu.jpg}}
\subfigure[ index function $\Phi$ ]{
\label{Ex6sig.sub.6}
\includegraphics[width=0.23\textwidth]{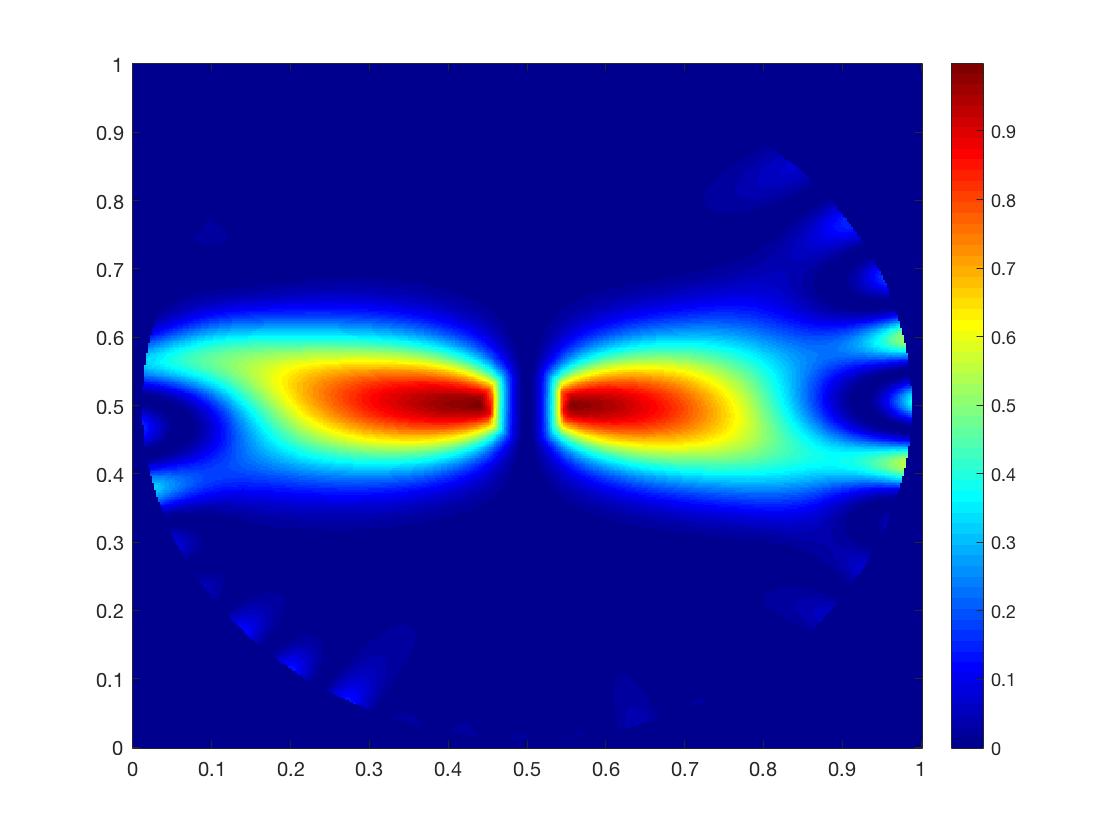}}
\subfigure[ index function $\Phi|_D$]{
\label{Ex6sig.sub.7}
\includegraphics[width=0.23\textwidth]{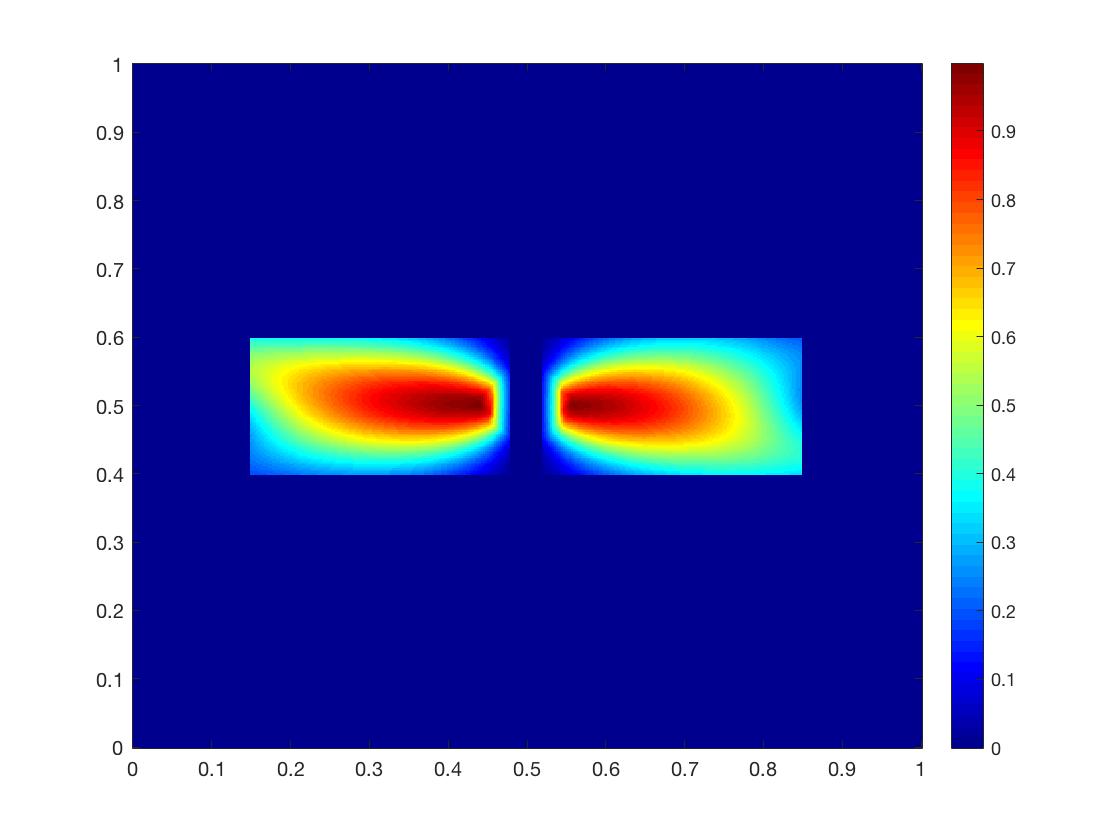}}
\subfigure[ least-squares]{
\label{Ex6sig.sub.8}
\includegraphics[width=0.23\textwidth]{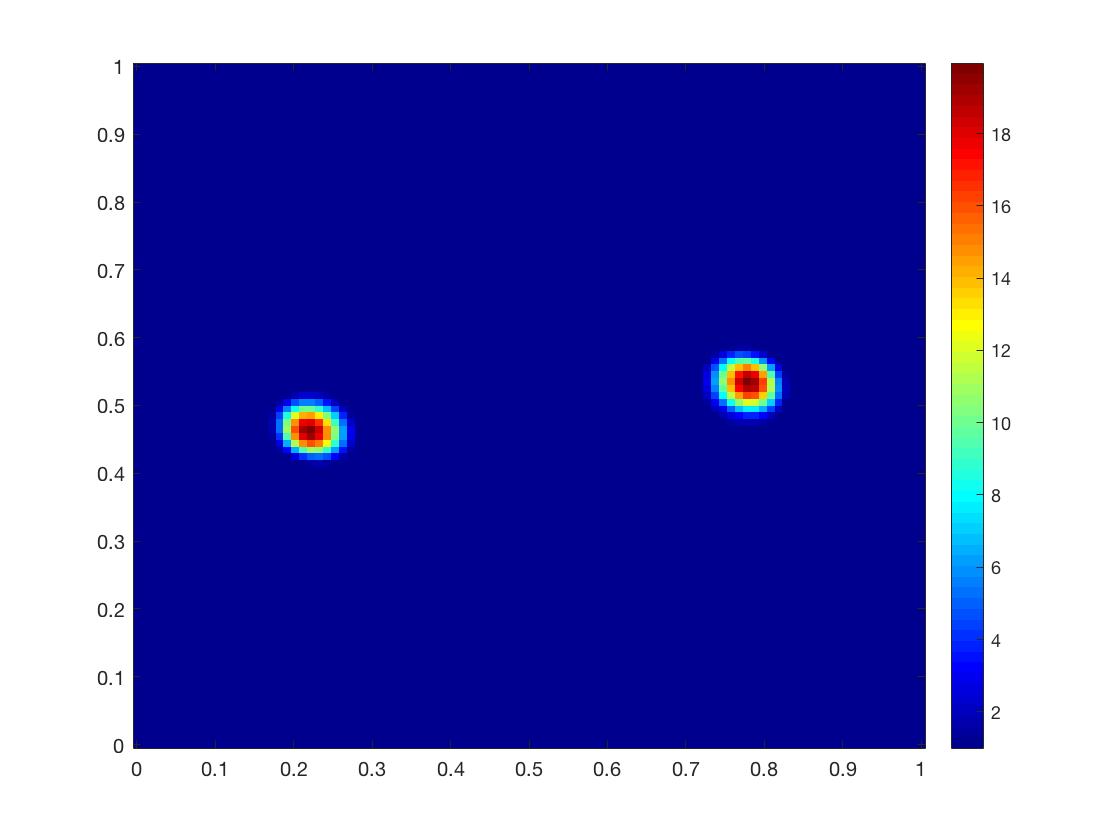}}
\caption{Numerical results for Example \ref{example3}: (a) true $\sigma$, (b) index $\Phi$ for $\sigma$ from DSM, (c) index $\Phi|_D$ for $\sigma$ (index function constrained to the chosen subdomain $D$), (d) least-squares reconstruction. (e), (f), (g), (h) are corresponding graphs for $\mu$.  These two rows are derived using data  with $20\%$ noise.}
\label{Ex127.main}
\end{figure}

%All examples above only consist of inclusions of trivial shape. Now we consider a ring-shaped inclusion.
\begin{exam}\label{example4}
 In this example we reconstruct diffusion coefficient $\sigma$ with ring-shaped square inclusion as shown in Fig.~\ref{Ex7sig.sub.1} with two sets of measurement. The outer and inner side length of the ring are $0.2$ and $0.15$, 
 and the rectangle ring is centered at $(0.5,0.6)$. The coefficient $\sigma$ is taken to be $20$ inside the region.
\end{exam}
%It is the most challenging case to reconstruct. One single set of measurement is insufficient to completely retrieve the ring-shape inhomogeneity, as only some parts of the ring can be observed from the tomography, depending on the pattern of imposed Neumann boundary condition. 
We use two sets of measurement from different directions for the two-stage least-squares method. For this challenging case, the index function from the DSM  only reflects an approximated location of inclusion as in Fig.~\ref{Ex7sig.sub.2}, and we have no clue about the shape of inclusion from only DSM reconstruction.  With two sets of measurement, the
least-squares formulation can reconstruct the edges of the ring-shape inclusion as shown in Fig.~\ref{Ex7sig.sub.4}. 
After adding $20\%$ noise, one has the reconstruction (see Fig.~\ref{Ex7sig.sub.8}) that is very similar to the result without noise, which shows the approximation is quite stable with respect to the noise.
\begin{figure}[ht!]
\centering
\subfigure[ true $\sigma$]{
\label{Ex7sig.sub.1}
\includegraphics[width=0.23\textwidth]{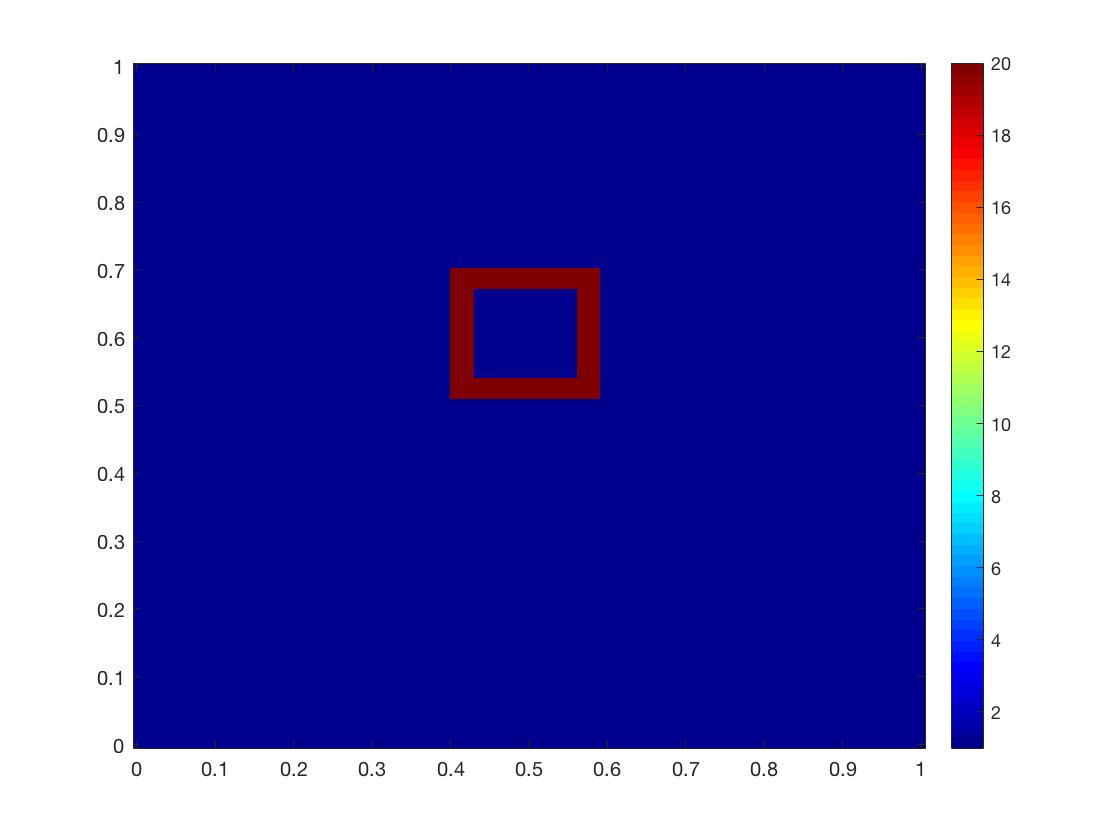}}
\subfigure[index $\Phi$]{
\label{Ex7sig.sub.2}
\includegraphics[width=0.23\textwidth]{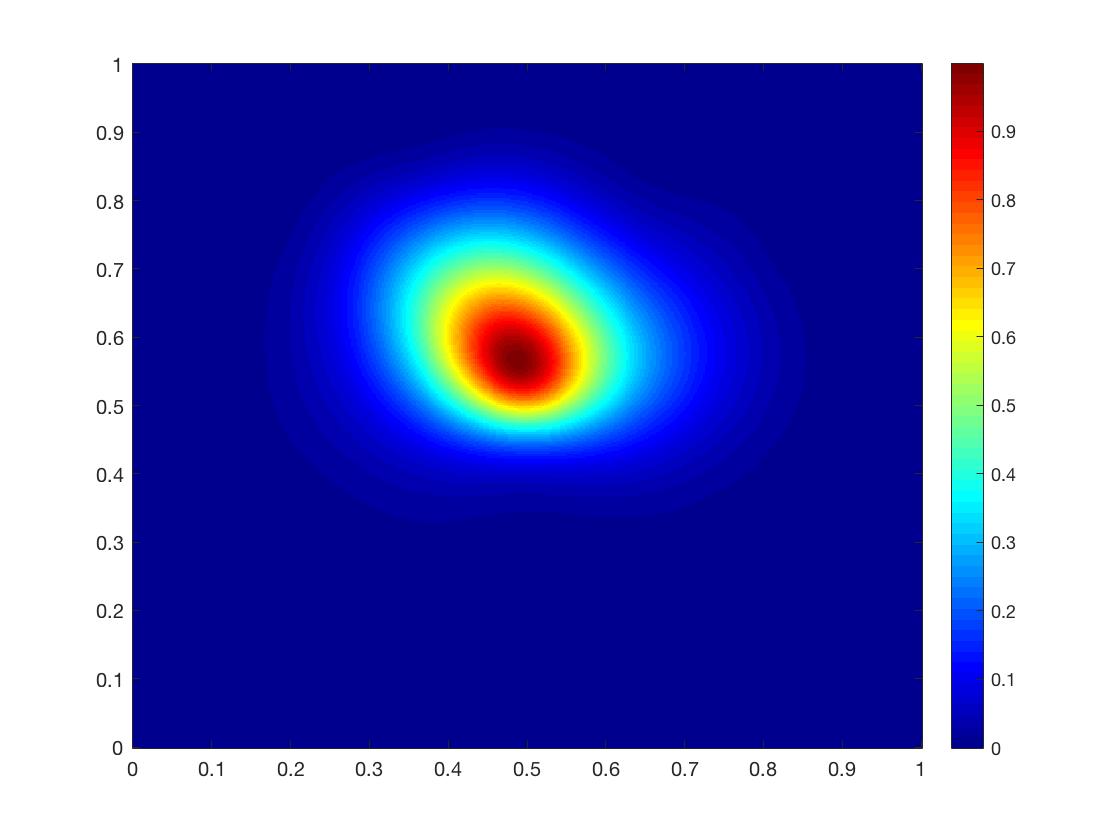}}
\subfigure[  index $\Phi|_D$]{
\label{Ex7sig.sub.3}
\includegraphics[width=0.23\textwidth]{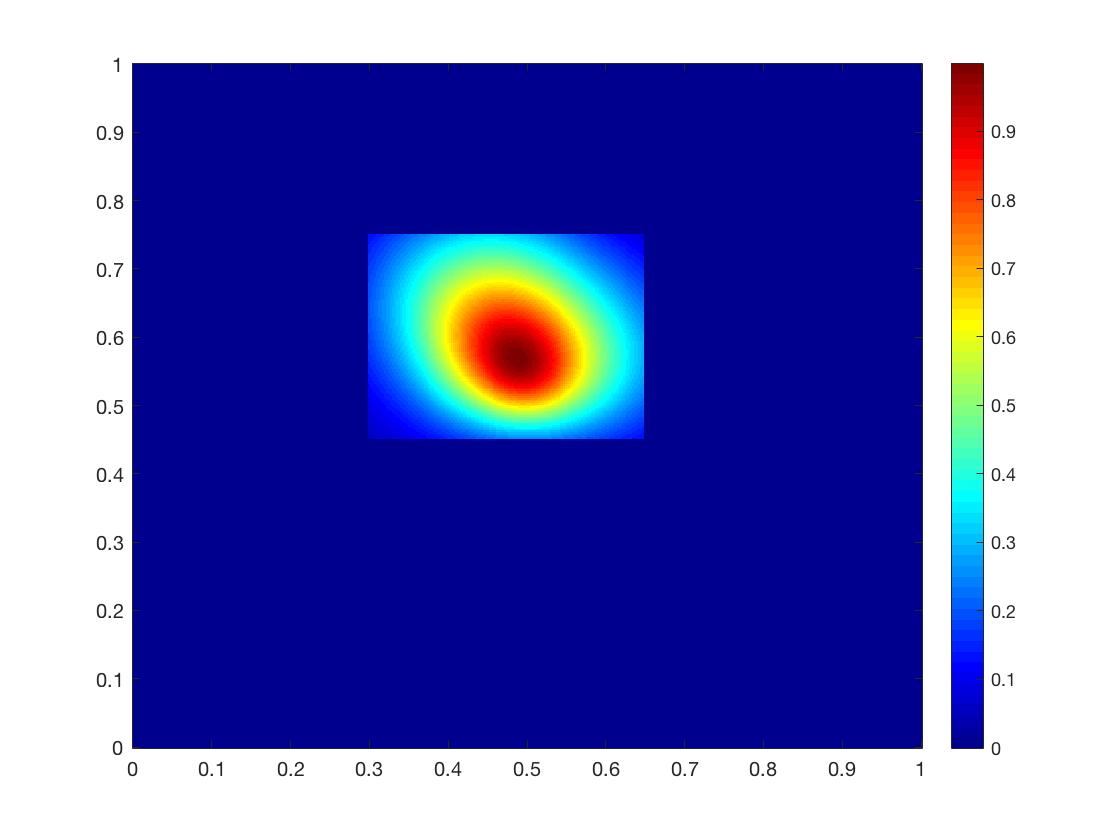}}
\subfigure[ least-squares]{
\label{Ex7sig.sub.4}
\includegraphics[width=0.23\textwidth]{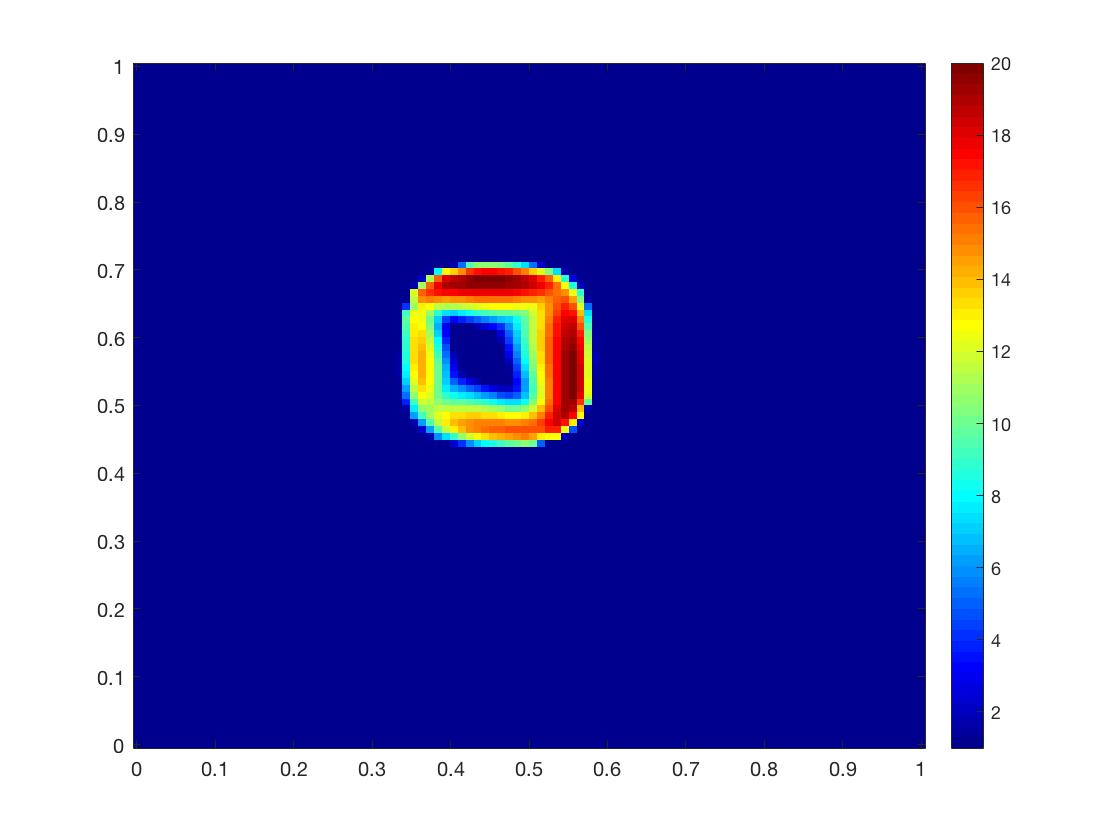}}
\subfigure[ true $\sigma$]{
\label{Ex7sig.sub.5}
\includegraphics[width=0.23\textwidth]{figure/ringtrue.jpg}}
\subfigure[ index $\Phi$]{
\label{Ex7sig.sub.6}
\includegraphics[width=0.23\textwidth]{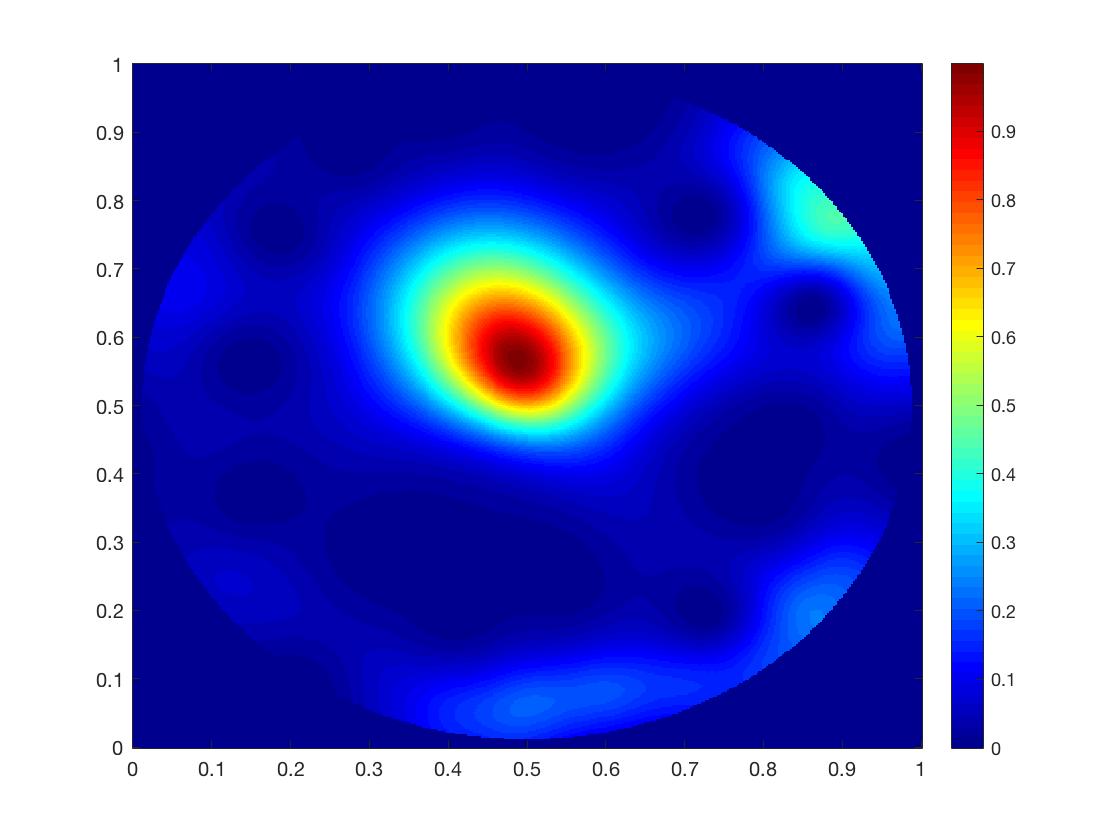}}
\subfigure[ index $\Phi|_D$]{
\label{Ex7sig.sub.7}
\includegraphics[width=0.23\textwidth]{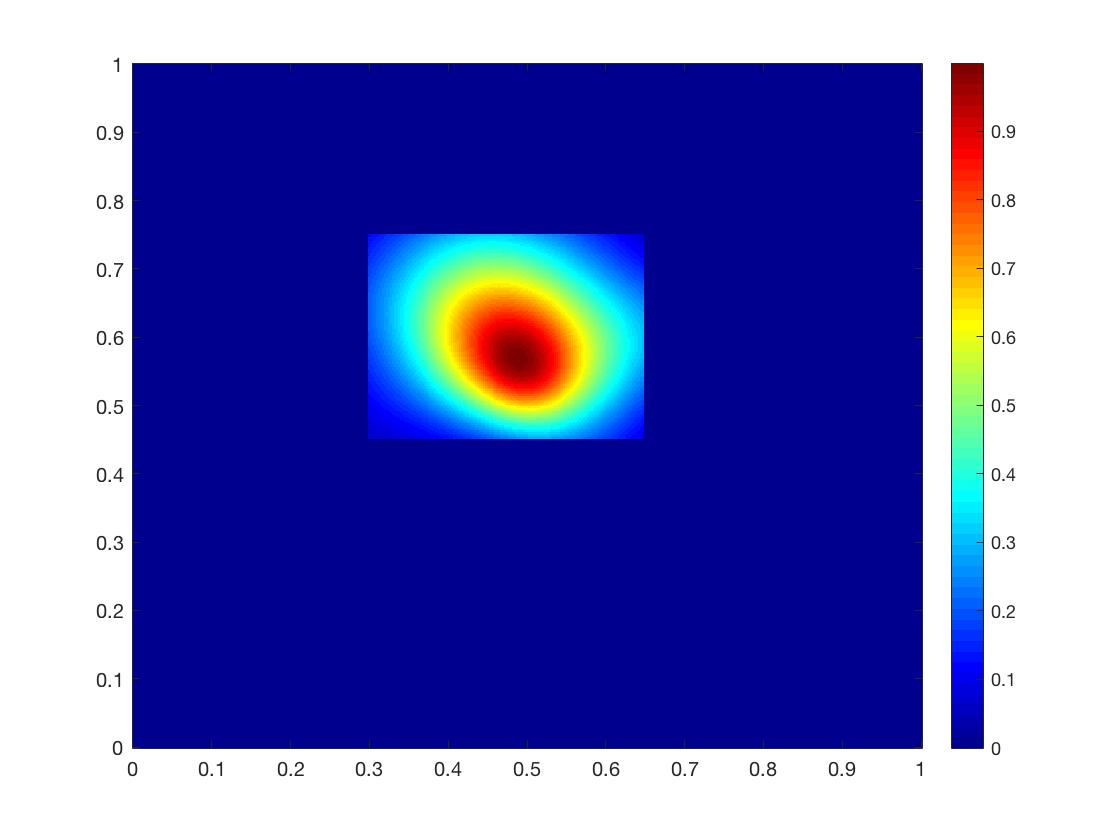}}
\subfigure[ least-squares]{
\label{Ex7sig.sub.8}
\includegraphics[width=0.23\textwidth]{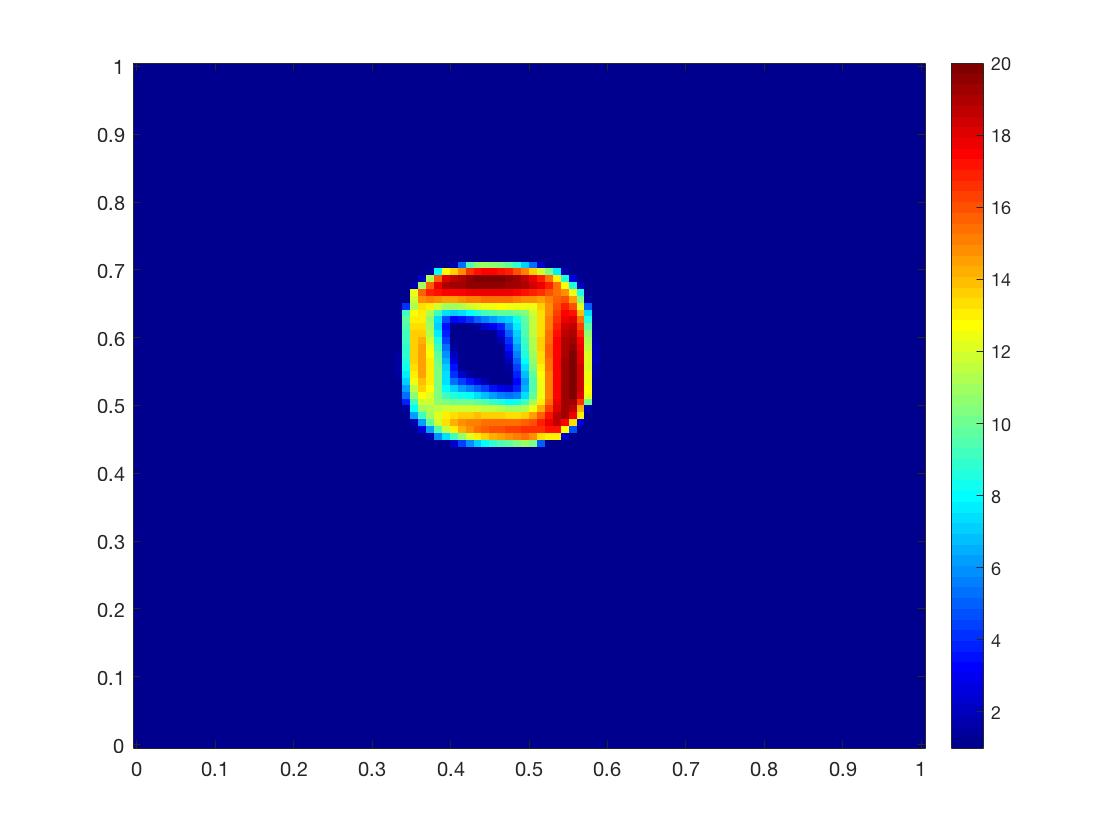}}
\caption{Numerical results for Example \ref{example4}: (a) true $\sigma$,  (b) index $\Phi$ from DSM, (c) index $\Phi|_D$ (index function constrained to the chosen subdomain $D$), (d) least-squares reconstruction.  (e), (f), (g), (h) are corresponding graphs using data with $20\%$ noise. }
\label{Ex126.main}
\end{figure}

\end{subsection}
\end{section}

 \bibliographystyle{unsrt}
% \bibliography{lsq_dot}
 \end{document}